\documentclass[11pt,reqno]{amsart}
\usepackage{amsmath, amsfonts, amsthm, amssymb,amscd, graphicx, amscd}
\usepackage{float,epsf}
\usepackage[english]{babel}
\usepackage{enumerate}
\usepackage{tikz}
\usepackage{mathrsfs}
\usepackage[numbers,sort&compress]{natbib}

\usepackage{srcltx}
\usepackage{geometry}
\usepackage{verbatim}
\usepackage{mathrsfs}
\usepackage{hyperref}
\usepackage{enumitem} 
\usepackage{bbm} 
\usepackage{stmaryrd}

\usepackage{relsize}
\usepackage{exscale}

\usepackage{mathtools}

\newtheorem{thm}{Theorem}[section]
\newtheorem{defn}[thm]{Definition}

\newtheorem{lem}[thm]{Lemma}

\newtheorem{rem}[thm]{Remark}

\numberwithin{equation}{section}

\def\dd{{\rm d}}
\hypersetup{bookmarksdepth=2}

\begin{document}
\title[Quantum Quasi-neutral Limits]{Quantum Quasi-neutral Limits \\and Isothermal Euler Equations}

\author{Immanuel Ben-Porat}
\address{Immanuel Ben-Porat, Mathematical Universit\"at Basel
 Spiegelgasse 1
 CH-4051 Basel, Switzerland.}
\email{immanuel.ben-porath@unibas.ch}

\author{Gui-Qiang G. Chen}
\address{Gui-Qiang G. Chen, Mathematical Institute, University of Oxford, Oxford OX2 6GG, UK.}
\email{chengq@maths.ox.ac.uk}

\author{Difan Yuan}
\address{Difan Yuan,
School of Mathematical Sciences and Laboratory of
Mathematics and Complex Systems, Beijing Normal University; Mathematical Institute, University of Oxford, Oxford OX2 6GG,
UK}
 \email{yuandf@amss.ac.cn}


\begin{abstract}
We provide a rigorous justification of the semiclassical quasi-neutral and the quantum many-body limits to the isothermal Euler equations. We consider the nonlinear Schr\"{o}dinger-Poisson-Boltzmann system under a quasi-neutral scaling and establish the convergence of its solutions to the isothermal Euler equations.  
Different from the previous results that dealt with the linear Poisson equations, 
the system under our consideration accounts for the exponential nonlinearity in the potential. 
A modulated energy method is adopted, 
allowing us to derive the stability estimates and asymptotics. 
Furthermore, we focus our analysis on the many-body quantum problem via the von Neumann equation 
and establish a mean-field limit in one dimension by using Serfaty's functional 
inequalities, and thus connecting the quantum many-body dynamics with the macroscopic hydrodynamic equations. 
A refined analysis of the quasi-neutral scaling for the massless systems is presented, 
and the well-posedness of the underlying quantum dynamics is established. 
Moreover, the construction of general admissible initial data is obtained. 
Our results provide a rigorous mathematical analysis for the derivation of quantum hydrodynamic models 
and their limits, contributing to the broader understanding of interactions between 
quantum mechanics and compressible fluid dynamics.
\end{abstract}
\maketitle

\setcounter{tocdepth}{1}
\thispagestyle{empty}

\section{Introduction }
We are concerned with the rigorous justification of the semi-classical and
quantum many-body quasi-neutral limits to the macroscopic systems of nonlinear partial differential equations. Consider the following $d$-dimensional ($d$-D)
\textit{Schr\"odinger-Poisson-Boltzmann system with quasi-neutral scaling}:
\begin{align}
\begin{cases}
i\hbar\partial_{t}\psi_{\varepsilon,\hbar}=-\frac{\hbar^{2}}{2}\Delta \psi_{\varepsilon,\hbar} +V_{\varepsilon,\hbar}\psi_{\varepsilon,\hbar}, \qquad & x\in \mathbb{T}^d,\,t>0,\\[1mm]
-\varepsilon \Delta V_{\varepsilon,\hbar}=\left\vert \psi_{\varepsilon,\hbar} \right\vert^{2}-e^{V_{\varepsilon,\hbar}},& x\in \mathbb{T}^d,\,t>0,\\[1mm]
\psi_{\varepsilon,\hbar}|_{t=0}=\psi^{\mathrm{in}}_{\varepsilon,\hbar},& x\in \mathbb{T}^d.
\end{cases}\label{eq:Hartree VPME}
\end{align}
The unknown $\psi_{\varepsilon,\hbar}(t,\cdot)$ is a time-dependent element of the Hilbert space $\mathfrak{H\coloneqq}L^{2}(\mathbb{T}^{d})$
such that $\int_{\mathbb{T}^{d}}\vert \psi_{\varepsilon,\hbar}(t,x) \vert^{2}\,\dd x=1$. The symbol $\hbar$ represents the Planck constant and should be regarded as a small parameter. The 1-body quantum Hamiltonian $H_{\varepsilon,\hbar}$ is the differential
operator given by
\[
H_{\varepsilon,\hbar}\coloneqq-\frac{\hbar^{2}}{2}\Delta+V_{\varepsilon,\hbar},
\]
so that  \eqref{eq:Hartree VPME} can be written as
\begin{align*}
\begin{cases}
i\hbar\partial_{t}\psi_{\varepsilon,\hbar}=H_{\varepsilon,\hbar}\psi_{\varepsilon,\hbar},& x\in \mathbb{T}^d,\,t>0,\\[1mm]
-\varepsilon \Delta V_{\varepsilon,\hbar}=\left\vert \psi_{\varepsilon,\hbar} \right\vert^{2}-m_{\varepsilon,\hbar},\qquad& x\in \mathbb{T}^d,\,t>0,\\[1mm]
\psi_{\varepsilon,\hbar}|_{t=0}=\psi_{\varepsilon,\hbar}^{\mathrm{in}},& x\in \mathbb{T}^d,
\end{cases}
\end{align*}
where $m_{\varepsilon,\hbar}\coloneqq e^{V_{\varepsilon,\hbar}}$. 

Notice that this system is a variant of the much more well studied Schr\"odinger-Poisson system:
\begin{align}\label{SP eq}
\begin{cases}
i\hbar\partial_{t}\psi_{\varepsilon,\hbar}=-\frac{\hbar^{2}}{2}\Delta \psi_{\varepsilon,\hbar} +V_{\varepsilon,\hbar}\psi_{\varepsilon,\hbar},\qquad & x\in \mathbb{T}^d,\,t>0,\\[1mm]
-\varepsilon \Delta V_{\varepsilon,\hbar}=\left\vert \psi_{\varepsilon,\hbar} \right\vert^{2}-1,& x\in \mathbb{T}^d,\,t>0,\\[1mm]
\psi_{\varepsilon,\hbar}|_{t=0}=\psi^{\mathrm{in}}_{\varepsilon,\hbar},& x\in \mathbb{T}^d.
\end{cases}
\end{align}

The fundamental difference between \eqref{eq:Hartree VPME}  and \eqref{SP eq} is encapsulated 
in the exponential nonlinearity in system \eqref{eq:Hartree VPME}, {\it i.e.}, from the fact that the dynamics are coupled with the so called \textit{Poisson-Boltzman} equation, instead of the \textit{linear Poisson} equation. This nonlinearity is the source of several mathematical obstructions: for instance, while the well-posedness theory of the linear Poisson equation:
$$
-\Delta V=h-1
$$ 
is classical even for $h$ a probability measure, a 
well-posedness theory has been developed only relatively recently for the Poisson-Boltzman 
equation (see \cite{CKS2025,griffin2021global}): 
$$
-\Delta V=h-e^{V},
$$
where $h\in L^{p}$ for arbitrary $p>1$ and dimension $d \geq 1$.  
Thus, the classical well-posedness results for the Schr\"odinger-Poisson system (established for instance in \cite{bove1974existence,ginibre1980class}) are still inadequate for the present settings. 
Therefore, one of our purposes in this paper is to prove the well-posedness of system \eqref{eq:Hartree VPME}. 
For this purpose, we rely heavily on the elliptic PDE theory for the Poisson-Boltzmann equation
developed in \cite{griffin2021global}, where $h$ is a given bounded function. 
The well-posedness of system \eqref{eq:Hartree VPME} is summarized in the following theorem:

\begin{thm}\label{second main thm intro}
Assume $d \in \left\{2,3\right\}$. 
Assume
\begin{align}\label{initialdata}
\psi^{\mathrm{in}}_{\varepsilon,\hbar}\in C^{\infty}(\mathbb{T}^{d}) \qquad 
\mbox{with $\,\,\,\int_{\mathbb{\mathbb{T}}^{d}}\big|\psi^{\mathrm{in}}_{\varepsilon,\hbar}(x)\big|^{2}\,\dd x=1$}.
\end{align}
Then, for each fixed $\varepsilon>0$ and $\hbar>0$, there exists a unique solution $\psi_{\varepsilon,\hbar} \in \mathrm{Lip}([0,T];H^{2}(\mathbb{T}^{d}))$ of the Cauchy problem{\rm :}
\begin{align}\label{well posedness of hartree vpme}
\begin{cases}
i\hbar\partial_{t}\psi_{\varepsilon,\hbar}(t,x)
=-\frac{1}{2}\Delta\psi_{\varepsilon,\hbar}(t,x)+V_{\varepsilon,\hbar}(t,x)\psi_{\varepsilon,\hbar}(t,x),\qquad& x\in \mathbb{T}^d,\,t>0,\\[1mm]
-\varepsilon\Delta V_{\varepsilon,\hbar}(t,x)=\left|\psi_{\varepsilon,\hbar}(t,x)\right|^{2}-e^{V_{\varepsilon,\hbar}(t,x)},& x\in \mathbb{T}^d,\,t>0,\\[1mm]
\psi_{\varepsilon,\hbar}|_{t=0}=\psi^{\mathrm{in}}_{\varepsilon,\hbar}, & x\in \mathbb{T}^d.
\end{cases}
\end{align}
\label{wellposed intro}
\end{thm}
As the main focus of this paper is on the semi-classical and quantum many-body limits, 
the proof of Theorem \ref{wellposed intro} is postponed to \S\ref{Well-posedness theory}. 
Our first main result is the rigorous derivation of the isothermal Euler equations: 
\begin{align}\label{Isothermal Euler}
\begin{cases}
\partial_{t}\rho+\mathrm{div}\,(\rho u)=0,\\[1mm]
\partial_{t}(\rho u)+\mathrm{div}\,(\rho u\otimes u)+\nabla\rho=0,
\end{cases}
\end{align}
as a combined semi-classical quasi-neutral limit. 
Before we state this main result, we remark that this problem has already been studied in the context of the usual Schr\"odinger-Poisson system \eqref{SP eq} by Puel \cite{puel2002convergence} for pure states 
and later by Rosenzweig \cite{rosenzweig2021quantum} for general states where the density operator formulation of the Schr\"odinger-Poisson system, also known as the Hartree equations, is considered. 
Motivated by Rosenzweig's approach, we make use of the modulated energy method. 
However, the modulated energy has to be modified in comparison to \cite{rosenzweig2021quantum}: 
The modulated energy used in \cite{rosenzweig2021quantum} is the quantum analogue of the quantity introduced in \cite{brenier2000convergence}, while the modulated energy studied in this paper is the quantum analogue
of the quantity introduced in \cite{han2011quasineutral}. J\"{u}ngel-Wang \cite{Jungel} obtained 
combined semi-classical and quasi-neutral limits in the bipolar defocusing nonlinear Schr\"odinger-Poisson system 
in the whole space. The electron and current density defined by the solution of the Schr\"odinger-Poisson system 
converge to the solution of the compressible Euler equations with nonlinear pressure. We also refer to Golse and Paul \cite{golse2022} for the validity of the joint mean-field and classical limit of the quantum N-body dynamics leading to the pressureless Euler-Poisson system for factorized initial data whose first marginal has a monokinetic Wigner measure. Given a solution $\psi_{\varepsilon,\hbar}$ 
to system \eqref{eq:Hartree VPME}, consider the time-dependent trace-class 
operator  $R_{\varepsilon,\hbar}(t)\coloneqq \left|\psi_{\varepsilon,\hbar}(t)\left\rangle \right\langle \psi_{\varepsilon,\hbar}(t)\right|$ (here the bra-ket $\left|\psi\left\rangle \right\langle \psi\right|$ 
is the rank-1 projection on $\mathfrak{H}$ defined by $\varphi \mapsto \left\langle\psi,\varphi\right\rangle\psi$ in the conventional notation). 
The crux of the argument reduces to obtaining an evolution estimate on the following quantity, called \textit{the modulated energy}:
\begin{align}
\mathcal{\mathcal{E}}_{\varepsilon,\hbar}(t)\coloneqq&\,\frac{1}{2}\sum^{d}_{j=1}\mathrm{tr}\big((i\hbar\partial_{x_{j}}+u^{j})^{2}R_{\varepsilon,\hbar}(t)\big) +\frac{\varepsilon}{2}\int_{\mathbb{T}^{d}}\left|\nabla V_{\varepsilon,\hbar}(t,x)\right|^{2}\,\dd x  \notag\\&+\int_{\mathbb{T}^{d}}\Big(m_{\varepsilon,\hbar}(t,x)\log\big(\frac{m_{\varepsilon,\hbar}(t,x)}{\rho(t,x)}\big)-m_{\varepsilon,\hbar}(t,x)+\rho(t,x)\Big)\,\dd x.
\label{Modulated energy intro}
\end{align}
The \textit{total energy}, denoted by $\mathcal{F}_{\varepsilon,\hbar}(t)$ and defined explicitly in \eqref{total} below, 
is a conserved quantity obtained from $\mathcal{E}_{\varepsilon,\hbar}(t)$ by plugging in $\rho=1$ and $u=0$.  
The quantum current 
is defined by
\begin{align*}
J_{\varepsilon,\hbar}(t,x)
\coloneqq \hbar\,\mathrm{Im}\big(\overline{\psi}_{\varepsilon,\hbar}(t,x)\nabla \psi_{\varepsilon,\hbar}(t,x)\big).
\end{align*}
Our first main result is described in the following theorem:

\begin{thm}
Let $(\rho,u)\in C^{1}([0,T);H^{s-1}(\mathbb{T}^{d}))\times\left(  C([0,T);H^{s}(\mathbb{T}^{d}))\cap C^{1}([0,T)\times \mathbb{T}^{d})\right)$ be the solution of \eqref{Isothermal Euler} with initial data $(\rho_0,u_0)\in \big(H^{s}(\mathbb{T}^{d})\cap\mathcal{P}(\mathbb{T}^{d})\big)\times H^{s}(\mathbb{T}^{d})$ and $\rho_{0}>0$ for some $s>\frac{d}{2}+1$ 
and $T>0$.  Let $\psi_{\varepsilon,\hbar}(t,x)$ be the unique solution in  {\rm Theorem \ref{wellposed intro}} with the initial data $\psi^{\mathrm{in}}_{\varepsilon,\hbar}(x)$ satisfying \eqref{initialdata}.  
Assume that there is some constant $\mathcal{F}_{0}$ 
such that $\mathcal{F}_{\varepsilon,\hbar}(0)\leq \mathcal{F}_{0}$ uniformly in $(\varepsilon,\hbar)$. 
Let $\mathcal{E}_{\varepsilon,\hbar}(t)$ be given by \eqref{Modulated energy intro}. Then
\begin{align*}
    \mathcal{E}_{\varepsilon,\hbar}(t)\leq e^{Ct}\left(\mathcal{E}_{\varepsilon,\hbar}(0)+\sqrt{\varepsilon}\right) 
    \qquad \text{ for all $t\in [0,T]$,}
\end{align*}
where
 $C=C(\Vert \nabla u\Vert _{L^\infty_{t,x}},\Vert \log(\rho)\Vert _{W_{t}^{1,\infty}H_{x}^{1}},
 \Vert \nabla(u\cdot\nabla\log(\rho))\Vert _{L^{\infty}_{t}L^{2}_{x}},T,\mathcal{F}_{0}).$
 Consequently,
 \begin{align*}
  \underset{t \in [0,T]}{\sup} \mathcal{E}_{\varepsilon,\hbar}(t)\underset{\varepsilon+\hbar\rightarrow0}{\longrightarrow}  0
 \end{align*}
provided $ \mathcal{E}_{\varepsilon,\hbar}(0)\underset{\varepsilon+\hbar\rightarrow0}{\longrightarrow}0$, and
\[
(\rho_{\varepsilon,\hbar},J_{\varepsilon,\hbar})(t,\cdot)
\rightharpoonup (\rho, \rho u)(t,\cdot)
\]
for the narrow topology of signed Radon measures on $\mathbb{T}^{d}$, as $\varepsilon+\hbar\rightarrow0$ locally uniformly in time. 
\label{1body main intro}
\end{thm}
 
The initial data $\psi^{\mathrm{in}}_{\varepsilon,\hbar}$ can be constructed such that $\mathcal{E}_{\varepsilon,\hbar}(0)\underset{\varepsilon+\hbar\rightarrow0}{\longrightarrow}0$. We will elaborate on this in \S\ref{well prepared sec}. \label{rem abot inital data}

The stability estimate and the implied weak convergence are valid for arbitrarily large-time intervals. 
However, the existence and uniqueness of a smooth solution to the isothermal Euler equations are known 
to be true only for a short-time interval.  \label{rem about short time}

\smallskip
Our second main result concerns the $N$-body problem. Given a configuration $X_{N}=(x_{1},\cdots,x_{N})\in\mathbb{T}^{N},$ denote by $\mu_{X_{N}}$ the empirical measure centered at $X_{N}$, {\it i.e.,}
\begin{align*}
\mu_{X_{N}}\coloneqq \frac{1}{N}\sum^{N}_{j=1} \delta_{x_{j}}.
\end{align*}
In variance with the quasi-neutral quantum many-body limit established in \cite{rosenzweig2021quantum}, 
our main result in this context is limited to the 1-D settings. 
Let $V_{\varepsilon,X_{N}}$ be the solution to the equation:
\[
-\varepsilon\Delta V_{\varepsilon,X_{N}}=\mu_{X_{N}}-m_{\varepsilon,X_{N}}
\qquad \mbox{with $\, m_{\varepsilon,X_{N}}=e^{V_{\varepsilon,X_{N}}}$}.
\]
The $N$-body quantum Hamiltonian is the differential operator acting
on the Hilbert space $\mathfrak{H}^{\otimes N}=L^{2}(\mathbb{T}^{N})$ given
by
\begin{align*}
\mathscr{H}_{\varepsilon,\hbar,N}(X_{N})&\coloneqq-\frac{\hbar^{2}}{2}\sum^N_{k=1}\Delta_{x_{k}}+N\int_{\mathbb{T}}V_{\varepsilon,X_{N}}(x)m_{\varepsilon,X_{N}}(x)\,\dd x\nonumber\\
&\quad\,\,\,+\frac{N}{2\varepsilon}\int_{\mathbb{T}\times\mathbb{T}}K(x-y)(\mu_{X_{N}}-m_{\varepsilon,X_{N}})^{\otimes2}(\dd x\dd y)\\
&\coloneqq \mathscr{K}_{N,\hbar}+\mathscr{I}_{\varepsilon,X_{N}}+\mathscr{V}_{\varepsilon,X_{N}}.
\end{align*}
Here, $K$ designates the Green function of the Laplacian. In 1-D, the Green function admits the explicit formula $K(x)=\frac{x^2-\left\vert x \right\vert}{2}$. Note that the terms
$\mathscr{I}_{\varepsilon,X_{N}}$ and $\mathscr{V}_{\varepsilon,X_{N}}$ are viewed as a multiplication operator
in the variable $X_{N}$. The Cauchy problem for the $N$-body von Neumann equation is given by
\begin{equation}\label{eq:von N intro}
 \begin{cases}
i\hbar\partial_{t}R_{\varepsilon,\hbar,N}(t)=\left[\mathscr{H}_{\varepsilon,\hbar,N},R_{\varepsilon,\hbar,N}(t)\right],\\[1mm]
R_{\varepsilon,\hbar,N}|_{t=0}=R^\mathrm{in}_{\varepsilon,\hbar,N}.
\end{cases}
\end{equation}
The unknown $R_{\varepsilon,\hbar,N}$ is a time-dependent symmetric density operator, 
and $[\,\cdot,\cdot\,]$ designates the commutator defined by $[A,B]\coloneqq AB-BA$ 
for any given operators $A$ and $B$. By a \textit{density operator}, 
we mean a bounded operator $R$ such that $R$ is self-adjoint non-negative ($R=R^{*}\geq 0$) 
and $\mathrm{tr}(R)=1$. By a \textit{symmetric operator}  on $\mathfrak{H}^{\otimes N}$ 
we mean an operator $R_{N}$ such that, for all permutations $\sigma \in \mathfrak{S}_{N}$ 
($\mathfrak{S}_{N}$ is the symmetric group on $N$ elements), 
\begin{align*}
U_{\sigma}R_{N}U_{\sigma}^{\ast}=R_{N},
\end{align*}
where  $U_{\sigma}$ is the operator on $\mathfrak{H}^{\otimes N}$ defined by
\begin{align*}
(U_{\sigma}\Psi_{N})(x_{1},\cdots,x_{N})\coloneqq\Psi_{N}(x_{\sigma^{-1}(1)}, \cdots, x_{\sigma^{-1}(N)}) 
\qquad \text{ for any } \Psi_{N}\in \mathfrak{H}^{\otimes N}.
\end{align*}
We denote the class of density operators on $\mathfrak{H}$ by $\mathcal{D}(\mathfrak{H})$ and the class of symmetric density operators on $\mathfrak{H}^{\otimes N}$ by $\mathcal{D}_{s}(\mathfrak{H}^{\otimes N})$. 
If $R_{N}\in \mathcal{D}_{s}(\mathfrak{H}^{\otimes N})$ with kernel $k(X_{N},Y_{N})$, 
then  $\rho_{N}\in L^{1}(\mathbb{T}^{N})$ is denoted as the function 
defined by $\rho_{N}(X_{N})\coloneqq k(X_{N},X_{N})$. 
It is well known that $\rho_{N}$ is a symmetric probability density on $\mathbb{T}^{N}$
(see the footnote pp. 61--62 in \cite{golse2017schrodinger} for more details). 
We call $\rho_{N}$ the \textit{density} associated with $R_{N}$.
Hereafter, we denote by $\rho_{\varepsilon,\hbar,N}(t)$ the density associated with the operator $R_{\varepsilon,\hbar,N}(t)$.
With these definition, we can define the following  modulated energy:
\begin{align*}
&\mathcal{\mathcal{E}}_{\varepsilon,\hbar,N}(t)\\
&\coloneqq\,\frac{1}{2N}\sum^{N}_{j=1}\mathrm{tr}\big((i\hbar\partial_{x_{j}}+u(t,x_{j}))^{2}R_{\varepsilon,\hbar,N}(t)\big)\\
&\quad\,\,+\frac{1}{2\varepsilon}\int_{\mathbb{T}^{N}}\int_{\mathbb{T}\times\mathbb{T}}K(x-y)(\mu_{X_{N}}-m_{\varepsilon,X_{N}})^{\otimes2}(\dd x\dd y)\,\rho_{\varepsilon,\hbar,N}(t,X_{N})\,\dd X_{N}\\
&\quad\,\,+\int_{\mathbb{T}^{N}}\int_{\mathbb{T}}\Big(m_{\varepsilon,X_{N}}(t,x)\log\big(\frac{m_{\varepsilon,X_{N}}(t,x)}{\rho(t,x)}\big)
-m_{\varepsilon,X_{N}}(t,x)+\rho(t,x)\Big)\,\dd x\,\rho_{\varepsilon,\hbar,N}(t,X_{N})\,\dd X_{N}.
\end{align*}
The total energy of the system is denoted by $\mathcal{F}_{\varepsilon,\hbar,N}(t)$ (see \eqref{totalN} below).
We can also define an $N$-body analogue of the current that is denoted 
by $J_{\varepsilon,\hbar,N}$ (see Definition \ref{def of current} below). 
Let $R_{\varepsilon,\hbar,N:1}(t)$ be its first marginal density (see Definition \ref{def of marginal} below for the definition of the marginal). 
Denote the density function and the current of $R_{\varepsilon,\hbar,N:1}(t)$ by $\rho_{\varepsilon,\hbar,N:1}(t,\cdot)$ and $J_{\varepsilon,\hbar,N:1}(t,\cdot)$, respectively.

Our second main result is summarized below:

\begin{thm}
Let $(\rho,u)\in C^{1}([0,T);H^{s-1}(\mathbb{T}))\times\left(  C([0,T);H^{s}(\mathbb{T}))\cap C^{1}([0,T)\times \mathbb{T})\right)$ be the solution of \eqref{Isothermal Euler} with initial data $(\rho_0,u_0)\in \big(H^{s}(\mathbb{T})\cap\mathcal{P}(\mathbb{T})\big)\times H^{s}(\mathbb{T})$ and $\rho_{0}>0$ for some $s>\frac{3}{2}$ and $T>0$. Let $R_{\varepsilon,\hbar,N}^{\mathrm{in}}\in \mathcal{D}_{s}(\mathfrak{H}^{\otimes N})$ and
$\mathrm{tr}\big((-\Delta_{N})^{2}R_{\varepsilon,\hbar,N}^{\mathrm{in}})<\infty$.
Let $R_{\varepsilon,\hbar,N}(t)$ be the solution
to \eqref{eq:von N intro} $($ensured by {\rm Lemma \ref{Kato thm }}$)$. Assume that there is some constant $\mathcal{F}_{0}$ such that $\mathcal{F}_{\varepsilon,\hbar,N}(0)\leq \mathcal{F}_{0}$ uniformly in $(\varepsilon,\hbar,N)$.
Then
\[
\mathcal{E}_{\varepsilon,\hbar,N}(t)\leq e^{Ct}\Big(\mathcal{E}_{\varepsilon,\hbar,N}(0)+N^{-\lambda}+\sqrt{\varepsilon}+\frac{e^{\frac{1}{\varepsilon}}}{\varepsilon N^{2}}\Big) \qquad\,\,\mbox{ for all $t\in [0,T]$},
\]
where $C=C(T,\Vert u\Vert_{L^{\infty}W^{1,\infty}}, \Vert \log(\rho)\Vert _{W_{t,x}^{1,\infty}},\mathcal{F}_{0})$ 
and $\lambda>0$. 
Consequently,  
\begin{align*}
  \underset{t \in [0,T]}{\sup} \mathcal{E}_{\varepsilon,\hbar,N}(t)\underset{\varepsilon+\hbar+\frac{1}{N}\rightarrow0}{\longrightarrow}   0,
 \end{align*}
provided $ \mathcal{E}_{\varepsilon,\hbar,N}(0)\underset{\varepsilon+\hbar+\frac{1}{N}\rightarrow0}{\rightarrow}0$ and  $\varepsilon=\varepsilon(N)$ is such that $\frac{e^{\frac{1}{\varepsilon}}}{\varepsilon N^{2}}\underset{N\rightarrow\infty}{\rightarrow}0$,
and 
\[
(\rho_{\varepsilon,\hbar,N:1},\, J_{\varepsilon,\hbar,N:1})(t,\cdot)
\,\rightharpoonup \, (\rho,\, \rho u)(t,\cdot)
\]
for the narrow topology of signed Radon measures on $\mathbb{T}$, as $\varepsilon+\hbar+\frac{1}{N}\rightarrow0$ locally uniformly
in time.
\label{3rd main result intro}
\end{thm}
A curious feature of the proof of Theorem \ref{3rd main result intro} is that 
it exploits the commutator estimates of Serfaty \cite{duerinckx2020mean}. 
These commutator estimates also played a decisive role in the derivation of the incompressible Euler equations  
as a mean-field quasi-neutral limit in \cite{han2021} or as a quantum many-body quasi-neutral limit in \cite{rosenzweig2021quantum}.   
However, this approach, which is a reminiscent of the {\it renormalized energy} method, 
is absent from the literature on the mean-field quasi-neutral limit of massless Vlasov-Poisson-like equations; 
in this sense, it reflects a novelty of the method. 
The constraint that $\varepsilon(N)$ vanishes sufficiently slowly in $N$ is natural. 
Note, however, that in our setting, the vanishing rate of $\varepsilon$ is logarithmic in $N$, 
whereas it is polynomial in $N$ in the usual quasi-neutral limit (considered, for example, in \cite{rosenzweig2021quantum}). This slower decay rate emerges from the $L^{\infty}$ bound on $e^{V_{\varepsilon,X_N}}$ in the argument. \\

To put our results in an appropriate context, we close this introduction by providing 
a brief overview of the classical quasi-neutral limit. There has been an extensive study of the classical 
quasi-neutral limit from the Vlasov-Poisson equations with massless electrons (VPME), 
also known as the Vlasov-Poisson system for ions, that is a mathematical model from plasma physics, 
which can be described
through the following kinetic equations:
\begin{align}
\begin{cases}
\partial_{t}f+\xi\cdot\nabla_{x}f-\nabla V\cdot\nabla_{\xi}f=0,\\[1mm]
-\Delta V=\int_{\mathbb{R}^{d}}f(t,x,\xi)\,\dd \xi-e^{V},
\end{cases}\label{eq:VPME INTRO}
\end{align}
with $f(0,x,\xi)=f^{0}\geq0$ and 
$\int_{\mathbb{T}^{d}\times\mathbb{R}^{d}}f^{0}(x,\xi)\,\dd x\dd \xi=1$.
The unknown $f:[0,T]\times \mathbb{T}^{d}\times\mathbb{R}^{d}\rightarrow\mathbb{R}$
is a time-dependent probability density $f(t,\cdot,\cdot)\in\mathcal{P}(\mathbb{T}^{d}\times\mathbb{R}^{d})$.
This system is a variant of the well-known  Vlasov-Poisson (VP) system that reads
\begin{align}
\begin{cases}
\partial_{t}f+\xi\cdot\nabla_{x}f -\nabla V\cdot\nabla_{\xi}f=0,\\[1mm]
-\Delta V=\int_{\mathbb{R}^{d}}f(t,x,\xi)\,\dd \xi-1,
\end{cases}\label{eq:VP INTRO}
\end{align}
with $f(0,x,\xi)=f^{0}\geq0$ and 
$\int_{\mathbb{T}^{d}\times\mathbb{R}^{d}}f^{0}(x,\xi)\,\dd x\dd \xi=1$.

As before, the notable difference between \eqref{eq:VPME INTRO} and
\eqref{eq:VP INTRO} is reflected in the fact that the Poisson equation
coupled to the Vlasov equation is nonlinear in \eqref{eq:VPME INTRO}, where it is linear and explicitly 
solvable in \eqref{eq:VP INTRO}. 
The variance between \eqref{eq:VPME INTRO} and \eqref{eq:VP INTRO} is demonstrated
in a decisive manner already at the level of well-posedness: the well-posedness theory
on the entire space and the periodic case of the Vlasov-Poisson system
has been established in the 1990s in Lions-Perthame \cite{lions1991propagation} 
and Batt-Rein \cite{batt1991global}, respectively; 
also see \cite{griffin2021recent} for an overview of recent developments in the well-posedness theory of
the Vlasov equations of this type.
Only much later has the question of existence and uniqueness for the case of 
the VPME system \eqref{eq:VPME INTRO} has been settled: the existence of global weak solutions has been proved in \cite{han2014quasineutral} in the periodic 1-D case. 
Later, global well-posedness is proved in the torus in \cite{griffin2021global}.      
Endowed with this well-posedness, it is a natural question to consider hydrodynamical limits 
associated with this system, particularly the quasi-neutral limit.
As was done in the quantum regime considered before, 
we write the VP system and the VPME system with quasi-neutral scaling. 
The VPME system with quasi-neutral scaling reads
\begin{align}\label{eq:VPME WITH QUASINEUTRAL SCALING INTRO}
\begin{cases}
\partial_{t}f_{\varepsilon}+\xi\cdot\nabla_{x}f_{\varepsilon}-\nabla V_{\varepsilon}\cdot\nabla_{\xi}f_{\varepsilon}=0,\\[1mm]
-\varepsilon\Delta V_{\varepsilon}=\int_{\mathbb{R}^{d}}f_{\varepsilon}(t,x,\xi)\,\dd \xi-e^{V_{\varepsilon}},\\
\end{cases}
\end{align}
with $f_{\varepsilon}(0,x,\xi)=f_{\varepsilon}^{0}\geq0$ and 
$\int_{\mathbb{T}^{d}\times\mathbb{R}^{d}}f_{\varepsilon}^{0}(x,\xi)\,\dd x\dd \xi=1,$
while the VP system with quasi-neutral scaling reads
\begin{align}\label{eq:VP WITH QUASINEUTRAL INTRO}
\begin{cases}
\partial_{t}f_{\varepsilon}+\xi\cdot\nabla_{x}f_{\varepsilon}-\nabla V_{\varepsilon}\cdot\nabla_{\xi}f_{\varepsilon}=0,\\[1mm]
-\varepsilon\Delta V_{\varepsilon}=\int_{\mathbb{R}^{d}}f_{\varepsilon}(t,x,\xi)\,\dd \xi-1,\\
\end{cases}
\end{align}
with $f_{\varepsilon}(0,x,\xi)=f_{\varepsilon}^{0}\geq0$ 
and $\int_{\mathbb{T}^{d}\times\mathbb{R}^{d}}f_{\varepsilon}^{0}(x,\xi)\,\dd x\dd \xi=1.$

In the classical quasi-neutral limit, one seeks to study the convergence of $f_{\varepsilon}$ 
as $\varepsilon \rightarrow0$. 
Formal considerations suggest, as the limit system, the incompressible Euler equations in the 2-D VP case
and the isothermal Euler equations in the VPME case. 
More precisely, given a solution $f_{\varepsilon}(t,x,\xi)$ to system \eqref{eq:VPME WITH QUASINEUTRAL SCALING INTRO}  
or \eqref{eq:VP WITH QUASINEUTRAL INTRO}, consider the density
\begin{align}
    \rho_{\varepsilon}(t,x)\coloneqq \int_{\mathbb{R}^{d}} f_{\varepsilon}(t,x,\xi)\,\dd \xi
    \end{align}
and the current
\begin{align}
 J_{\varepsilon}(t,x):= \int_{\mathbb{R}^{d}} \xi f_{\varepsilon}(t,x,\xi)\,\dd \xi.
 \label{classical current}
\end{align}
Then, in the VPME case, one expects the limit system to be the isothermal Euler equations \eqref{Isothermal Euler}
in the sense that the convergences: $\rho_{\varepsilon}(t,\cdot)\rightarrow \rho(t,\cdot)$
and $J_{\varepsilon}(t,\cdot)\rightarrow (\rho u)(t,\cdot)$ is  propagated in time. 
In the VP case, the expected limit system is the incompressible Euler equations:
\begin{align}\label{Incompressible Euler}
\begin{cases}
\rho=1,\\
\partial_{t}u+u\cdot\nabla_{x}u+\nabla_{x}p=0,\\
\mathrm{div}\,u=0.
\end{cases}
\end{align}

Therefore, the quasi-neutral limit offers a derivation of fluid equations from kinetic equations. 
The rigorous justification of the quasi-neutral limit for VP has been initiated in
Brenier-Grenier  \cite{brenier1994limite} and Brenier \cite{brenier2000convergence}. See
also Golse-Saint-Raymond \cite{golse1999vlasov,golse2003vlasov}
for an approach involving compactness arguments. 
Since we are mostly concerned with the quasi-neutral limit for quantum analogues of VPME, 
we will not dwell further on this limit for VP.  Various works have been dedicated to the study of the derivation
of \eqref{Isothermal Euler} from \eqref{eq:VPME WITH QUASINEUTRAL SCALING INTRO}.
The quasi-neutral limit has been studied in Han-Kwan \cite{han2011quasineutral}
by using a modulated energy approach inspired by Brenier's method \cite{brenier2000convergence}. 
As already explained, this method will also play an important role in the present work. 
It should be mentioned that this approach can be extended to cover other variants
of the Vlasov system \eqref{eq:VPME INTRO}, including in the
presence of magnetic fields. 
The 1-D quasi-neutral limit, along with the combined 1-D mean-field quasi-neutral limit,
has been analyzed by Han-Kwan-Iacobelli in \cite{han2014quasineutral,han2021}.
Higher dimensions have been studied by Griffin-Pickering-Iacobelli
in \cite{griffin2020singular}, 
for mollified empirical measures, with a mollification
constant depending on the parameter $\varepsilon$. 

Theorem \ref{1body main intro} is the 1-body semi-classical analogue of the main result in \cite{han2011quasineutral}, 
while Theorem \ref{3rd main result intro} is the N-body quantum analogue of the main result 
in \cite{han2011quasineutral}. 
The latter limit, being a singular limit, requires the use of the functional inequalities in \cite{duerinckx2020mean}, 
and thus part of the novelty of the argument in the proof is reflected in the observation that
these functional inequalities can be incorporated within the
framework of quantum massless dynamics.     The fact that we are working in the quantum regime introduces technical difficulties at several different levels: the well-posedness of the underlying dispersive dynamics (addressed in Theorem \ref{wellposed intro}), the construction of admissible initial data  (addressed  in Section \ref{well prepared sec}) and, most importantly, the calculation of the quantum analogues of the modulated energy (addressed through Theorems \ref{1body main intro} and \ref{3rd main result intro}).\\

The paper is organized as follows. In Section 2, we calculate the time derivative of the modulated energy 
$\mathcal{E}_{\varepsilon,\hbar}$ and, as a result, obtain a Gr\"onwall estimate for this quantity. 
In Section 3, we renormalize this argument by using the functional inequalities discovered 
in \cite{duerinckx2020mean}, thus obtaining a stability estimate for $\mathcal{E}_{\varepsilon,\hbar,N}$. 
In Section 4, the general initial data $\psi^{\mathrm{in}}_{\varepsilon,\hbar}$ and $R^{\mathrm{in}}_{\varepsilon,\hbar,N}$ 
are constructed for which $\mathcal{E}_{\varepsilon,\hbar}(0)\underset{\varepsilon+\hbar\rightarrow 0}{\rightarrow} 0$ 
and $\mathcal{E}_{\varepsilon,\hbar,N}(0)\underset{\varepsilon+\hbar+\frac{1}{N}\rightarrow 0}{\rightarrow} 0$, respectively.
Finally, in Section 5, we prove the well-posedness of system \eqref{eq:Hartree VPME}, based on \cite{griffin2021global}. 
As will be explained, the well-posedness of system \eqref{eq:von N intro} 
is a direct consequence of the Kato-Rellich theorem.

\section{The 1-Body Semi-Classical Limit}\label{The 1-Body Semi-Classical Limit}
This section concerns the derivation of the isothermal Euler equations:
\begin{equation}
\begin{cases}
\partial_{t}\rho+\mathrm{div}\,(\rho u)=0,\\[1mm]
\partial_{t}(\rho u)+\mathrm{div}\,(\rho u \otimes u)+\nabla \rho=0,
\end{cases}
\label{Isothermal Euler Sec 2}
\end{equation}
as a limit of \eqref{eq:Hartree VPME} when $\varepsilon+\hbar\rightarrow0$.
Observe that an alternative formulation of \eqref{Isothermal Euler Sec 2}
is 
\begin{equation}
\begin{cases}
\partial_{t}\log\rho+\mathrm{div}\, u+u\cdot\nabla\log\rho=0,\\[1mm]
\partial_{t}u+u\cdot\nabla_{x}u+\nabla\log\rho=0.
\end{cases}
\label{isothermal log}
\end{equation}
The well-posedness theory of the isothermal Euler equations is summarized in the following theorem,
which can be found in \cite{Majda}.

\begin{thm}
Let $(\rho_{0},u_{0})\in (H^{s}(\mathbb{T}^{d}))^2$
for some $s>\frac{d}{2}+1$, 
with $\rho_{0}(x)>0$ for $x\in\mathbb{T}^d$. Then there exists some $T_{\ast}>0$ such that there exists a unique classical 
solution $(\rho,u)\in C^{1}([0,T_{\ast});H^{s-1}(\mathbb{T}^{d}))\times\left(  C([0,T_{\ast});H^{s}(\mathbb{T}^{d}))\cap C^{1}([0,T_{\ast})\times \mathbb{T}^{d})\right)$ of the Cauchy problem for system \eqref{isothermal log} with the initial data{\rm :}
\begin{equation}\label{InitialData}
(\rho,u)|_{t=0}=(\rho_{0}(x),u_{0}(x)).
\end{equation}
Moreover, $\log \rho\in  C([0,T_{\ast});H^{s}(\mathbb{T}^{d}))\cap C^{1}([0,T_{\ast})\times \mathbb{T}^{d})$. \label{existence and uniqueness for isothermal Euler}
\end{thm}

\begin{rem}
The well-posedness of the Cauchy problem \eqref{isothermal log}--\eqref{InitialData} is stated in {\rm \cite{han2011quasineutral}} on the whole space. However, the proof extends with no difficulties to the periodic settings considered here. See also {\rm \cite{Majda}}.
\end{rem}
We briefly review
the elliptic theory necessary for the study of the nonlinear elliptic
equation governing $V$, as developed in \cite{griffin2021global}.   First, we recap the existence and uniqueness theory and the Schauder-type estimates, as summarized in
the following:

\begin{lem}[\cite{griffin2021global}, Proposition 3.1]
\label{existence uniqueness for hat{U}}
Assume
that $d\in\{2,3\}$ and $h\in L^{\infty}(\mathbb{T}^{d})$. 
Then there exist both a unique $\tilde{V}\in H^{1}(\mathbb{T}^{d})$ with zero mean and $\hat{V}\in H^{1}(\mathbb{T}^{d})$ satisfying
\[
-\Delta\tilde{V}=
h-1, \quad -\Delta\hat{V}=1-e^{\hat{V}+\tilde{V}}.
\]
Furthermore, the following estimates hold{\rm :}
\begin{enumerate}
\item[\rm (i)] For all $\alpha \in (0,1)$,  
\begin{align*}
&\| \tilde{V}\|_{C^{1,\alpha}}\leq C_{\alpha,d}(1+\left\Vert h\right\Vert_{\infty}),\\
&\|\hat{V}\|_{C^{1,\alpha}}\leq C_{\alpha,d}\exp\big(C_{\alpha,d}(1+\| h\|_{\frac{d+2}{d}})\big),
\end{align*}
for some constant $C_{\alpha,d}>0.$
\smallskip
\item[\rm (ii)]  For all $\alpha \in (0,\frac{1}{5}]$ if $d=3$ and $\alpha\in(0,1)$ if $d=2$, 
\begin{align*}
\|\hat{V}\|_{C^{2,\alpha}}\leq C_{\alpha,d}\exp\exp\big(C_{\alpha,d}(1+\left\Vert h\right\Vert_{\frac{d+2}{d}})\big),
\end{align*}
for some constant $C_{\alpha,d}>0.$
\end{enumerate}
\end{lem}
\begin{lem}[\cite{griffin2023stability}, Proposition 3.2]
\label{existence uniqueness 1D}
Assume that $\rho \in \mathcal{P}(\mathbb{T})$ is a probability measure. Then there exist both a unique $\tilde{V}\in H^{1}(\mathbb{T})$ with zero mean and $\hat{V}\in H^{1}(\mathbb{T})$ satisfying
\begin{align*}
-\tilde{V}''=\rho-1, \quad -\hat{V}''=1-e^{\tilde{V}+\hat{V}}.
\end{align*}
Moreover, the following 
estimate holds{\rm :}
\begin{align*}
\| \hat{V}'\|_{\mathrm{Lip}}\leq 1.
\end{align*}
\end{lem}
\begin{rem}
It can be checked directly that, subject to the assumptions 
of {\rm Lemma \ref{existence uniqueness for hat{U}}} or {\rm \ref{existence uniqueness 1D}}, 
the function $e^{\tilde{V}+\hat{V}}$ is a probability density, \it{i.e.,} $\int_{\mathbb{T}^{d}}e^{\tilde{V}(x)+\hat{V}(x)}\,\dd x=1$.  \end{rem}

\begin{rem}
According to {\rm Lemma 2.1 in \cite{bouchut1991global}}, given $h\in L^{\infty}(\mathbb{T}^{d})$, 
there exists a unique solution $V\in H^1(\mathbb{T})$ to the problem{\rm :} $-\Delta V=h-e^{V}$. 
Therefore, if $h$ and $(\tilde{V},\hat{V})$ are as in {\rm Lemma \ref{existence uniqueness for hat{U}}}, we can represent the unique solution $V$ as $V=\tilde{V}+\hat{V}$. The same remark applies when $d=1$ 
and $\rho \in \mathcal{P}(\mathbb{T})$ is a probability measure. 
Hereafter, we denote $V\coloneqq \tilde{V}+\hat{V}$, where $\tilde{V}$ and $\hat{V}$ are the solutions 
guaranteed by {\rm Lemmas \ref{existence uniqueness for hat{U}}} and {\rm \ref{existence uniqueness 1D}}, respectively.
\end{rem}
\begin{rem}
If we further assume that $h\in C^{\infty}(\mathbb{T}^{d})$ in {\rm Lemma \ref{existence uniqueness for hat{U}}},
by the Schauder estimates and by a standard iteration argument, 
then we see that 
in fact 
$V\in C^{\infty}(\mathbb{T}^{d})$ with the estimate{\rm :}
\begin{align*}
\left\Vert V\right\Vert_{C^{k,\alpha}}\leq C
\end{align*}
for some $C=C(k,\alpha,d,\left\Vert h\right\Vert_{\infty})$. \label{smoothness of V}
\end{rem}
Recall that the modulated energy is given by
\begin{align*}
&\mathcal{\mathcal{E}}_{\varepsilon,\hbar}(t)\coloneqq\underbrace{\frac{1}{2}\sum^d_{j=1}\mathrm{tr}\big((i\hbar\partial_{x_{j}}+u_{j})^{2}R_{\varepsilon,\hbar}(t)\big)}_{\coloneqq \mathcal{K}_{\varepsilon,\hbar}(t)}+\underbrace{\frac{\varepsilon}{2}\int_{\mathbb{T}^{d}}\left|\nabla V_{\varepsilon,\hbar}(t,x)\right|^{2}\,\dd x. }_{\coloneqq \mathcal{V}_{\varepsilon,\hbar}(t)}\\
&\quad\quad\quad\quad+\int_{\mathbb{T}^{d}}\Big(m_{\varepsilon,\hbar}(t,x)\log\big(\frac{m_{\varepsilon,\hbar}(t,x)}{\rho(t,x)}\big)-m_{\varepsilon,\hbar}(t,x)+\rho(t,x)\Big) \,\dd x\\
\end{align*}
and the total energy is given by
\begin{align}
\mathcal{F}_{\varepsilon,\hbar}(t)\coloneqq\frac{1}{2}\mathrm{tr}\left(-\hbar^{2}\Delta R_{\varepsilon,\hbar}(t)\right)+\frac{\varepsilon}{2}\int_{\mathbb{T}^{d}}\left|\nabla V_{\varepsilon,\hbar}(t,x)\right|^{2}\,\dd x+\int_{\mathbb{T}^{d}}V_{\varepsilon,\hbar}(t,x)m_{\varepsilon,\hbar}(t,x)\,\dd x.\label{total}
\end{align}

\begin{rem}\label{rem2.8}
It is classical that, given probability densities $\mu,\nu \in \mathcal{P}(\mathbb{T}^{d})\cap L^{1}(\mathbb{T}^d)$, the associated relative entropy is non-negative{\rm :} $\int_{\mathbb{T}^{d}}\mu\log \left(\frac{\mu}{\nu}\right)\,\dd x\geq 0$. 
This explains why the last terms in the definitions of $\mathcal{E}_{\varepsilon,\hbar}(t)$ 
and $\mathcal{F}_{\varepsilon,\hbar}(t)$ are non-negative and, as a result, $\mathcal{E}_{\varepsilon,\hbar}(t)\geq0$.
\end{rem}
\begin{rem}\label{rem2.9}
Note that
\begin{align*}
\mathcal{K}_{\varepsilon,\hbar}(t)=\frac{1}{2}\int_{\mathbb{T}^{d}}\left\vert (i\hbar\nabla+u)\psi_{\varepsilon,\hbar}\right\vert^{2}(t,x)\,\dd x
\end{align*}
and that $R_{\varepsilon,\hbar}(t)$ is determined by the Cauchy problem for 
the 
Hartree-type system{\rm :}
\begin{align*}
\begin{cases}
i\hbar\partial_{t}R_{\varepsilon,\hbar}(t)=[H_{\varepsilon,\hbar},R_{\varepsilon,\hbar}(t)],\\[0.5mm]
-\varepsilon \Delta V_{\varepsilon,\hbar}=\left\vert \psi_{\varepsilon,\hbar} \right\vert^{2}-e^{V_{\varepsilon,\hbar}},\\[0.5mm]
\psi_{\varepsilon,\hbar}|_{t=0}=\psi_{\varepsilon,\hbar}^{\mathrm{in}}.
\end{cases}
\end{align*}
It is instructive to invoke the Hartree formulation and the trace, since it clarifies the underlying algebraic structure.
\end{rem}
Before proving the stability estimate for $\mathcal{\mathcal{E}}_{\varepsilon,\hbar}(t)$ 
as stated in Theorem \ref{1body main intro}, we observe that the density is governed by an evolution equation 
and the conservation of the total energy $\mathcal{F}_{\varepsilon,\hbar}(t)$. 
In the forthcoming calculation, we frequently use the anticommutator which is denoted $\vee$ 
and defined by $A\vee B\coloneqq AB + BA$ for any given operators $A,B$.

\begin{lem}\label{density evolution}
Let $R_{\varepsilon,\hbar}(t)$ and $\rho_{\varepsilon,\hbar}(t,\cdot)$ be as in {\rm Theorem \ref{1body main intro}}. Then
\begin{equation}\label{equation for density}
\partial_{t}\rho_{\varepsilon,\hbar}+\mathrm{div}\,J_{\varepsilon,\hbar}=0.
\end{equation}
\end{lem}

\begin{proof} For any $a\in C(\mathbb{T}^{d};\mathbb{T}^{d}),$ we compute
\begin{align*}
\frac{\dd}{\dd t}\int_{\mathbb{T}^{d}}a(x)\rho_{\varepsilon,\hbar}(t,x)\,\dd x
&=\frac{\dd}{\dd t}\mathrm{tr}\big(aR_{\varepsilon,\hbar}(t)\big)
=\frac{1}{i\hbar}\mathrm{tr}\big(a [-\frac{\hbar^{2}}{2}\Delta+V_{\varepsilon,\hbar},R_{\varepsilon,\hbar}(t)]\big)\\
&=\frac{i}{\hbar}\mathrm{tr}\big([-\frac{\hbar^{2}}{2}\Delta,a]R_{\varepsilon,\hbar}(t)\big)
=\frac{1}{2}\underset{k}{\sum}\mathrm{tr}\big([-i\hbar\partial_{x_{k}}\vee\partial_{x_{k}},a]R_{\varepsilon,\hbar}(t)\big)\\
&=\frac{1}{2}\underset{k}{\sum}\mathrm{tr}\big((-i\hbar\partial_{x_{k}})\vee(\partial_{x_{k}}a)\,R_{\varepsilon,\hbar}(t)\big)=\int_{\mathbb{T}^{d}}\nabla a(x)\cdot J_{\varepsilon,\hbar}(t,x)\,\dd x,
\end{align*}
where we have used the equations in {\rm Remark \ref{rem2.9}}.
\end{proof}

\begin{lem}
\label{Conservation of energy} Let  $R_{\varepsilon,\hbar}(t)$ be as in {\rm Theorem
\ref{1body main intro}}. Then $\frac{\dd}{\dd t}\mathcal{F}_{\varepsilon,\hbar}(t)=0.$
\end{lem}
\begin{proof}
We first compute
\begin{align*}
 \frac{\dd}{\dd t}\mathrm{tr}\left(-\hbar^{2}\Delta R_{\varepsilon,\hbar}(t)\right)
 &=\frac{1}{i\hbar}\mathrm{tr}\big(-\hbar^{2}\Delta[-\frac{1}{2}\hbar^{2}\Delta+V_{\varepsilon,\hbar},R_{\varepsilon,\hbar}(t)]\big)\\
&=\frac{i}{\hbar}\mathrm{tr}\big([-\frac{1}{2}\hbar^{2}\Delta+V_{\varepsilon,\hbar},-\hbar^{2}\Delta]R_{\varepsilon,\hbar}(t)\big)=\frac{i}{\hbar}\mathrm{tr}\big([V_{\varepsilon,\hbar},-\hbar^{2}\Delta]R_{\varepsilon,\hbar}(t)\big),
\end{align*}
where the equations in {\rm Remark \ref{rem2.9}} have been used.

The commutator can be computed as
\begin{align*}
\left[V_{\varepsilon,\hbar},-\hbar^{2}\Delta\right]=\underset{k}{\sum}(-i\hbar\partial_{x_{k}})\vee\left[V_{\varepsilon,\hbar},-i\hbar\partial_{x_{k}}\right]=\underset{k}{\sum}(-i\hbar\partial_{x_{k}})\vee(i\hbar\partial_{x_{k}}V_{\varepsilon,\hbar}),
\end{align*}
so that
\begin{align}\label{eq:-20}
\frac{1}{2}\frac{\dd}{\dd t}\mathrm{tr}\left(-\hbar^{2}\Delta R_{\varepsilon,\hbar}(t)\right)&=\frac{i}{2\hbar}\underset{k}{\sum}\mathrm{tr}\big((-i\hbar\partial_{x_{k}})\vee(i\hbar\partial_{x_{k}}V_{\varepsilon,\hbar})R_{\varepsilon,\hbar}(t)\big)\notag\\
&=\frac{1}{2}\underset{k}{\sum}\mathrm{tr}\big((-i\hbar\partial_{x_{k}})\vee(-\partial_{x_{k}}V_{\varepsilon,\hbar})R_{\varepsilon,\hbar}(t)\big)\notag\\
&=-\int_{\mathbb{T}^{d}}(\nabla V_{\varepsilon,\hbar}\cdot J_{\varepsilon,\hbar})(t,x) \,\dd x.
\end{align}

Define 
$$
G(V_{\varepsilon,\hbar})=V_{\varepsilon,\hbar}e^{V_{\varepsilon,\hbar}}-e^{V_{\varepsilon,\hbar}}.
$$

Then it follows from \eqref{equation for density} that
\begin{align}\label{m}
-\int_{\mathbb{T}^d} \nabla V_{\varepsilon,\hbar} \cdot J_{\varepsilon,h}\,\dd x&=\int_{\mathbb{T}^d} \mathrm{div}\, J_{\varepsilon,h}\cdot V_{\varepsilon,\hbar}\,\dd x=-\int_{\mathbb{T}^d} \partial_t\rho^{\varepsilon}\cdot V_{\varepsilon,\hbar}\,\dd x\nonumber\\
&=\varepsilon\int_{\mathbb{T}^d}\Delta\partial_t V_{\varepsilon,\hbar}\cdot V_{\varepsilon,\hbar}\,\dd x-\int_{\mathbb{T}^d}e^{V_{\varepsilon,\hbar}}\partial_t V_{\varepsilon,\hbar}\cdot V_{\varepsilon,\hbar}\,\dd x\nonumber\\
&=-\frac{\varepsilon}{2}\frac{\dd}{\dd t}\int_{\mathbb{T}^d}|\nabla V_{\varepsilon,\hbar}|^2\,\dd x-\int_{\mathbb{T}^d}G'(V_{\varepsilon,\hbar})\partial_t V_{\varepsilon,\hbar}\,\dd x\nonumber\\
&=-\frac{\varepsilon}{2}\frac{\dd}{\dd t}\int_{\mathbb{T}^d}|\nabla V_{\varepsilon,\hbar}|^2\,\dd x-\frac{\dd}{\dd t}\int_{\mathbb{T}^d}G(V_{\varepsilon,\hbar}) \,\dd x.
\end{align}

Notice that
$$
\frac{\dd}{\dd t}\int_{\mathbb{T}^{d}}e^{V_{\varepsilon,\hbar}(t,x)}\,\dd x=0.
$$

Then, combining \eqref{eq:-20} with \eqref{m}, we conclude the proof.
\end{proof}

\medskip
\textit{Proof of Theorem \ref{1body main intro}}.
The proof is divided into five steps.

\smallskip
\textbf{1}.
Note that
\begin{align*}
\mathcal{\mathcal{E}}_{\varepsilon,\hbar}(t)&=\mathcal{F}_{\varepsilon,\hbar}(t)+\frac{1}{2}\underset{j}{\sum}\mathrm{tr}\left((i\hbar\partial_{x_{j}})\vee u^{j}R_{\varepsilon,\hbar}(t)\right)+\frac{1}{2}\int_{\mathbb{T}^{d}}\left(\rho_{\varepsilon,\hbar}|u\right|^{2})(t,x)\,\dd x\nonumber\\
&\quad\,\,+\int_{\mathbb{T}^{d}} \Big(\big(m_{\varepsilon,\hbar}\log\left(1/\rho\right)\big)(t,x)+\rho(t,x)\Big)\,\dd x.
\end{align*}
Therefore, by the conservation of energy (Lemma \ref{Conservation of energy}),
we have
\begin{align*}
\frac{\dd}{\dd t}\mathcal{\mathcal{E}}_{\varepsilon,\hbar}(t)& =\frac{1}{2}\frac{\dd}{\dd t}\underset{j}{\sum}\mathrm{tr}\left((i\hbar\partial_{x_{j}})\vee u^{j}R_{\varepsilon,\hbar}(t)\right)+\frac{1}{2}\frac{\dd}{\dd t}\int_{\mathbb{T}^{d}}(\rho_{\varepsilon,\hbar}\left|u\right|^{2})(t,x)\,\dd x\\
&\quad\,\,+\int_{\mathbb{T}^{d}} \partial_{t}\big(m_{\varepsilon,\hbar}\log(1/\rho)\big)(t,x)\,\dd x+\int_{\mathbb{T}^{d}}\partial_{t}\rho(t,x) \,\dd x:=\sum^4_{j=1}I^{j}(t).
\end{align*}

\textbf{2}. We first claim that the following identity holds:
\begin{align}
I^{1}(t)+I^{2}(t)=&\, \frac{1}{2}\underset{j}{\sum}\mathrm{tr}\big((i\hbar\partial_{x_{j}}+u^{j})\vee(\partial_{t}u^{j}+(u\cdot\nabla u)^{j})R_{\varepsilon,\hbar}(t)\big) \notag\\
&-\frac{1}{4}\underset{j,k}{\sum}\mathrm{tr}\big((i\hbar\partial_{x_{j}}+u^{j})\vee((i\hbar\partial_{x_{k}}+u^{k})\vee(\partial_{x_{k}}u^{j}))R_{\varepsilon,\hbar}(t)\big) \notag\\
&+\int_{\mathbb{T}^{d}}\rho_{\varepsilon,\hbar}(t,x)\,\nabla V_{\varepsilon,\hbar}(t,x)\cdot u(t,x) \,\dd x.\label{eq:-2}
\end{align}
We compute:
\begin{align*}
&\frac{\dd}{\dd t}\sum_j\mathrm{tr}\big(i\hbar\partial_{x_{j}}\vee u^{j}R_{\varepsilon,\hbar}(t)\big)+\frac{\dd}{\dd t}\mathrm{tr}\big(|u|^{2}R_{\varepsilon,\hbar}(t)\big)\\
&=\mathrm{tr}\big((\underset{j}{\sum}i\hbar\partial_{x_{j}}\vee u^{j}+|u|^{2})\partial_{t}R_{\varepsilon,\hbar}(t)\big)+\mathrm{tr}\big(\partial_{t}(\underset{j}{\sum}i\hbar\partial_{x_{j}}\vee u^{j}+|u|^{2})R_{\varepsilon,\hbar}(t)\big)\\
&=\frac{1}{i\hbar}\mathrm{tr}\big((\underset{j}{\sum}i\hbar\partial_{x_{j}}\vee u^{j}+|u|^{2})[H_{\varepsilon,\hbar}(t),R_{\varepsilon,\hbar}(t)]\big)
+\mathrm{tr}\big(\partial_{t}(\underset{j}{\sum}i\hbar\partial_{x_{j}}\vee u^{j}+|u|^{2})R_{\varepsilon,\hbar}(t)\big)\\
&=\frac{i}{\hbar}\mathrm{tr}\big([H_{\varepsilon,\hbar}(t),\underset{j}{\sum}i\hbar\partial_{x_{j}}\vee u^{j}+|u|^{2}]R_{\varepsilon,\hbar}(t)\big)+\mathrm{tr}\big(\partial_{t}(\underset{j}{\sum}i\hbar\partial_{x_{j}}\vee u^{j}+|u|^{2})R_{\varepsilon,\hbar}(t)\big)\\
&=\mathrm{tr}\big((\partial_{t}+\frac{i}{\hbar}[H_{\varepsilon,\hbar}(t),\cdot\,])(\underset{j}{\sum}i\hbar\partial_{x_{j}}\vee u^{j}+|u|^{2})R_{\varepsilon,\hbar}(t)\big)\\
&=\sum_j\mathrm{tr}\big((\partial_{t}+\frac{i}{\hbar}[-\frac{\hbar^{2}}{2}\Delta,\cdot\,])(i\hbar\partial_{x_{j}}\vee u^{j}+\frac{1}{2}u^{j}\vee u^{j})R_{\varepsilon,\hbar}(t)\big)\\
&\quad\,\,+\underset{j}{\sum}\mathrm{tr}\big((\frac{i}{\hbar}[V_{\varepsilon,\hbar},\cdot\,])(i\hbar\partial_{x_{j}}\vee u^{j}+\frac{1}{2}u^{j}\vee u^{j})R_{\varepsilon,\hbar}(t)\big)\\
&=:\underset{j}{\sum}\mathrm{tr}\big(J_{1,j}R_{\varepsilon,\hbar}(t)\big)+\underset{j}{\sum}\mathrm{tr}\big(J_{2,j}R_{\varepsilon,\hbar}(t)\big).
\end{align*}
We start with the term $J_{1,j}$. Using the  Leibniz rule: $\left[A,B\vee C\right]=\left[A,B\right]\vee C+\left[A,C\right]\vee B,$ we can write $J_{1,j}$ as
\begin{align}
J_{1,j}&=\big(\partial_{t}+\frac{i}{\hbar}[-\frac{\hbar^{2}}{2}\Delta,\cdot\,]\big)\big((\frac{1}{2}u^{j}+i\hbar\partial_{x_{j}})\vee u^{j}\big) \notag\\
&=\big((\partial_{t}+\frac{i}{\hbar}[-\frac{\hbar^{2}}{2}\Delta,\cdot\,])(\frac{1}{2}u^{j}+i\hbar\partial_{x_{j}})\big)\vee u^{j}+(\frac{1}{2}u^{j}+i\hbar\partial_{x_{j}})\vee\big((\partial_{t}+\frac{i}{\hbar}[-\frac{\hbar^{2}}{2}\Delta,\cdot\,])u^{j}\big) \notag\\
&=\big((\partial_{t}+\frac{i}{\hbar}[-\frac{\hbar^{2}}{2}\Delta,\cdot\,])u^{j}\big)\vee u^{j}+i\hbar\partial_{x_{j}}\vee\big((\partial_{t}+\frac{i}{\hbar}[-\frac{\hbar^{2}}{2}\Delta,\cdot\,])u^{j}\big) \notag\\
&=(u^{j}+i\hbar\partial_{x_{j}})\vee\big((\partial_{t}+\frac{i}{\hbar}[-\frac{\hbar^{2}}{2}\Delta,\cdot\,])u^{j}\big) \notag\\
&=(u^{j}+i\hbar\partial_{x_{j}})\vee\big(\partial_{t}u^{j}+\underset{k}{\sum}u^{k}\partial_{x_{k}}u^{j}\big)+(u^{j}+i\hbar\partial_{x_{j}})\vee\big((\frac{i}{\hbar}[-\frac{\hbar^{2}}{2}\Delta,\cdot\,]-\underset{k}{\sum}u^{k}\partial_{x_{k}})u^{j}\big).\label{eq for J1j}
\end{align}
Notice that
\begin{align*}
\frac{i}{\hbar}[-\frac{\hbar^{2}}{2}\partial_{x_{k}x_{k}},u^{j}]
=-\frac{i\hbar}{2}\partial_{x_{k}}\vee [\partial_{x_{k}},u^{j}]
=-\frac{1}{2}i\hbar\partial_{x_{k}}\vee\partial_{x_{k}}u^{j},
\end{align*}
which shows that the right-hand side of \eqref{eq for J1j} is
\begin{equation*}
(u^{j}+i\hbar\partial_{x_{j}})\vee(\partial_{t}u^{j}+\underset{k}{\sum}u^{k}\partial_{x_{k}}u^{j})-\frac{1}{2}\underset{k}{\sum}(u^{j}+i\hbar\partial_{x_{j}})\vee\big((i\hbar\partial_{x_{k}}+u^{k})\vee\partial_{x_{k}}u^{j}\big),\end{equation*}
so that
\begin{align}
\underset{j}{\sum}\mathrm{tr}\big(J_{1,j}R_{\varepsilon,\hbar}(t)\big)
&=\underset{j}{\sum}\mathrm{tr}\big((u^{j}+i\hbar\partial_{x_{j}})\vee(\partial_{t}u^{j}+(u\cdot\nabla u)^{j})R_{\varepsilon,\hbar}(t)\big) \notag \\
&\quad\,\,-\frac{1}{2}\underset{k,j}{\sum}\mathrm{tr}\big((u^{j}+i\hbar\partial_{x_{j}})\vee((i\hbar\partial_{x_{k}}+u^{k})\vee\partial_{x_{k}}u^{j})R_{\varepsilon,\hbar}(t)\big).   \label{eq for J1}
\end{align}
As for $J_{2,j}$, we see that
\begin{align*}
[V_{\varepsilon,\hbar},u^{j}\vee u^{j}]=0, \qquad  
-[V_{\varepsilon,\hbar},\partial_{x_{j}}\vee u^{j}]=2\partial_{x_{j}}V_{\varepsilon,\hbar}u^{j},
\end{align*}
which yields
\begin{align*}
J_{2,j}=2\int_{\mathbb{T}^{d}}\rho_{\varepsilon,\hbar}(t,x)\partial_{x_{j}}V_{\varepsilon,\hbar}(t,x)\,u^j(t,x)\,\dd x.  \end{align*}
Then
\begin{align}
\underset{j}{\sum}\mathrm{tr}\big(J_{2,j}R_{\varepsilon,\hbar}(t)\big)=2\int_{\mathbb{T}^{d}}\rho_{\varepsilon,\hbar}(t,x)\nabla V_{\varepsilon,\hbar}(t,x)\cdot u(t,x)\,\dd x.
\label{eq for J2}
\end{align}
Combining \eqref{eq for J1} with \eqref{eq for J2}, we obtain \eqref{eq:-2}.

\smallskip
\textbf{3}. We manipulate the last term on the right-hand side of \eqref{eq:-2}:
\begin{align}
&\int_{\mathbb{T}^{d}}\rho_{\varepsilon,\hbar}(t,x)\,\nabla V_{\varepsilon,\hbar}(t,x)\cdot u(t,x)\,\dd x\notag\\
&=\int_{\mathbb{T}^{d}}\left(e^{V_{\varepsilon,\hbar}(t,x)}-\varepsilon\Delta V_{\varepsilon,\hbar}(t,x)\right)\nabla V_{\varepsilon,\hbar}(t,x)\cdot u(t,x)\,\dd x\notag\\
&=\int_{\mathbb{T}^{d}} e^{V_{\varepsilon,\hbar}(t,x)}\nabla V_{\varepsilon,\hbar}(t,x)\cdot u(t,x)\,\dd x
-\varepsilon\int_{\mathbb{T}^{d}}\Delta V_{\varepsilon,\hbar}(t,x)\nabla V_{\varepsilon,\hbar}(t,x)\cdot u(t,x)\,\dd x\notag\\
&=\int_{\mathbb{T}^{d}}\nabla e^{V_{\varepsilon,\hbar}(t,x)}\cdot u(t,x)\,\dd x
-\varepsilon\int_{\mathbb{T}^{d}}\big(\nabla:(\nabla V_{\varepsilon,\hbar}\otimes\nabla V_{\varepsilon,\hbar})\cdot u\big)(t,x)\,\dd x 
\notag\\
&\quad\,\,
+\frac{\varepsilon}{2}\int_{\mathbb{T}^{d}}(\nabla\left|\nabla V_{\varepsilon,\hbar}\right|^{2}\cdot u)(t,x)\,\dd x
\notag\\ 
&=-\int_{\mathbb{T}^{d}} e^{V_{\varepsilon,\hbar}(t,x)}\mathrm{div}\,u(t,x)\,\dd x 
+\varepsilon\int_{\mathbb{T}^{d}} Du(t,x):\left(\nabla V_{\varepsilon,\hbar}\otimes\nabla V_{\varepsilon,\hbar}\right)(t,x)\,\dd x \notag\\
&\quad\,\,-\frac{\varepsilon}{2}\int_{\mathbb{T}^{d}}\left|\nabla V_{\varepsilon,\hbar}(t,x)\right|^{2}\mathrm{div}\,u(t,x)\,\dd x. \hspace{-0.5 cm} \label{eq:-19-1}
\end{align}
Next, using Lemma \ref{density evolution}, we can prove the following estimates:
\begin{align}
I^{3}(t)+I^{4}(t)&=\frac{\dd}{\dd t}\int_{\mathbb{T}^{d}} m_{\varepsilon,\hbar}(t,x)\log(1/\rho(t,x))\,\dd x+\int_{\mathbb{T}^{d}}\partial_{t}\rho(t,x)\,\dd x \notag\\
&=\int_{\mathbb{T}^{d}}\partial_{t}e^{V_{\varepsilon,\hbar}(t,x)}\log(1/\rho(t,x))\,\dd x-\int_{\mathbb{T}^{d}}\frac{e^{V_{\varepsilon,\hbar}(t,x)}}{\rho(t,x)}\partial_{t}\rho(t,x)\,\dd x+\int_{\mathbb{R}^{d}}\partial_{t}\rho(t,x)\,\dd x \notag\\
&=\int_{\mathbb{T}^{d}}\Big(-\frac{e^{V_{\varepsilon,\hbar}(t,x)}}{\rho(t,x)}+1\Big)\partial_{t}\rho(t,x)\,\dd x+\int_{\mathbb{T}^{d}}\partial_{t}e^{V_{\varepsilon,\hbar}(t,x)}\log(1/\rho(t,x))\,\dd x \notag\\
&\quad\,\,-\int_{\mathbb{T}^{d}}\mathrm{div}\,J_{\varepsilon,\hbar}(t,x)\log(1/\rho(t,x))\,\dd x,\label{eq:-24}
\end{align}
\begin{align}
\frac{1}{2}&\mathrm{tr}\left((u^{j}+i\hbar\partial_{x_{j}})\vee\partial_{x_{j}}\log\left(\rho(t,x)\right)R_{\varepsilon,\hbar}(t)\right) \notag\\
=&\int_{\mathbb{T}^{d}} u(t,x)\cdot\nabla\log(\rho(t,x))\rho_{\varepsilon,\hbar}(t,x)\,\dd x-\int_{\mathbb{T}^{d}} J_{\varepsilon,\hbar}(t,x)\cdot\nabla\log(\rho(t,x))\,\dd x \notag\\
=&\int_{\mathbb{T}^{d}} e^{V_{\varepsilon,\hbar}(t,x)}u(t,x)\cdot\nabla\log(\rho(t,x))\,\dd x-\varepsilon\int_{\mathbb{T}^{d}}\Delta V_{\varepsilon,\hbar}(t,x)u(t,x)\cdot\nabla\log(\rho(t,x))\,\dd x \notag\\
&-\int_{\mathbb{T}^{d}} J_{\varepsilon,\hbar}(t,x)\cdot\nabla\log(\rho(t,x))\,\dd x.\label{eq:}
\end{align}
Therefore, gathering \eqref{eq:-2} with \eqref{eq:-19-1}--\eqref{eq:-24} yields
\begin{align*}
&I^{1}(t)+I^{2}(t)+I^{3}(t)+I^{4}(t)\\
&=\frac{1}{2}\underset{j}{\sum}\mathrm{tr}\left(\left(i\hbar\partial_{x_{j}}+u^{j}\right)\vee\left(\partial_{t}u^{j}+(u\cdot\nabla u)^{j}\right)R_{\varepsilon,\hbar}(t)\right)\\
&\quad\,\,-\frac{1}{4}\underset{j,k}{\sum}\mathrm{tr}\left(\left(i\hbar\partial_{x_{j}}+u^{j}\right)\vee\left(\left(i\hbar\partial_{x_{j}}+u^{j}\right)\vee(\partial_{x_{k}}u^{j})\right)R_{\varepsilon,\hbar}(t)\right)\\
&\quad\,\,-\int_{\mathbb{T}^{d}} e^{V_{\varepsilon,\hbar}(t,x)}\mathrm{div}\,u(t,x)\,\dd x
+\varepsilon\int_{\mathbb{T}^{d}} Du(t,x):\left(\nabla V_{\varepsilon,\hbar}\otimes\nabla V_{\varepsilon,\hbar}\right)(t,x)\,\dd x\\
&\quad\,\,-\frac{\varepsilon}{2}\int_{\mathbb{T}^{d}}\left|\nabla V_{\varepsilon,\hbar}(t,x)\right|^{2}\mathrm{div}\,u(t,x)\,\dd x\\
&\quad\,\,+\int_{\mathbb{T}^{d}}\Big(-\frac{e^{V_{\varepsilon,\hbar}(t,x)}}{\rho(t,x)}+1\Big)\partial_{t}\rho(t,x)\,\dd x+\varepsilon\int_{\mathbb{T}^{d}} \partial_{t}\Delta V_{\varepsilon,\hbar}(t,x)\log(1/\rho(t,x))\,\dd x\\
&\quad\,\,-\int_{\mathbb{T}^{d}}\mathrm{div}\,J_{\varepsilon,\hbar}(t,x)\log(1/\rho(t,x))\,\dd x.
\end{align*}
Thanks to equation \eqref{eq:}, we can rewrite the last term on the right-hand side of the last identity to obtain
\begin{align*}
&I^{1}(t)+I^{2}(t)+I^{3}(t)+I^{4}(t)\\
&=\frac{1}{2}\underset{j}{\sum}\mathrm{tr}\big((i\hbar\partial_{x_{j}}+u^{j})\vee(\partial_{t}u^{j}+(u\cdot \nabla u)^{j})R_{\varepsilon,\hbar}(t)\big)\\
&\quad\,\,-\frac{1}{4}\underset{j,k}{\sum}\mathrm{tr}\big((i\hbar\partial_{x_{j}}+u^{j})\vee((i\hbar\partial_{x_{j}}+u^{j})\vee(\partial_{x_{k}}u^{j}))R_{\varepsilon,\hbar}(t)\big)\\
&\quad\,\,-\int_{\mathbb{T}^{d}} e^{V_{\varepsilon,\hbar}(t,x)}\mathrm{div}\,u(t,x)\,\dd x+\varepsilon\int_{\mathbb{T}^{d}} Du(t,x):\left(\nabla V_{\varepsilon,\hbar}\otimes\nabla V_{\varepsilon,\hbar}\right)(t,x)\,\dd x\nonumber\\
&\quad\,\,-\frac{\varepsilon}{2}\int_{\mathbb{T}^{d}}\left|\nabla_{x}V_{\varepsilon,\hbar}(t,x)\right|^{2}\mathrm{div}\,u(t,x)\,\dd x\\
&\quad\,\,+\int_{\mathbb{T}^{d}}\Big(-\frac{e^{V_{\varepsilon,\hbar}(t,x)}}{\rho(t,x)}+1\Big)\partial_{t}\rho(t,x)\,\dd x+\varepsilon\int_{\mathbb{T}^{d}}\partial_{t}\Delta V_{\varepsilon,\hbar}(t,x)\log(1/\rho(t,x))\,\dd x\\
&\quad\,\,+\varepsilon\int_{\mathbb{T}^{d}}\Delta V_{\varepsilon,\hbar}(t,x)u(t,x)\cdot\nabla\log(\rho(t,x))\,\dd x+\frac{1}{2}\mathrm{tr}\big((u^{j}+i\hbar\partial_{x_{j}})\vee(\partial_{x_{j}}\log(\rho))\,R_{\varepsilon,\hbar}(t)\big)\\
&\quad\,\,-\int_{\mathbb{T}^{d}} e^{V_{\varepsilon,\hbar}(t,x)}u(t,x)\cdot\nabla\log(\rho(t,x))\,\dd x.
\end{align*}
Rearranging the terms leads to
\begin{align}
&I^{1}(t)+I^{2}(t)+I^{3}(t)+I^{4}(t) \notag\\
&=-\int_{\mathbb{T}^{d}} e^{V_{\varepsilon,\hbar}(t,x)}\big(\partial_{t}\log(\rho(t,x))+\mathrm{div}\,u(t,x)+u(t,x)\cdot\nabla\log(\rho(t,x))\big)\,\dd x
+\int_{\mathbb{T}^{d}}\partial_{t}\rho(t,x)\,\dd x\notag\\[1mm]
&\quad\,\,+\frac{1}{2}\underset{j}{\sum}\mathrm{tr}\big((i\hbar\partial_{x_{j}}+u^{j})\vee(\partial_{t}u^{j}+(u\cdot\nabla u)^{j}+\partial_{x_{j}}\log(\rho(t,x)))R_{\varepsilon,\hbar}(t)\big) \notag\\
&\quad\,\,+\varepsilon\int_{\mathbb{T}^{d}}\Delta V_{\varepsilon,\hbar}(t,x)u(t,x)\cdot\nabla\log(\rho(t,x))\,\dd x+\varepsilon\int_{\mathbb{T}^{d}}\partial_{t}\Delta V_{\varepsilon,\hbar}(t,x)\log(1/\rho(t,x))\,\dd x \notag\\[1mm]
&\quad\,\,-\frac{1}{4}\underset{j,k}{\sum}\mathrm{tr}\big((i\hbar\partial_{x_{j}}+u^{j})\vee((i\hbar\partial_{x_{j}}+u^{j})\vee(\partial_{x_{k}}u^{j}))R_{\varepsilon,\hbar}(t)\big) \notag\\
&\quad\,\,+\varepsilon\int_{\mathbb{T}^{d}}Du(t,x):\left(\nabla V_{\varepsilon,\hbar}\otimes\nabla V_{\varepsilon,\hbar}\right)(t,x)\,\dd x-\frac{\varepsilon}{2}\int_{\mathbb{T}^{d}}\left|\nabla_{x}V_{\varepsilon,\hbar}(t,x)\right|^{2}\mathrm{div}\,u(t,x)\,\dd x.\label{eq:-3}
\end{align}
Using \eqref{isothermal log}, we see that the first three
terms in \eqref{eq:-3} vanish, which shows that
\begin{align}
&I^{1}(t)+I^{2}(t)+I^{3}(t)+I^{4}(t) \notag\\
&=\varepsilon\int_{\mathbb{T}^{d}}\Delta V_{\varepsilon,\hbar}(t,x)u(t,x)\cdot\nabla\log(\rho(t,x))\,\dd x
+\varepsilon\int_{\mathbb{T}^{d}}\partial_{t}\Delta V_{\varepsilon,\hbar}(t,x)\log(1/\rho(t,x))\,\dd x \notag\\
&\quad-\frac{1}{4}\underset{j,k}{\sum}\mathrm{tr}\big((i\hbar\partial_{x_{j}}+u^{j})\vee((i\hbar\partial_{x_{j}}+u^{j})\vee(\partial_{x_{k}}u^{j}))R_{\varepsilon,\hbar}(t)\big) \notag\\
&\quad+\varepsilon\int_{\mathbb{T}^{d}}Du(t,x):\left(\nabla V_{\varepsilon,\hbar}\otimes\nabla V_{\varepsilon,\hbar}\right)(t,x)\,\dd x-\frac{\varepsilon}{2}\int_{\mathbb{T}^{d}}\left|\nabla_{x}V_{\varepsilon,\hbar}(t,x)\right|^{2}\mathrm{div}\,u(t,x)\,\dd x\nonumber\\
&=:\sum^{5}_{k=1}\mathcal{T}_{k}.\label{eq:-4}
\end{align}

\smallskip
\textbf{4}. We proceed by estimating separately each of the summands
$\mathcal{T}_{k}$ on the right-hand side in \eqref{eq:-4}. We have
\begin{align}\label{eq:-6}
\mathcal{T}_{1}&=\varepsilon\int_{\mathbb{T}^{d}}\Delta V_{\varepsilon,\hbar}(t,x)\,u(t,x)\cdot\nabla\log(\rho(t,x))\,\dd x\nonumber\\
&=-\varepsilon\int_{\mathbb{T}^{d}}\nabla V_{\varepsilon,\hbar}(t,x)\cdot\nabla\big(u(t,x)\cdot\nabla\log(\rho(t,x))\big)\,\dd x \notag\\[1mm]
&\leq\sqrt{\varepsilon}\left\Vert \sqrt{\varepsilon}\nabla V_{\varepsilon,\hbar}\right\Vert _{L^{\infty}_{t}L^{2}_{x}}\left\Vert \nabla\big(u\cdot\nabla\log(\rho)\big)\right\Vert _{L^{\infty}_{t}L^{2}_{x}} \notag\\[1mm]
&\leq\sqrt{\varepsilon}\mathcal{F}_{0}\left\Vert \nabla\big(u\cdot\nabla\log(\rho)\big)\right\Vert _{L^{\infty}_tL^{2}_x},
\end{align}
where we have used that $\left\Vert \sqrt{\varepsilon}\nabla V_{\varepsilon,\hbar}\right\Vert _{L^{\infty}_tL^{2}_x}$
is uniformly bounded in $(\varepsilon,\hbar)$ by  $\mathcal{F}_{0}$ thanks to Lemma \ref{Conservation of energy}
and the assumption that $\mathcal{F}_{\varepsilon,\hbar}(0)\leq\mathcal{F}_{0}$. 
Furthermore, integrating by parts in time yields
\begin{align}\label{eq:-5}
\int_{0}^{t}\mathcal{T}_{2}(s)\,\dd s&=-\varepsilon\int_{0}^{t}\int_{\mathbb{T}^{d}}\partial_{s}\Delta V_{\varepsilon,\hbar}(s,x)\log(\rho(s,x))\,\dd s\dd x \nonumber\\
&=\varepsilon\int_{0}^{t}\int_{\mathbb{T}^{d}}\Delta V_{\varepsilon,\hbar}\partial_{s}\log(\rho(s,x))\,\dd s\dd x-\varepsilon\int_{\mathbb{T}^{d}}\nabla\log(\rho(0,x))\nabla V_{\varepsilon,\hbar}(0,x)\,\dd x \nonumber\\
&\quad\,\,+\varepsilon\int_{\mathbb{T}^{d}}\nabla\log(\rho(t,x))\nabla V_{\varepsilon,\hbar}(t,x)\,\dd x \nonumber\\
&\leq\sqrt{\varepsilon}\left\Vert \log(\rho)\right\Vert _{W_{t}^{1,\infty}H_{x}^{1}}
\int_{0}^{t}\left\Vert \sqrt{\varepsilon}\nabla V_{\varepsilon}(\tau,\cdot)\right\Vert _{2}\,\dd \tau \notag\\
&\leq\sqrt{\varepsilon}\left\Vert \log(\rho)\right\Vert _{W_{t}^{1,\infty}H_{x}^{1}}t\mathcal{F}_{0}.
\end{align}
Clearly, we have the inequality:
\begin{equation}
\left|\mathcal{T}_{3}\right|\leq2\left\Vert \nabla u\right\Vert _{L^\infty_{t,x}}\mathrm{tr}\big((u^{j}+i\hbar\partial_{x_{j}})^{2}R_{\varepsilon,\hbar}(t)\big)=4\left\Vert \nabla u\right\Vert _{L^\infty_{t,x}}\mathcal{K}_{\varepsilon,\hbar}(t).\label{eq:-7}
\end{equation}
Finally, observe that
\begin{align}
&\left|\mathcal{T}_{4}\right|\leq2\varepsilon\left\Vert \nabla u\right\Vert _{L^\infty_{t,x}}\int_{\mathbb{T}^{d}}\left|\nabla V_{\varepsilon,\hbar}(t,x)\right|^{2}\,\dd x\leq2\left\Vert \nabla u\right\Vert _{L^\infty_{t,x}}\mathcal{V}_{\varepsilon,\hbar}(t),
\label{eq:-8}\\[1mm]
&\left|\mathcal{T}_{5}\right|\leq\frac{\varepsilon}{2}\left\Vert \mathrm{div}\,u\right\Vert _{L^\infty_{t,x}}\int_{\mathbb{T}^{d}}\left|\nabla V_{\varepsilon,\hbar}(t,x)\right|^{2}\,\dd x\leq\left\Vert \mathrm{div}\,u\right\Vert _{L^\infty_{t,x}}\mathcal{V}_{\varepsilon,\hbar}(t).\label{eq:-9}
\end{align}
Gathering \eqref{eq:-6}--\eqref{eq:-9}, we have
\[
\mathcal{E}_{\varepsilon,\hbar}(t)\leq\mathcal{E}_{\varepsilon,\hbar}(0)+C\Big(\int_{0}^{t}\mathcal{E}_{\varepsilon,\hbar}(s)\,\dd s+\sqrt{\varepsilon}t\mathcal{F}_{0}\Big),
\]
where $C=C(\Vert \nabla u\Vert_{L^\infty_{t,x}},\Vert \log(\rho)\Vert _{W_{t}^{1,\infty}H_{x}^{1}},\Vert \nabla(u\cdot\nabla\log(\rho))\Vert_{L^{\infty}_{t}L^{2}_{x}}).$
The estimate inequality of Theorem \ref{1body main intro} follows now by the Gr\"onwall inequality.

\smallskip
\textbf{5}. Using the convergence 
$$
\underset{t \in [0,T]}{\sup}\mathcal{E}_{\varepsilon,\hbar}(t) \underset{\varepsilon+\hbar \rightarrow 0}{\longrightarrow} 0
$$
proved in Step 4,  we can now establish the convergence of $\rho_{\varepsilon,\hbar}$ and $J_{\varepsilon,\hbar}$. 
The convergence
\begin{align*}
 \varepsilon\int_{\mathbb{T}^{d}} \left\vert \nabla V_{\varepsilon,\hbar}(t,x)\right\vert^{2}\,\dd x \underset{\varepsilon+\hbar\rightarrow0}{\longrightarrow} 0
\end{align*}
implies that
\begin{align*}
 \underset{t\in [0,T]}{\sup}{\left\Vert \rho_{\varepsilon,\hbar}(t,\cdot)-m_{\varepsilon,\hbar}(t,\cdot)\right\Vert_{\dot{H}^{-1}}\underset{\varepsilon+\hbar\rightarrow0}{\longrightarrow} 0}.
\end{align*}

To prove the convergence of $\rho_{\varepsilon,\hbar},$ we need the following  Csisz\'{a}r-Kullback-Pinsker inequality:
\begin{lem}[\cite{villani2009optimal}]\label{CZK}
Let $(\rho,m)\in \big(\mathcal{P}(\mathbb{T}^{d})\cap L^{1}(\mathbb{T}^{d})\big)^2$. Then
\begin{align*}
 \|\rho-m\|_{1}\leq\sqrt{2\int_{\mathbb{T}^{d}}\log\big(\frac{m(x)}{\rho(x)}\big)m(x) \,\dd x}.
\end{align*}
\end{lem}

\smallskip
By virtue of Lemma \ref{CZK}, the convergence
\begin{align*}
    \int_{\mathbb{T}^{d}} m_{\varepsilon,\hbar}(t,x)\log\big(\frac{m_{\varepsilon,\hbar}(t,x)}{\rho(t,x)}\big)\,\dd x\underset{\varepsilon+\hbar\rightarrow0}{\longrightarrow} 0
\end{align*}
implies
\begin{align*}
 \underset{t\in [0,T]}{\sup}{\left\Vert m_{\varepsilon,\hbar}(t,\cdot)-\rho(t,\cdot)\right\Vert_{1}\underset{\varepsilon+\hbar\rightarrow0}{\longrightarrow} 0}.
\end{align*}
Thus, by the triangle inequality, we obtain
\begin{align}
 \underset{t\in [0,T]}{\sup}{\left\Vert \rho_{\varepsilon,\hbar}(t,\cdot)-\rho(t,\cdot)\right\Vert_{\dot{H}^{-1}}\underset{\varepsilon+\hbar\rightarrow0}{\longrightarrow} 0}.  \label{Convergence of density}
\end{align}
The convergence of $J_{\varepsilon,\hbar}$ is established as follows:
Given a Lipschitz vector field $b \in W^{1,\infty}(\mathbb{T}^{d}; \mathbb{T}^{d})$, we have
\begin{align*}
&\int_{\mathbb{T}^{d}} \big(J_{\varepsilon,\hbar}(t,x)-\rho(t,x)u(t,x)\big)b(x)\,\dd x\\
&=
\int_{\mathbb{T}^{d}} \big(\rho_{\varepsilon,\hbar}(t,x)-\rho(t,x)\big)u(t,x)b(x)\,\dd x
+
\int_{\mathbb{T}^{d}} \big(J_{\varepsilon,\hbar}(t,x)-\rho_{\varepsilon,\hbar}(t,x)u(t,x)\big)b(x)\,\dd x.
\end{align*}
The first integral on the right-hand side above tends to $0$ as $\varepsilon+\hbar\rightarrow 0$ by \eqref{Convergence of density}, while the second integral tends to $0$ due to the convergence of the kinetic part: $\underset{t\in [0,T]}{\sup}\mathcal{K}_{\varepsilon,\hbar}(t)\underset{\varepsilon+\hbar\rightarrow0}{\longrightarrow} 0$. 
Indeed, we have
\begin{align*}
\bigg\vert \int_{\mathbb{T}^{d}} \big(J_{\varepsilon,\hbar}(t,x)-\rho_{\varepsilon,\hbar}(t,x)u(t,x)\big)b(x)\,\dd x\bigg\vert
&=\bigg\vert \frac{1}{2}\sum_{j=1}^{d}\mathrm{tr}\big((i\hbar\partial_{x_{j}}+u^{j})\vee b^{j}R_{\varepsilon,\hbar}(t)\big)\bigg\vert\\
&\leq  2\left\Vert b\right\Vert_{\infty}\sqrt{\mathcal{K}_{\varepsilon,\hbar}(t)},
\end{align*}
by the Cauchy-Schwartz inequality.
This completes the proof of Theorem \ref{1body main intro}.
$\hfill\Box$

\section{The N-Body Problem}
This section concerns the $N$-body quantum mean-field limit. Being
a triple limit, this limit is naturally more involved and imposes
additional considerations, namely the dimensionality and an asymptotic
relation between $\varepsilon$ and $N$. Unless otherwise stated, throughout this section, we take $d=1$. Let us start by recalling the basic notation from the theory of density operators.

\begin{defn}
Let $R_{N}\in \mathcal{D}_{s}(\mathfrak{H}^{\otimes N})$. 
\begin{enumerate}
\item[\rm (i)]
For each $1\leq k\leq N,$ define \textit{the $k$-th marginal of $R_{N}$}, as the unique element $R_{N:k}\in \mathcal{D}_{s}(\mathfrak{H}^{\otimes k})$, such that
\begin{align*}
\mathrm{tr}_{\mathfrak{H}^{\otimes k}}(A_{k}R_{N:k})=\mathrm{tr}_{\mathfrak{H}^{\otimes N}}\big((A_{k}\otimes I^{\otimes(N-k)})R_{N}\big)
\end{align*}
for all bounded operators $A_{k}$ on $\mathfrak{H}^{\otimes k}$. \label{def of marginal}

\smallskip
\item[\rm (ii)] 
The \textit{current} of $R_{N}$, denoted by $J_{\hbar,N:1}$, is the unique signed Radon measure on $\mathbb{T}$ such that, for all $a \in C(\mathbb{T})$, 
 \begin{align*}
\int_{\mathbb{T}}a(x)J_{\hbar,N:1}(\dd x)=\frac{1}{2}\mathrm{tr}\big(a\lor(-i\hbar\partial_{x})R_{N:1}\big).
 \end{align*}
 \label{def of current}
 \end{enumerate}
\end{defn}
Recall that the quantum
Hamiltonian is given by
\begin{align*}
\mathscr{H}_{\varepsilon,\hbar,N}(X_{N})
&=-\frac{\hbar^{2}}{2}\sum^N_{k=1}\Delta_{x_{k}}+N\int_{\mathbb{T}}V_{\varepsilon,X_{N}}(x)m_{\varepsilon,X_{N}}(x)\,\dd x\nonumber\\
&\quad\,\,+\frac{N}{2\varepsilon}\int_{\mathbb{T}\times\mathbb{T}}K(x-y)(\mu_{X_{N}}-m_{\varepsilon,X_{N}})^{\otimes2}(\dd x\dd y)\\
&=: \mathscr{K}_{N,\hbar}+\mathscr{I}_{\varepsilon,X_{N}}+\mathscr{V}_{\varepsilon,X_{N}},
\end{align*}
where
\begin{equation}
-\varepsilon V_{\varepsilon,X_{N}}''=\mu_{X_{N}}-m_{\varepsilon,X_{N}}
\qquad \mbox{with $\,\,m_{\varepsilon,X_{N}}=e^{V_{\varepsilon,X_{N}}}$}.\label{eq:-11}
\end{equation}
The associated von Neumann equation is
\begin{align*}
\begin{cases}
i\hbar\partial_{t}R_{\varepsilon,\hbar,N}=\left[\mathscr{H}_{\varepsilon,\hbar,N},R_{\varepsilon,\hbar,N}\right],\\[1mm]
R_{\varepsilon,\hbar,N}(0)=R_{\varepsilon,\hbar,N}^{\mathrm{in}}.
\end{cases}
\end{align*}
Recall that the modulated energy is given by
\begin{align*}
&\mathcal{\mathcal{E}}_{\varepsilon,\hbar,N}(t)\\
&=\underbrace{\frac{1}{2N}\sum^N_{j=1}\mathrm{tr}\big((i\hbar\partial_{x_{j}}+u(t,x_{j}))^{2}R_{\varepsilon,\hbar,N}(t)\big)}_{\coloneqq \mathcal{K}_{\varepsilon,\hbar,N}(t)}\\
&\quad\,+\int_{\mathbb{T}^{N}}\int_{\mathbb{T}}\Big(m_{\varepsilon,X_{N}}(t,x)\log(\frac{m_{\varepsilon,X_{N}}(t,x)}{\rho(t,x)})
-m_{\varepsilon,X_{N}}(t,x)+\rho(t,x)\Big)\,\dd x\,\rho_{\varepsilon,\hbar,N}(t,X_{N})\,\dd X_{N}\\
&\quad\,+\underbrace{\frac{1}{2\varepsilon}\int_{\mathbb{T}^{N}}\int_{\mathbb{T}\times\mathbb{T}}K(x-y)(\mu_{X_{N}}-m_{\varepsilon,X_{N}})^{\otimes2}(\dd x \dd y)\,\rho_{\varepsilon,\hbar,N}(t,X_{N})\,\dd X_{N}}_{\coloneqq \mathcal{V}_{\varepsilon,\hbar,N}(t)},
\end{align*}
and the total energy is given by
\begin{align}
\mathcal{F}_{\varepsilon,\hbar,N}(t)
\coloneqq&\,\frac{1}{2N}\mathrm{tr}\big(-\hbar^{2}\sum^{N}_{j=1}\Delta_{x_{j}}R_{\varepsilon,\hbar,N}(t)\big) \notag\\
&\,\,+\int_{\mathbb{T}^{N}}\int_{\mathbb{T}}V_{\varepsilon,X_{N}}(x)m_{\varepsilon,X_{N}}(x)\,\,\dd x\,\rho_{\varepsilon,\hbar,N}(t,X_{N})\,\dd X_{N} \notag\\
&\,\,+\frac{1}{2\varepsilon}\int_{\mathbb{T}^{N}}\int_{\mathbb{T}\times\mathbb{T}}K(x-y)(\mu_{X_{N}}-m_{\varepsilon,X_{N}})^{\otimes2}(\dd x\dd y)\,\rho_{\varepsilon,\hbar,N}(t,X_{N})\,\dd X_{N}. \label{totalN}
\end{align}
As in the case of 1-body dynamics, the conservation of energy holds, and $\rho_{\varepsilon,\hbar,N:1}$ evolves according to an evolution equation.
\begin{lem}
\label{N body Conservation of energy }Let  $R_{\varepsilon,\hbar,N}(t)$ be as in {\rm Theorem
\ref{3rd main result intro}}. Then

\smallskip
\begin{enumerate}
    \item [\rm (i)] $\frac{\dd}{\dd t}\mathcal{F}_{\varepsilon,\hbar,N}(t)=0${\rm ,}

    \medskip
    \item[\rm (ii)] $\partial_{t}\rho_{\varepsilon,\hbar,N:1}+\partial_{x}J_{\varepsilon,\hbar,N:1}=0.$
\end{enumerate}

\end{lem}
The proof is omitted, since it is similar to the previous considerations.
In variance with the quasi-neutral limit for VP, the quasi-neutral limit
for VPME necessitates an $L^{\infty}$--bound on $e^{V_{\varepsilon,X_N}},$
which is not uniformly bounded in $\varepsilon$.  
Therefore, in order to derive the mean-field limit, the
vanishing rate of $\varepsilon=\varepsilon(N)$ with respect to $N$ is slower 
than in the usual quasi-neutral mean-field limit. This bound is included in
the following lemma{\rm :}

\begin{lem}[\cite{griffin2023stability}, Proposition 3.2]
\label{ nonlinear interaction} 
Let $(\tilde{V}_{\varepsilon,X_{N}},\hat{V}_{\varepsilon,X_{N}})$ be the solution of the system{\rm :}
\begin{equation}
\begin{cases}
-\varepsilon \tilde{V}_{\varepsilon,X_{N}}''=\mu_{X_{N}}-1,\\[0.5mm]
-\varepsilon \hat{V}_{\varepsilon,X_{N}}''=1-e^{\tilde{V}_{\varepsilon,X_{N}}+\hat{V}_{\varepsilon,X_{N}}}\label{eq:-30}
\end{cases}
\end{equation}
ensured by {\rm Lemma \ref{existence uniqueness 1D}}.
Then $V_{\varepsilon,X_{N}}=\tilde{V}_{\varepsilon,X_{N}}+\hat{V}_{\varepsilon,X_{N}}$ satisfies the estimate:
\begin{align*}
    \left\Vert V_{\varepsilon,X_{N}}\right\Vert_{\infty}\leq \frac{1}{\varepsilon}.
\end{align*}
\end{lem}
We also make use of the following lemma that provides the stability with respect to the perturbations by measures:
\begin{lem}[\cite{griffin2023stability}, Proposition 3.2]
Let $h_{i}\in \mathcal{P}(\mathbb{T})\,$ {\rm(}$i=1,2${\rm)} be probability measures. 
Let $(\tilde{V}_{i}, \hat{V}_{i})\in (H^1(\mathbb{T}))^2$ be the solution of the system:
\begin{align*}
\begin{cases}
-\varepsilon \tilde{V}''_{i}=h_{i}-1,\\[0.5mm]
-\varepsilon \hat{V}''_{i}= 1-e^{\tilde{V}_{i}+\hat{V}_{i}}.
\end{cases}
\end{align*}
Then
\begin{align*}
\| \tilde{V}'_{1}-\tilde{V}'_{2}\|_{2}
+ 4 \sqrt{\varepsilon}\, \|\hat{V}'_{1}-\hat{V}'_{2}\|_{2}
\leq   \frac{1}{\varepsilon}W_{1}(h_{1},h_{2}),
\end{align*}
where $W_1$ denotes the $1$-Wasserstein distance.
\label{1d stability for measure}
\end{lem}
An additional significant ingredient in the renormalization argument is the
following asymptotic positivity, coercivity inequality, and commutator estimates, all of which are fundamental discoveries due to \cite{duerinckx2020mean}. We state these inequalities only for the 1-D periodic case, although the results hold in greater generality.  Denote
\begin{align*}
 \mathscr{E}(X_{N},\mu)\coloneqq  \int_{\mathbb{T}\times\mathbb{T}}K(x-y)(\mu_{X_{N}}-\mu)^{\otimes2}(\dd x\dd y)
\end{align*}
so that
\begin{align*}
\mathscr{V}_{\varepsilon,X_{N}}=\frac{N}{2\varepsilon}\mathscr{E}(X_{N},m_{\varepsilon,X_{N}}).
\end{align*}

\begin{lem}[\cite{duerinckx2020mean}, Corollary 3.5] \label{duerinckslemma}
 Let $\mu\in L^{\infty}(\mathbb{T})\cap \mathcal{P}(\mathbb{T})$ and  ${X}_{N}\in \mathbb{T}^{N}\setminus \triangle_{N}$. 
 Then

 \smallskip
 \begin{enumerate}
     \item [\rm (i)]  $\mathscr{E}(X_{N},\mu) +\frac{1+\left\Vert \mu \right\Vert_{\infty}}{N^{2}}\geq 0${\rm ;}

     \smallskip
     \item [\rm (ii)] There are some $\lambda,C>0$ such that, for all $\varphi\in W^{1,\infty}(\mathbb{T}^{d}),$ 
\[
\bigg\vert \int_{\mathbb{T}}\left(\mu_{X_{N}}-\mu\right)\varphi(x)\,\dd x \bigg\vert 
\leq C\left\Vert \nabla\varphi\right\Vert _{\infty}N^{-\lambda}+\left\Vert \nabla\varphi\right\Vert _{2}\Big(\mathscr{E}(X_{N},\mu)+\frac{1+\left\Vert \mu \right\Vert_{\infty}}{N^{2}}\Big)^{\frac{1}{2}}.
\]
 \end{enumerate}
\end{lem}

\begin{lem} [\cite{duerinckx2020mean}, Proposition 1.1] \label{commutator estimate} 
Let the assumptions of {\rm Lemma \ref{duerinckslemma}} hold, and assume further that $u:\mathbb{T}\rightarrow \mathbb{T}$ is Lipschitz. Then
 \begin{align*}
\bigg\vert \int_{\mathbb{T}\times\mathbb{T}\setminus\triangle}\left(u(x)-u(y)\right) K'(x-y)\left(\mu_{N}-\mu\right)^{\otimes2}(\dd x\dd y)\bigg\vert 
\leq C\Big(\mathscr{E}(X_{N},\mu)+\frac{1+\left\Vert \mu \right\Vert_{\infty}}{N^{2}}\Big),
\end{align*}
where $C=C\left(\left\Vert u \right\Vert_{W^{1,\infty}}\right)$.
\end{lem}
The significance of the above-mentioned results for what concerns the classical problem of deriving a Vlasov-like equation 
as a mean-field limit is beyond the scope of this paper. We refer to \cite{serfaty2024lectures} (and especially Chapter 2) 
for an exhaustive discussion. The following simple lemma concerns the Lipschitz continuity of  $V_{\varepsilon,X_{N}}$ with respect to the configuration $X_{N}$ and related properties.
\begin{lem}\label{Lip in configuration}
Let $V_{\varepsilon,X_{N}}$ and $m_{\varepsilon,X_{N}}$ be the solutions guaranteed by {\rm Lemma \ref{existence uniqueness 1D}} to the equation
\begin{align}
-\varepsilon V_{\varepsilon,X_{N}}''=\mu_{X_{N}}-m_{\varepsilon,X_{N}}
\qquad\mbox{with $\,\,m_{\varepsilon,X_{N}}=e^{V_{\varepsilon,X_{N}}}$}.
\end{align}
Then the following statements hold{\rm :}
\begin{enumerate}
    \item[\rm (i)] The functions $X_{N}\mapsto V_{\varepsilon,X_{N}}(x)$ and $X_{N}\mapsto m_{\varepsilon,X_{N}}(x)$ are Lipschitz continuous uniformly in $x$. Moreover, for each $1\leq j \leq N$,
\begin{align*}
 \left\vert \partial_{x_{j}}V_{\varepsilon,X_{N}}(x)\right\vert\leq \frac{1}{\varepsilon^{\frac{3}{2}}N}\qquad
 \text{ for all $x\in \mathbb{T}$}.
\end{align*}
    \item[\rm (ii)] For any fixed $X_{N}$, the functions: $x\mapsto V_{X_{N}}(x)$ and $x\mapsto m_{X_{N}}(x)$ are Lipschitz.
\end{enumerate}
\end{lem}

\begin{proof} This can be seen as follows:

\smallskip
(i). Let
$X_{N}=(x_{1}^{0},\cdots,x_{j},\cdots,x_{N}
^{0})$ and $Y_{N}=(x_{1}^{0},\cdots,y_{j},\cdots,x_{N}
^{0})$,
where $x_{k}^{0}$ is fixed for any $k \neq j$. Thanks to the Sobolev embedding, we have
\begin{align}
\left|V_{\varepsilon,X_{N}}(x)-V_{\varepsilon,Y_{N}}(x)\right|&\leq\| \tilde{V}_{\varepsilon,X_{N}}-\tilde{V}_{\varepsilon,Y_{N}}\| _{\infty}+\| \hat{V}_{\varepsilon,X_{N}}-\hat{V}_{\varepsilon,Y_{N}}\|_{\infty} \notag\\
&\leq\| \tilde{V}_{\varepsilon,X_{N}}-\tilde{V}_{\varepsilon,Y_{N}}\| _{\infty}+\| \hat{V}_{\varepsilon,X_{N}}-\hat{V}_{\varepsilon,Y_{N}}\|_{H^{1}}. \label{stepi ine}
\end{align}
Clearly, $x_{j}\mapsto \tilde{V}_{\varepsilon,X_{N}}(x)$ is Lipschitz with the estimate: 
$$
| \tilde{V}_{\varepsilon,X_{N}}(x)-\tilde{V}_{\varepsilon,Y_{N}}(x)|\leq \frac{1}{\varepsilon N}| x_{j}-y_{j}|.
$$ 
Moreover, by  Lemma  \ref{1d stability for measure}, we have the estimate:
\begin{align*}
\| \hat{V}_{\varepsilon,X_{N}}-\hat{V}_{\varepsilon,Y_{N}}\|_{H^{1}}\leq \varepsilon^{-\frac{3}{2}}W_{1}(\mu_{X_{N}},\mu_{Y_{N}})\leq  \frac{1}{\varepsilon^{\frac{3}{2}}N}| x_{j}-y_{j} |.
\end{align*}
Together with inequality \eqref{stepi ine}, we deduce that
\begin{align*}
\left|V_{\varepsilon,X_{N}}(x)-V_{\varepsilon,Y_{N}}(x)\right|\leq \frac{2}{\varepsilon^{\frac{3}{2}}N}\left\vert x_{j}-y_{j} \right\vert \qquad \text{for all $x\in \mathbb{T}$}.
\end{align*}
Consequently, $X_{N}\mapsto m_{X_{N}}(x)$ is also Lipschitz uniformly in $x$. 

\smallskip
(ii). This is an immediate consequence of Lemma \ref{existence uniqueness 1D}, the fact that $x\mapsto \tilde{V}_{X_{N}}$ is Lipschitz uniformly in $X_{N}$, and the decomposition $V_{X_{N}}=\tilde{V}_{X_{N}}+\hat{V}_{X_{N}}$.
\end{proof}
We now renormalize the argument given in Section 2 in order
to establish the quantum mean-field limit.

\bigskip
\textit{Proof of Theorem \ref{3rd main result intro}}.
We divide the proof into five steps:

\medskip
\textbf{1}. Thanks to the conservation of total energy (Lemma \ref{N body Conservation of energy }),
we have
\begin{align*}
\frac{\dd}{\dd t}\mathcal{E}_{\varepsilon,\hbar,N}(t)
&=\frac{1}{2N}\frac{\dd}{\dd t}\mathrm{tr}\big(\sum^{N}_{j=1}i\hbar\partial_{x_{j}}\vee u(t,x_{j})R_{\varepsilon,\hbar,N}(t)\big)
+\frac{1}{2}\frac{\dd}{\dd t}\int_{\mathbb{T}}\left|u\right|^{2}(t,x)\,\rho_{\varepsilon,\hbar,N:1}(t,x)\,\dd x\\
&\quad\,\,+\frac{\dd}{\dd t}\int_{\mathbb{T}^{N}}\int_{\mathbb{T}}m_{\varepsilon,X_{N}}(x)\log(1/\rho(t,x))\,\dd x\,\rho_{\varepsilon,\hbar,N}(t,\dd X_{N})+\int_{\mathbb{T}}\partial_{t}\rho(t,x)\,\dd x\nonumber\\
&=:\sum_{j=1}^{4}I^{j}(t).
\end{align*}

\textbf{2}. In this step, we establish the following identity (compare with \eqref{eq:-2}):
\begin{align}
&\frac{1}{2N}\frac{\dd}{\dd t}\mathrm{tr}\big(\sum^N_{j=1}i\hbar\partial_{x_{j}}\vee u(t,x_{j})R_{\varepsilon,\hbar,N}(t)\big)
+\frac{1}{2}\frac{\dd}{\dd t}\int_{\mathbb{T}^{d}}\left|u\right|^{2}(t,x)\rho_{\varepsilon,\hbar,N:1}(t,x)\,\dd x
\notag\\
&=\frac{1}{2}\mathrm{tr}\big((i\hbar\partial_{x}+u)\vee(\partial_{t}u+u\partial_{x} u)R_{\varepsilon,\hbar,N:1}(t)\big)\notag\\
&\quad\,\,-\frac{1}{4}\mathrm{tr}\big((i\hbar\partial_{x}+u)\vee\big((i\hbar\partial_{x}+u)\vee u\partial_{x}u\big)R_{\varepsilon,\hbar,N:1}(t)\big)\notag\\
&\quad\,\,+\frac{1}{N}\sum^N_{j=1}\int_{\mathbb{T}^{N}}\partial_{x_{j}}(\mathscr{I}_{\varepsilon,X_{N}}+\mathscr{V}_{\varepsilon,X_{N}})u(t,x_{j})\rho_{\varepsilon,\hbar,N}(t,X_{N}) \,\dd X_{N}.\label{eq:-25}
\end{align}
We proceed in a manner similar to Step 2 in the proof of Theorem \ref{second main thm intro}. 
To make the equation lighter, we omit the dependency on time whenever there is no ambiguity.
{\small
\begin{align*}
&\frac{\dd}{\dd t}\underset{j}{\sum}\mathrm{tr}\big(i\hbar\partial_{x_{j}}\vee u(x_{j})R_{\varepsilon,\hbar,N}\big)+\frac{\dd}{\dd t}\underset{j}{\sum}\mathrm{tr}\big(\left|u\right|^{2}(x_{j})R_{\varepsilon,\hbar,N}\big)\\
&=\underset{j}{\sum}\mathrm{tr}\big((i\hbar\partial_{x_{j}}\vee u(x_{j})+|u|^{2}(x_{j}))\partial_{t}R_{\varepsilon,\hbar,N}\big)
+\underset{j}{\sum}\mathrm{tr}\big(\partial_{t}(i\hbar\partial_{x_{j}}\vee u(x_{j})+|u|^{2}(x_{j}))R_{\varepsilon,\hbar,N}\big)\\
&=\frac{1}{i\hbar}\underset{j}{\sum}\mathrm{tr}\big((i\hbar\partial_{x_{j}}\vee u(x_{j})+|u|^{2}(x_{j}))[\mathscr{H}_{\varepsilon,\hbar,N},R_{\varepsilon,\hbar,N}]\big)
+\underset{j}{\sum}\mathrm{tr}\big(\partial_{t}(i\hbar\partial_{x_{j}}\vee u(x_{j})+|u|^{2}(x_{j}))R_{\varepsilon,\hbar,N}\big)\\
&=\frac{i}{\hbar}\mathrm{tr}\big([\mathscr{H}_{\varepsilon,\hbar,N},\underset{j}{\sum}(i\hbar\partial_{x_{j}}\vee u(x_{j})+|u|^{2}(x_{j}))]R_{\varepsilon,\hbar,N}\big)
+\mathrm{tr}\big(\partial_{t}(\underset{j}{\sum}(i\hbar\partial_{x_{j}}\vee u(x_{j})+|u|^{2}(x_{j})))R_{\varepsilon,\hbar,N}\big)\\
&=\mathrm{tr}\big((\partial_{t}+\frac{i}{\hbar}[\mathscr{H}_{\varepsilon,\hbar,N},\cdot\,])(\underset{j}{\sum}(i\hbar\partial_{x_{j}}\vee u(x_{j})+\left|u\right|^{2}(x_{j})))R_{\varepsilon,\hbar,N}\big)\\
&=\mathrm{tr}\big((\partial_{t}+\frac{i}{\hbar}[\mathscr{K}_{N,\hbar},\cdot\,])(\underset{j}{\sum}(i\hbar\partial_{x_{j}}\vee u(x_{j})+\frac{1}{2}u(x_{j})\vee u(x_{j})))R_{\varepsilon,\hbar,N}\big)\\
&\quad\,\,+\mathrm{tr}\big((\frac{i}{\hbar}[\mathscr{I}_{\varepsilon,X_{N}}+\mathscr{V}_{\varepsilon,X_{N}},\cdot\,])(\underset{j}{\sum}(i\hbar\partial_{x_{j}}\vee u(x_{j})+\frac{1}{2}u(x_{j})\vee u(x_{j})))R_{\varepsilon,\hbar,N}\big)\\
&=:\underset{j}{\sum}\mathrm{tr}\big(J_{1,j}R_{\varepsilon,\hbar,N}\big)+\underset{j}{\sum}\mathrm{tr}(J_{2,j}R_{\varepsilon,\hbar,N}).
\end{align*}
}

We now simplify the expression for $J_{1,j}$ and $J_{2,j}$. We write $J_{1,j}$ as
\begin{align}
J_{1,j}&=\big(\partial_{t}+\frac{i}{\hbar}[-\frac{\hbar^{2}}{2}\Delta_{N},\cdot\,]\big)\big((\frac{1}{2}u(x_{j})+i\hbar\partial_{x_{j}})\vee u(x_{j})\big) \notag\\
&=\big((\partial_{t}+\frac{i}{\hbar}[-\frac{\hbar^{2}}{2}\Delta_{N},\cdot\,])(\frac{1}{2}u(x_{j})+i\hbar\partial_{x_{j}})\big)\vee u(x_{j}) \notag\\
&\quad\,\,+\big(\frac{1}{2}u(x_{j})+i\hbar\partial_{x_{j}}\big)\vee\big((\partial_{t}+\frac{i}{\hbar}[-\frac{\hbar^{2}}{2}\Delta_{N},\cdot\,])u(x_{j})\big) \notag\\
&=\big((\partial_{t}+\frac{i}{\hbar}[-\frac{\hbar^{2}}{2}\Delta_{N},\cdot\,])u(x_{j})\big)\vee u(x_{j})+i\hbar\partial_{x_{j}}\vee\big((\partial_{t}+\frac{i}{\hbar}[-\frac{\hbar^{2}}{2}\Delta_{N},\cdot\,])u(x_{j})\big),  \notag
\end{align}
so that
\begin{align}
J_{1,j}&=\big(u(x_{j})+i\hbar\partial_{x_{j}}\big)\vee\big((\partial_{t}+\frac{i}{\hbar}[-\frac{\hbar^{2}}{2}\Delta_{N},\cdot\,])u(x_{j})\big) \notag\\
&=\big(u(x_{j})+i\hbar\partial_{x_{j}}\big)\vee\big(\partial_{t}u(x_{j})+u(x_{j})\partial_{x_{j}}u(x_{j})\big) \notag\\
&\quad\,\,+\big(u(x_{j})+i\hbar\partial_{x_{j}}\big)\vee\big(\frac{i}{\hbar}[-\frac{\hbar^{2}}{2}\Delta_{N},u(x_{j})]-u(x_{j})\partial_{x_{j}}u(x_{j})\big)\label{eq:-23}.
\end{align}
Notice that
\[
\frac{i}{\hbar}[-\frac{\hbar^{2}}{2}\Delta_{N},u(x_{j})]
=-\frac{i\hbar}{2}\partial_{x_{j}}\vee[\partial_{x_{j}},u(x_{j})]
=-\frac{1}{2}i\hbar\partial_{x_{j}}\vee\partial_{x_{j}}u(x_{j}),
\]
which shows that the right-hand side of \eqref{eq:-23} is equal to
\begin{equation*}
\big(u(x_{j})+i\hbar\partial_{x_{j}}\big)\vee\big(\partial_{t}u(x_{j})+u(x_{j})\partial_{x_{j}}u(x_{j})\big)-\frac{1}{2}\big(u(x_{j})+i\hbar\partial_{x_{j}}\big)\vee\big((i\hbar\partial_{x_{j}}+u(x_{j}))\vee\partial_{x_{j}}u(x_{j})\big),\end{equation*}
so that
\begin{align}
\frac{1}{N}\underset{j}{\sum} \mathrm{tr}(J_{1,j}R_{\varepsilon,\hbar,N})
=& \frac{1}{N}\underset{j}{\sum}\mathrm{tr}\big((u(x_{j})+i\hbar\partial_{x_{j}})\vee(\partial_{t}u(x_{j})+u(x_{j})\partial_{x_{j}}u(x_{j}))R_{\varepsilon,\hbar,N}\big) \notag\\
&-\frac{1}{2N}\underset{j}{\sum}\mathrm{tr}\big((u(x_{j})+i\hbar\partial_{x_{j}})\vee((i\hbar\partial_{x_{j}}+u(x_{j}))\vee\partial_{x_{j}}u(x_{j}))R_{\varepsilon,\hbar,N}\big). \label{trace of J1N}
\end{align}
Turning to the calculation of $J_{2,j}$, we observe that
\begin{align*}
\mathrm{tr}(J_{2,j}R_{\varepsilon,\hbar,N})
&=-\mathrm{tr}\big([\mathscr{I}_{\varepsilon,X_{N}}+\mathscr{V}_{\varepsilon,X_{N}},u(t,x_{j})\partial_{x_{j}}]R_{\varepsilon,\hbar,N}\big)\\
&=\int_{\mathbb{T}^{N}}\partial_{x_{j}}(\mathscr{I}_{\varepsilon,X_{N}}+\mathscr{V}_{\varepsilon,X_{N}})u(t,x_{j})\rho_{\varepsilon,\hbar,N}(t,X_{N})\ \dd X_{N},
\end{align*}
so that
\begin{align}
\frac{1}{N}\underset{j}{\sum}\mathrm{tr}\big(J_{2,j}R_{\varepsilon,\hbar,N}\big)
=\frac{1}{N}\sum^{N}_{j=1}\int_{\mathbb{T}^{N}}\partial_{x_{j}}(\mathscr{I}_{\varepsilon,X_{N}}+\mathscr{V}_{\varepsilon,X_{N}})u(t,x_{j})\rho_{\varepsilon,\hbar,N}(t,X_{N})\ \dd X_{N}. \label{interaction part step1N}
\end{align}
The combination of \eqref{trace of J1N} with \eqref{interaction part step1N} yields \eqref{eq:-25}.

\medskip
\textbf{3}. We simplify the last term in \eqref{eq:-25} as
\begin{align}
&\frac{1}{N}\sum^{N}_{j=1}\int_{\mathbb{T}^{N}}\partial_{x_{j}}(\mathscr{I}_{\varepsilon,X_{N}}+\mathscr{V}_{\varepsilon,X_{N}})u(t,x_{j})\rho_{\varepsilon,\hbar,N}(t,X_{N})\ \dd X_{N}\notag\\
&=\frac{1}{2\varepsilon}\sum^N_{j=1}\int_{\mathbb{T}^{N}}u(t,x_{j})\partial_{x_{j}}\Big(\int_{\mathbb{T}\times\mathbb{T}}K(x-y)(\mu_{X_{N}}-m_{\varepsilon,\hbar,X_{N}})^{\otimes2}(\dd x\dd y)\Big)\rho_{\varepsilon,\hbar,N}(t,X_{N})\ \dd X_{N}\notag\\
&\quad\,\,+\sum_{j=1}^{N}\int_{\mathbb{T}^{N}}u(t,x_{j})\partial_{x_{j}}\Big(\int_{\mathbb{T}}m_{\varepsilon,X_{N}}(x)V_{\varepsilon,X_{N}}(x)\,\dd x\Big)\rho_{\varepsilon,\hbar,N}(t,X_{N})\,\dd X_{N}. \label{INTERACTION TERM}
\end{align}
Recall that, by Lemma \ref{Lip in configuration}, $X_{N}\mapsto V_{X_{N}}$ and $X_{N}\mapsto m_{X_{N}}$ are Lipschitz
so that the forthcoming calculations are justified. First, we note that
\begin{align*}
 \partial_{x_{j}}\Big(\frac{1}{N^{2}}\underset{k\neq l}{\sum} K(x_{k}-x_{l})\Big)=\frac{2}{N^{2}}\underset{k:k\neq j}{\sum} K'(x_{k}-x_{j}).
\end{align*}
Therefore, we have
\begin{align}
&\partial_{x_{j}}\left(\int_{\mathbb{T}\times\mathbb{T}}K(x-y)(\mu_{X_{N}}-m_{\varepsilon,X_{N}})^{\otimes2}(\dd x\dd y)\right)\notag\\
&=\frac{2}{N^{2}}\underset{k:k\neq j}{\sum} K'(x_{k}-x_{j})-2\int_{\mathbb{T}}(K\star m_{\varepsilon,X_{N}})(x)\partial_{x_{j}}\mu_{X_{N}}(\dd x)\notag\\
&\quad\,\,-2\int_{\mathbb{T}}(K\star\partial_{x_{j}}m_{\varepsilon,X_{N}})(x)\mu_{X_{N}}(\dd x)
+2\int_{\mathbb{T}}(K\star\partial_{x_{j}}m_{\varepsilon,X_{N}})(x)m_{\varepsilon,X_{N}}(x)\,\dd x\notag\\
&\coloneqq \sum^4_{k=1}L_{k}^{j}.
\end{align}
We start with $L^{j}_{1}+L^{j}_{2}$. Note that
\[
\int_{\mathbb{T}}K\star m_{\varepsilon,X_{N}}\partial_{x_{j}}\mu_{X_{N}}(\dd x)=-\frac{1}{N}\int_{\mathbb{T}}K\star m_{\varepsilon,X_{N}}\delta_{x_{j}}'(\dd x)=\frac{1}{N} (K'\star m_{\varepsilon,X_{N}})(x_{j}),
\]
so that
\begin{align*}
&\frac{1}{\varepsilon}\sum^{N}_{j=1}u(t,x_{j})\Big(\frac{2}{N^{2}}\underset{k\neq j}{\sum^N_{j=1}} K'(x_{k}-x_{j})-2\int_{\mathbb{T}}(K\star m_{\varepsilon,X_{N}})\partial_{x_{j}}\mu_{X_{N}}(\dd x)\Big)\\
&=\frac{2}{\varepsilon}\int_{\mathbb{T}\times \mathbb{T}\setminus \triangle}u(t,x)K'(x-y)(\mu_{X_{N}}-m_{\varepsilon,X_{N}})(\dd y)\mu_{X_{N}}(\dd x).
\end{align*}
We can symmetrize the right-hand side of the last identity to find that it is equal to
\begin{align*}
&2\int_{\mathbb{T}\times \mathbb{T}\setminus \triangle}
u(t,x) K'(x-y)(\mu_{X_{N}}-m_{\varepsilon,X_{N}})(\dd x)(\mu_{X_{N}}-m_{\varepsilon,X_{N}})(\dd y)\\
&\quad\,\,+ 2\int_{\mathbb{T}\times \mathbb{T}\setminus \triangle}u(t,x) K'(x-y)(\mu_{X_{N}}-m_{\varepsilon,X_{N}})(\dd y)m_{\varepsilon,X_{N}}(\dd x)\\
&=\int_{\mathbb{T}\times \mathbb{T}\setminus \triangle}(u(t,x)-u(t,y)) K'(x-y)(\mu_{X_{N}}-m_{\varepsilon,X_{N}})(\dd x)(\mu_{X_{N}}-m_{\varepsilon,X_{N}})(\dd y)\\
&\quad\,\,+2\int_{\mathbb{T}\times \mathbb{T}\setminus \triangle}u(t,x)K'(x-y)(\mu_{X_{N}}-m_{\varepsilon,X_{N}})(\dd y)m_{\varepsilon,X_{N}}(\dd x).
\end{align*}
The second term on the right-hand side of the last identity is recast as
\begin{align*}
&-2\int_{\mathbb{T}} (K'\star (m_{\varepsilon,X_{N}}u))(t,y)(\mu_{X_{N}}-m_{\varepsilon,X_{N}})(\dd y)\\
&=2\varepsilon\int_{\mathbb{T}}(K'\star(m_{\varepsilon,X_{N}}u))(t,y) V_{\varepsilon,X_{N}}''(t,y)\,\dd y\\
&=2\varepsilon\int_{\mathbb{T}}\big((m_{\varepsilon,X_{N}}u)(t,y)-\int_{\mathbb{T}} (m_{\varepsilon,X_{N}}u)(t,x)\,{\rm d}x\big)V_{\varepsilon,X_{N}}'(t,y)\,\dd y\\
&=2\varepsilon\int_{\mathbb{T}}(\partial_{x}m_{\varepsilon,X_{N}}u)(t,x)\,\dd x\\
&=-2\int_{\mathbb{T}}\partial_{x}u(t,x)m_{\varepsilon,X_{N}}(x)\,\dd x.
\end{align*}
Thus, we have proved the identity
\begin{align}
\frac{1}{\varepsilon}\sum^N_{j=1}u(t,x_{j})(L^{j}_{1}+L^{j}_{2})
=&\int_{\mathbb{T}\times \mathbb{T}\setminus \triangle}\big(u(t,x)-u(t,y)\big)K'(x-y)(\mu_{X_{N}}-m_{\varepsilon,X_{N}})^{\otimes 2}(\dd x\dd y) \notag\\
&-2\int_{\mathbb{T}}\partial_{x}u(t,x)m_{\varepsilon,X_{N}}(x)\,\dd x. \label{L1+L2}
\end{align}
We continue with the calculation of $L_{3}^{j}+L_{4}^{j}$. We have
\begin{align*}
\frac{1}{\varepsilon}(L_{3}^{j}+L_{4}^{j})&=-\frac{2}{\varepsilon}\int_{\mathbb{T}}K\star\partial_{x_{j}}m_{\varepsilon,X_{N}}(\mu_{X_{N}}-m_{\varepsilon,X_{N}})(\dd x)\\
&=2\int_{\mathbb{T}}K\star\partial_{x_{j}}m_{\varepsilon,X_{N}} V_{\varepsilon,X_{N}}''(\dd x)\\
&=-2\int_{\mathbb{T}}\partial_{x_{j}}(m_{\varepsilon,X_{N}})V_{\varepsilon,X_{N}}(\dd x)\\
&=-2\partial_{x_{j}}\Big(\int_{\mathbb{T}}m_{\varepsilon,X_{N}}(x)V_{\varepsilon,X_{N}}(x)\,\dd x\Big),
\end{align*}
where the last identity is due to the observation that
\begin{align*}
\int_{\mathbb{T}} m_{\varepsilon,X_{N}}(x)\partial_{x_{j}}V_{\varepsilon,X_{N}}(x)\,\dd x= \partial_{x_{j}}\int_{\mathbb{T}}  m_{\varepsilon,X_{N}}(x)\,\dd x=0.
\end{align*}
Therefore, we have
\begin{align}
\frac{1}{\varepsilon}\sum_{j=1}^{N}u(t,x_{j})(L_{3}^{j}+L_{4}^{j})=-2\sum_{j=1}^{N}u(t,x_{j})\partial_{x_{j}}\Big(\int_{\mathbb{T}}m_{\varepsilon,X_{N}}(x)V_{\varepsilon,X_{N}}(x)\,\dd x\Big). \label{L3+L4}
\end{align}
Substituting \eqref{L1+L2} and \eqref{L3+L4} into \eqref{INTERACTION TERM}, we obtain
\begin{align*}
&\frac{1}{N}\sum^N_{j=1}\int_{\mathbb{T}^{N}}\partial_{x_{j}}(\mathscr{I}_{\varepsilon,X_{N}}+\mathscr{V}_{\varepsilon,X_{N}})u(t,x_{j})\rho_{\varepsilon,\hbar,N}(t,X_{N}) \ \dd X_{N}\\
&= \frac{1}{2\varepsilon}\int_{\mathbb{T}^{N}}\int_{\mathbb{T}\times \mathbb{T}\setminus \triangle}(u(t,x)-u(t,y)) K'(x-y)(\mu_{X_{N}}-m_{\varepsilon,X_{N}})^{\otimes 2}(\dd x\dd y) \,\rho_{\varepsilon,\hbar,N}(t,X_{N})\,\dd X_{N}\\
&\quad\,\,-\int_{\mathbb{T}^{N}}\int_{\mathbb{T}}\partial_{x}u(t,x)m_{\varepsilon,X_{N}}(x)\,\dd x\,\rho_{\varepsilon,\hbar,N}(t,X_{N})\,\dd X_{N},
\end{align*}
so that, in view of \eqref{eq:-25}, 
\begin{align}
&I^{1}(t)+I^{2}(t)\notag\\
&=\frac{1}{2}\mathrm{tr}\big((i\hbar\partial_{x}+u)\vee(\partial_{t}u+u\partial_{x}u)\,R_{\varepsilon,\hbar,N:1}(t)\big)\notag\\
&\quad\,\,-\frac{1}{4}\mathrm{tr}\big((i\hbar\partial_{x}+u)\vee(i\hbar\partial_{x}+u)\vee(\partial_{x} u)\,R_{\varepsilon,\hbar,N:1}(t)\big) \notag\\
&\quad\,\,-\int_{\mathbb{T}^{N}}\int_{\mathbb{T}}\partial_{x}u(t,x)m_{\varepsilon,X_{N}}(x)\,\dd x\,\rho_{\varepsilon,\hbar,N}(t,X_{N})\,\dd X_{N} \notag\\
&\quad\,\,+\frac{1}{2\varepsilon}\int_{\mathbb{T}^{N}}\int_{\mathbb{T}\times \mathbb{T}\setminus \triangle}\big(u(t,x)-u(t,y)\big)
K'(x-y)(\mu_{X_{N}}-m_{\varepsilon,X_{N}})^{\otimes 2}(\dd x\dd y)\,\rho_{\varepsilon,\hbar,N}(t,X_{N})\,\dd X_{N}.
\label{eq:-12}
\end{align}
Next we compute the term
\begin{align*}
\frac{\dd}{\dd t}\int_{\mathbb{T}^{N}}\int_{\mathbb{T}}m_{\varepsilon,X_{N}}(x)\log(1/\rho(t,x))\,\dd x\,\rho_{\varepsilon,\hbar,N}(t,X_{N})\,\dd X_{N}+\int_{\mathbb{T}}\partial_{t}\rho(t,x)\,\dd x.
\end{align*}
We have
\begin{align}
I^{3}(t)+I^{4}(t)=&\int_{\mathbb{T}^{N}}\int_{\mathbb{T}}\frac{\dd}{\dd t}\left(m_{\varepsilon,X_{N}}(x)\rho_{\varepsilon,\hbar,N}(t,X_{N})\right)\log(1/\rho(t,x))\,\dd x\dd X_{N} \notag\\
&-\int_{\mathbb{T}^{N}}\int_{\mathbb{T}}m_{\varepsilon,X_{N}}(x)\rho_{\varepsilon,\hbar,N}(t,X_{N})\frac{\partial_{t}\rho}{\rho}(t,x)\,\dd x\dd X_{N}
+\int_{\mathbb{T}}\partial_{t}\rho(t,x)\,\dd x \notag\\
=&\int_{\mathbb{T}}\Big(1-\frac{\int_{\mathbb{T}^{N}}e^{V_{\varepsilon,X_{N}}}\rho_{\varepsilon,\hbar,N}(t,\dd X_{N})}{\rho(t,x)}\Big)\partial_{t}\rho(t,x)\,\dd x \notag\\
&+\int_{\mathbb{T}}\frac{\dd}{\dd t}\Big(\int_{\mathbb{T}^{N}}\big(\mu_{X_{N}}+\varepsilon V_{\varepsilon,\hbar,X_{N}}''\big)\rho_{\varepsilon,\hbar,N}(t,X_{N})\,\dd X_{N}\Big)\log(1/\rho(t,x))\,\dd x.
\end{align}
Thanks to Lemma \ref{N body Conservation of energy } and noticing that $\int_{\mathbb{T}^{N}}\rho_{\varepsilon,\hbar,N}(t,X_{N})\mu_{X_{N}}(x)\,\dd X_{N}=\rho_{\varepsilon,\hbar,N:1}(t,x)$, we have
\begin{align}
 I^{3}(t)+I^{4}(t)=&\int_{\mathbb{T}}\Big(1-\frac{\int_{\mathbb{T}^{N}}e^{V_{\varepsilon,X_{N}}}\rho_{\varepsilon,\hbar,N}(t,X_{N})\,\dd X_{N}}{\rho(t,x)}\Big)\partial_{t}\rho(t,x)\,\dd x \notag\\
 &+\varepsilon\int_{\mathbb{T}}\frac{\dd}{\dd t}\Big(\int_{\mathbb{T}^{N}}V_{\varepsilon,X_{N}}''\rho_{\varepsilon,\hbar,N}(t,X_{N})\,\dd X_{N}\Big)\log(1/\rho(t,x))\,\dd x \notag\\
&-\int_{\mathbb{T}}\partial_{x}J_{\varepsilon,\hbar,N:1}(t,x)\log(1/\rho(t,x))\,\dd x.\label{eq:-27}
\end{align}
Also, observe that
\begin{equation*}
\begin{split}
&\frac{1}{2}\mathrm{tr}\big((i\hbar\partial_{x}+u)\vee\partial_{x}\log(\rho)\, R_{\varepsilon,\hbar,N:1}(t)\big)\\
&=\int_{\mathbb{T}}u(t,x)\partial_{x}\log(\rho(t,x))\,\rho_{\varepsilon,\hbar,N:1}(t,x)\,\dd x
-\int_{\mathbb{T}}J_{\varepsilon,\hbar,N:1}(t,x)\partial_{x}\log(\rho(t,x))\,\dd x\\
&=\int_{\mathbb{T}}\int_{\mathbb{T}^{N}}e^{V_{\varepsilon,X_{N}}(t,x)}\rho_{\varepsilon,\hbar,N}(t,X_{N})\,\dd X_{N}u(t,x)\partial_{x}\log(\rho(t,x))\,\dd x\\
&\quad\,\,-\varepsilon\int_{\mathbb{T}}\int_{\mathbb{T}^{N}} V_{\varepsilon,X_{N}}''(x)\rho_{\varepsilon,\hbar,N}(t,X_{N})\,\dd X_{N}u(t,x)\partial_{x}\log(\rho(t,x))\,\dd x\nonumber\\
&\quad\,\,-\int_{\mathbb{T}}J_{\varepsilon,\hbar,N:1}(t,x)\partial_{x}\log(\rho(t,x))\,\dd x.\label{eq:-28}
\end{split}
\end{equation*}
Consequently, we have proved
\begin{align*}
&\sum_{j=1}^{4}I^{j}(t)\notag\\
&=\frac{1}{2}\mathrm{tr}\big((i\hbar\partial_{x}+u)\vee(\partial_{t}u+u\partial_{x}u+\partial_{x}\log(\rho))R_{\varepsilon,\hbar,N:1}(t)\big) \notag\\
&\quad\,-\frac{1}{4}\mathrm{tr}\big((i\hbar\partial_{x}+u)\vee((i\hbar\partial_{x}+u)\vee(\partial_{x} u))R_{\varepsilon,\hbar,N:1}(t)\big) \notag\\
&\quad\,-\int_{\mathbb{T}^{N}}\int_{\mathbb{T}}\partial_{x}u(t,x)m_{\varepsilon,X_{N}}(x)\,\dd x\,\rho_{\varepsilon,\hbar,N}(t,X_{N})\,\dd X_{N} \notag\\
&\quad\,+\frac{1}{2\varepsilon}\int_{\mathbb{T}^{N}}\int_{\mathbb{T}\times \mathbb{T}\setminus \triangle}\big(u(t,x)-u(t,y)\big)
K'(x-y)(\mu_{X_{N}}-m_{\varepsilon,X_{N}})^{\otimes 2}(\dd x\dd y)\rho_{\varepsilon,\hbar,N}(t,X_{N})\,\dd X_{N}\\
&\quad\,+\int_{\mathbb{T}}\Big(-\frac{\int_{\mathbb{T}^{N}}e^{V_{\varepsilon,X_{N}}}\rho_{\varepsilon,\hbar,N}(t,X_{N})\,\dd X_{N}}{\rho(t,x)}+1\Big)\partial_{t}\rho(t,x)\,\dd x \notag\\
&\quad\,+\varepsilon\int_{\mathbb{T}}\frac{\dd}{\dd t}\Big(\int_{\mathbb{T}^{N}}V_{\varepsilon,X_{N}}''\rho_{\varepsilon,\hbar,N}(t,X_{N})\,\dd X_{N}\Big)\log(1/\rho(t,x))\,\dd x \notag\\
&\quad\,-\int_{\mathbb{T}}\int_{\mathbb{T}^{N}}e^{V_{\varepsilon,X_{N}}(t,x)}\rho_{\varepsilon,\hbar,N}(t,X_{N})\,\dd X_{N}u(t,x)\partial_{x}\log(\rho(t,x))\,\dd x\\
&\quad\,+\varepsilon\int_{\mathbb{T}}\int_{\mathbb{T}^{N}} V_{\varepsilon,X_{N}}''(x)\rho_{\varepsilon,\hbar,N}(t,X_{N})\,\dd X_{N}u(t,x)\partial_{x}\log(\rho(t,x))\,\dd x\nonumber
\end{align*}
By the equation, we have 
\begin{align*}
\frac{1}{2}\mathrm{tr}\big((i\hbar\partial_{x}+u)\vee(\partial_{t}u+u\partial_{x}u+\partial_{x}\log(\rho))\,R_{\varepsilon,\hbar,N:1}(t)\big)=0,    
\end{align*}
so that
\begin{align*}
&\sum_{j=1}^{4}I^{j}(t)\notag\\
&=\varepsilon\int_{\mathbb{T}}\int_{\mathbb{T}^{N}} V_{\varepsilon,X_{N}}''(x)\rho_{\varepsilon,\hbar,N}(t,X_{N})\,\dd X_{N}u(t,x)\partial_{x}\log(\rho(t,x))\,\dd x\\
&\quad\,+\varepsilon\int_{\mathbb{T}}\frac{\dd}{\dd t}\Big(\int_{\mathbb{T}^{N}}V_{\varepsilon,X_{N}}''\rho_{\varepsilon,\hbar,N}(t,X_{N})\,\dd X_{N}\Big)\log(1/\rho(t,x))\,\dd x\\
&\quad\,-\frac{1}{4}\mathrm{tr}\Big(\big(i\hbar\partial_{x}+u\big)\vee\big((i\hbar\partial_{x}+u)\vee(\partial_{x} u)\big)R_{\varepsilon,\hbar,N:1}(t)\Big)\\
&\quad\,+\frac{1}{2\varepsilon}\int_{\mathbb{T}^{N}}\int_{\mathbb{T}\times \mathbb{T}\setminus \triangle}\big(u(t,x)-u(t,y)\big)
K'(x-y)(\mu_{X_{N}}-m_{\varepsilon,X_{N}})^{\otimes 2}(\dd x\dd y)\rho_{\varepsilon,\hbar,N}(t,X_{N})\,\dd X_{N}\\
&\quad\,-\int_{\mathbb{T}^{N}}\int_{\mathbb{T}}\partial_{x}u(t,x)m_{\varepsilon,X_{N}}(x)\,\dd x\,\rho_{\varepsilon,\hbar,N}(t,X_{N})\,\dd X_{N}\\
&\quad\,+\int_{\mathbb{T}}\Big(-\frac{\int_{\mathbb{T}^{N}}e^{V_{\varepsilon,X_{N}}}\rho_{\varepsilon,\hbar,N}(t,X_{N})\,\dd X_{N}}{\rho(t,x)}+1\Big)\partial_{t}\rho(t,x)\,\dd x\\
&\quad\,-\int_{\mathbb{T}}\int_{\mathbb{T}^{N}}e^{V_{\varepsilon,X_{N}}(t,x)}\rho_{\varepsilon,\hbar,N}(t,X_{N})\,\dd X_{N}u(t,x)\partial_{x}\log(\rho(t,x))\,\dd x. 
\end{align*} 
Furthermore, by the equation: $\partial_{t}\rho+\partial_{x}(\rho u)=0$,  the last three terms in the identity above 
cancel out. Then we finally obtain 
\begin{align*}
&\sum_{j=1}^{4}I^{j}(t)\nonumber\\
&=\varepsilon\int_{\mathbb{T}}\int_{\mathbb{T}^{N}} V_{\varepsilon,X_{N}}''(t,x)u(t,x)\partial_{x}\log\rho(t,x)\,\dd x\rho_{\varepsilon,\hbar,N}(t,X_{N})\,\dd X_{N}\\
&\quad+\varepsilon\int_{\mathbb{T}}\frac{\dd}{\dd t}\Big(\int_{\mathbb{T}^{N}} V_{\varepsilon,\hbar,X_{N}}''(t,x)\rho_{\varepsilon,\hbar,N}(t,X_{N})\,\dd X_{N}\Big)\log(1/\rho(t,x))\,\dd x\\
&\quad-\frac{1}{4}\mathrm{tr}\big((i\hbar\partial_{x}+u)\vee(i\hbar\partial_{x}+u)\vee(\partial_{x} u))R_{\varepsilon,\hbar,N:1}(t)\big)\\
&\quad+\frac{1}{2\varepsilon}\int_{\mathbb{T}^{N}}\int_{\mathbb{T}\times \mathbb{T}\setminus \triangle}\big(u(t,x)-u(t,y)\big)K'(x-y)(\mu_{X_{N}}-m_{\varepsilon,X_{N}})^{\otimes 2}(\dd x\dd y)\,\rho_{\varepsilon,\hbar,N}(t,X_{N})\,\dd X_{N}\\
&=:\sum^4_{k=1}J_{k}.
\end{align*}

\medskip
\textbf{4}. By Lemma \ref{duerinckslemma}(ii), for $\lambda>0,$, we have
\begin{align*}
\left\vert J_{1}\right\vert&\leq\int_{\mathbb{T}^{N}}\bigg|\int_{\mathbb{T}}(\mu_{X_{N}}-m_{\varepsilon,N}(x))u(t,x)\partial_{x}\log(\rho(t,x))\,\dd x\bigg|\rho_{\varepsilon,\hbar,N}(t,X_{N})\,\dd X_{N}\\
&\leq C\left\Vert u\partial_{x}\log(\rho)\right\Vert _{L^\infty_{t}W^{1,\infty}_{x}}N^{-\lambda}\\
&\quad\,+C\left\Vert u\partial_{x}\log(\rho)\right\Vert _{L_{t}^{\infty}\dot{H}^{1}_{x}}\int_{\mathbb{T}^{N}}\Big(\mathscr{E}(X_{N},m_{\varepsilon,X_{N}})+\frac{1+\left\Vert m_{\varepsilon,X_{N}}\right\Vert _{\infty}}{N^{2}}\Big)^{\frac{1}{2}}\rho_{\varepsilon,\hbar,N}(t,X_{N})\,\dd X_{N}.
\end{align*}
Thanks to Lemma \ref{ nonlinear interaction}, we obtain the inequality:
\begin{align*}
&\Big(\mathscr{E}(X_{N},m_{\varepsilon,X_{N}})+\frac{1+\left\Vert m_{\varepsilon,X_{N}}\right\Vert _{\infty}}{N^{2}}\Big)^{\frac{1}{2}}\\
&\leq \varepsilon+\frac{1}{\varepsilon}\int_{\mathbb{T}\times\mathbb{T}}K(x-y)(\mu_{X_{N}}-m_{\varepsilon})^{\otimes2}(\dd x\dd y)+\frac{1+\left\Vert m_{\varepsilon,X_{N}}\right\Vert _{\infty}}{\varepsilon N^{2}}\\
&\leq \varepsilon+\frac{1}{\varepsilon}\int_{\mathbb{T}\times\mathbb{T}}K(x-y)(\mu_{X_{N}}-m_{\varepsilon})^{\otimes2}(\dd x\dd y)+\frac{1+e^{\frac{1}{\varepsilon}}}{\varepsilon N^{2}}.
\end{align*}
Therefore, it follows that
\begin{equation}\label{eq:-13}
\left\vert J_{1} \right\vert
\leq C \left\Vert u\partial_{x}\log(\rho)\right\Vert_{L^\infty_{t}W^{1,\infty}_{x}}\Big(\varepsilon+
 N^{-\lambda}+\mathfrak{\mathcal{\mathcal{V}}}_{\varepsilon,\hbar,N}(t)
 +\frac{2e^{\frac{1}{\varepsilon}}}{\varepsilon N^{2}}\Big).
\end{equation}
Next, note that
\begin{align*}
\int_{0}^{t}J_{2}(s)\,\dd s
&=\varepsilon\int_{0}^{t}\int_{\mathbb{T}}\partial_{s}\left(\int_{\mathbb{T}^{N}} V_{\varepsilon,X_{N}}''(x)\rho_{\varepsilon,\hbar,N}(s, X_{N})\dd X_{N}\right)\log\left(\rho(s,x)\right)\,\dd x \dd s\\
&=-\varepsilon\int_{0}^{t}\int_{\mathbb{T}^{N}}\int_{\mathbb{T}} V_{\varepsilon,X_{N}}''(x)\partial_{s}\log\left(\rho(s,x)\right)\,\dd x\, 
\rho_{\varepsilon,\hbar,N}(s,X_{N})\,\dd X_{N}\dd s\\
&\quad\,\,+\varepsilon\int_{\mathbb{T}^{N}}\int_{\mathbb{T}} \log\left(\rho(0,x)\right)V_{\varepsilon,X_{N}}''(x)\dd x\,\rho_{\varepsilon,\hbar,N}(0,X_{N})\,\dd X_{N}\\
&\quad\,\,-\varepsilon\int_{\mathbb{T}^{N}}\int_{\mathbb{T}}\log\left(\rho(t,x)\right)   V_{\varepsilon,X_{N}}''(x)\dd x\,\rho_{\varepsilon,\hbar,N}(t,X_{N})\,\dd X_{N}.
\end{align*}
Once again, in view of Lemma  \ref{duerinckslemma}(ii)  and Lemma \ref{ nonlinear interaction} for each $\varphi\in W^{1,\infty}(\mathbb{T}^{d})$, we have the estimate:
\begin{align}
&\left\vert\varepsilon\int_{\mathbb{T}^{N}}\int_{\mathbb{T}} V_{\varepsilon,X_{N}}''(x)\varphi(x)\,\dd x\,\rho_{\varepsilon,\hbar,N}(s,X_{N})\,\dd X_{N}\right\vert\nonumber\\
&\leq \varepsilon\int_{\mathbb{T}^{N}}\left\vert\int_{\mathbb{T}} V_{\varepsilon,X_{N}}''(x)\varphi(x)\,\dd x\right\vert\rho_{\varepsilon,\hbar,N}(s,X_{N})\,\dd X_{N}\notag\\
&\leq C\left\Vert \varphi\right\Vert _{W^{1,\infty}}N^{-\lambda}+\left\Vert \varphi\right\Vert _{H^{1}}\int_{\mathbb{T}^{N}}\Big(\mathscr{E}(X_{N},m_{\varepsilon,X_{N}})+\frac{2e^{\frac{1}{\varepsilon}}}{N^{2}}\Big)^{\frac{1}{2}}\rho_{\varepsilon,\hbar,N}(s,X_{N})\,\dd X_{N} \notag\\
&\leq  C\left\Vert \varphi\right\Vert_{W^{1,\infty}}N^{-\lambda}+\sqrt{\varepsilon}\left\Vert \varphi\right\Vert_{H^1}^{2}+\sqrt{\varepsilon}\mathcal{V}_{\varepsilon,\hbar,N}(s)+\frac{2e^{\frac{1}{\varepsilon}}}{N^{2}\sqrt{\varepsilon}} \notag\\
&\leq C\left\Vert \varphi\right\Vert_{W^{1,\infty}}N^{-\lambda}+\sqrt{\varepsilon}\left\Vert \varphi\right\Vert_{H^1}^{2}+\sqrt{\varepsilon}\mathcal{F}_{0}+\frac{2e^{\frac{1}{\varepsilon}}}{N^{2}\sqrt{\varepsilon}}
, \label{1st est for j2}
\end{align}
where the last inequality is due to the conservation of energy (Lemma \ref{N body Conservation of energy }) and the assumption that $\mathcal{F}_{\varepsilon,\hbar,N}(0)\leq \mathcal{F}_{0}$. 
Utilizing \eqref{1st est for j2} with $\varphi=\partial_{s}\log(\rho(s,x))$, 
$\log(\rho(0,x))$, and $\log(\rho(t,x))$, respectively, we see that
\begin{align}
\int_{0}^{t}J_{2}(s)\,\dd s\leq
 C\Big(N^{-\lambda}+\sqrt{\varepsilon}+\frac{e^\frac{1}{\varepsilon}}{\sqrt{\varepsilon}N^{2}}\Big),\label{estimate on J2}
\end{align}
where $C=C(\left\Vert \log(\rho)\right\Vert_{W^{1,\infty}_{t,x}},\mathcal{F}_{0},T)$.
As for $J_{3}$, we have
\begin{equation}
\left|J_{3}\right|\leq\Big|\mathrm{tr}\big((i\hbar\partial_{x}+u)\vee((i\hbar\partial_{x}+u)\vee(\partial_{x} u))R_{\varepsilon,\hbar,N:1}(t)\big)\Big|\leq2\left\Vert u\right\Vert _{L^{\infty}_{t}W^{1,\infty}_{x}}\mathcal{K}_{\varepsilon,\hbar,N}(t).\label{eq:-14}
\end{equation}
Finally, applying Lemmas \ref{ nonlinear interaction} and \ref{commutator estimate}, 
we have
\begin{equation}
\left|J_{4}\right|\leq C\Big(\mathcal{V}_{\varepsilon,\hbar,N}(t)+\frac{1+e^{\frac{1}{\varepsilon}}}{N^2\varepsilon}\Big),\label{est J4}
\end{equation}
where $C=C(\left\Vert u\right\Vert_{L^{\infty}_{t}W^{1,\infty}_{x}})$.
Thus, gathering inequalities \eqref{eq:-13} and \eqref{estimate on J2}--\eqref{est J4},
we find
\begin{align*}
\mathcal{E}_{\varepsilon,\hbar,N}(t)\leq\mathcal{E}_{\varepsilon,\hbar,N}(0)+C\left(\int_{0}^{t}\mathcal{E}_{\varepsilon,\hbar,N}(s)\,\dd s+N^{-\lambda}+\sqrt{\varepsilon}+\frac{e^{\frac{1}{\varepsilon}}}{\varepsilon N^{2}}\right),
\end{align*}
where $C=C\big(T,\left\Vert u\right\Vert _{L^{\infty}_{t}W^{1,\infty}_{x}},\left\Vert \log\left(\rho\right)\right\Vert _{W_{t,x}^{1,\infty}},\mathcal{F}_{0}\big)$.
By using the Gr\"onwall inequality and Lemma \ref{duerinckslemma}(i), it follows that
\begin{align}
\mathcal{E}_{\varepsilon,\hbar,N}(t)\leq e^{Ct}\Big(\mathcal{E}_{\varepsilon,\hbar,N}(0)+N^{-\lambda}+\sqrt{\varepsilon}+\frac{e^{\frac{1}{\varepsilon}}}{\varepsilon N^{2}}\Big).\label{gronwall for E}
\end{align}

\medskip
\textbf{5}. Put $\Tilde{m}_{\varepsilon,\hbar,N}(t,x)\coloneqq \int_{\mathbb{T}^{N}}m_{\varepsilon,X_{N}}(t,x)\rho_{\varepsilon,\hbar,N}(t,X_{N})\,\dd X_{N}$. 
According to Lemma \ref{ nonlinear interaction}, we have 
$$
\left\Vert m_{\varepsilon,X_N}\right\Vert_{\infty}\leq e^\frac{1}{\varepsilon}
$$ 
so that, in view of Lemma \ref{duerinckslemma}(ii), we find
\begin{align}
\underset{t \in [0,T]}{\sup}W_{1}(\rho_{\varepsilon,\hbar,N:1}(t,\cdot),\Tilde{m}_{\varepsilon,\hbar,N}(t,\cdot))\leq \frac{C}{N^{\lambda}}+\Big(\underset{t\in [0,T]}{\sup}\mathcal{E}_{\varepsilon,\hbar,N}(t)+\frac{Ce^{\frac{1}{\varepsilon}}}{\varepsilon N^{2}}\Big)^{\frac{1}{2}}. \label{wasserstein est}\end{align}
In addition, thanks to Lemma \ref{CZK}, we see that 
\begin{align*}
\int_{\mathbb{T}}\left\vert \Tilde{m}_{\varepsilon,\hbar,N}-\rho \right\vert(t,x)\,\dd x&\leq  
\int_{\mathbb{T}^{N}}\int_{\mathbb{T}} \left\vert m_{\varepsilon,\hbar,N}-\rho\right\vert(t,x)\,\dd x \rho_{\varepsilon,\hbar,N}(t,X_{N})\,\dd X_{N}\\
&\leq 
\int_{\mathbb{T}^{N}} \sqrt{2\int_{\mathbb{T}} m_{\varepsilon,X_{N}}(t,x)\log (\frac{m_{\varepsilon,X_{N}}(t,x)}{\rho(t,x)})\dd x}\rho_{\varepsilon,\hbar,N}(t,X_{N}) \,\dd X_{N}\\
&\leq \sqrt{2\int_{\mathbb{T}^{N}} \int_{\mathbb{T}} m_{\varepsilon,X_{N}}(t,x)\log (\frac{m_{\varepsilon,X_{N}}(t,x)}{\rho(t,x)})\dd x\rho_{\varepsilon,\hbar,N}(t,X_{N}) \,\dd X_{N}},
\end{align*}
where the last inequality is by Cauchy-Schwarz and the fact that $\rho_{\varepsilon,\hbar,N}$ is a probability density. Thus, it follows from Lemma \ref{duerinckslemma}(i) that
\begin{align}
\underset{t \in [0,T]}{\sup}\left\Vert (\tilde{m}_{\varepsilon,\hbar,N}-\rho)(t,\cdot)\right\Vert_{1}\leq C\Big(\underset{t \in [0,T]}{\sup}\mathcal{E}_{\varepsilon,\hbar,N}(t)+\frac{1+e^{\frac{1}{\varepsilon}}}{\varepsilon N^{2}}\Big)^{\frac{1}{2}} \label{L1 est}
\end{align}
for some constant $C>0$ independent of $(\varepsilon, \hbar, N)$.
Hence, thanks to the assumption on $\varepsilon=\varepsilon(N)$,  it follows from inequalities \eqref{gronwall for E}--\eqref{L1 est} that
\begin{align*}
\underset{t \in [0,T]}{\sup}W_{1}(\rho_{\varepsilon,\hbar,N:1}(t,\cdot),\rho(t,\cdot))\underset{\varepsilon+\hbar+\frac{1}{N}\rightarrow 0}{\longrightarrow}0.
\end{align*}
The convergence:
$$
J_{\varepsilon,\hbar,N:1}(t,\cdot)\underset{\varepsilon+\hbar+\frac{1}{N}\rightarrow 0}{\longrightarrow}(\rho u)(t,\cdot)
$$ 
is deduced by the same argument presented in Step 5 in the proof of Theorem \ref{1body main intro}. 
This completes the proof of Theorem \ref{3rd main result intro}.
\begin{flushright}
$\square$
\par\end{flushright}

\section{Well-Prepared Initial Data} \label{well prepared sec}
We start by constructing well-prepared initial data for the 1-body problem, {\it i.e.,} the initial data such that
$\mathcal{E}_{\varepsilon,\hbar}(0)\rightarrow0$ as $\varepsilon+\hbar\rightarrow0$.
Our construction of the initial data is summarized in the following lemma: 

\begin{lem}
Let $\rho_{0} \in H^{s}(\mathbb{T}^{d})\cap \mathcal{P}(\mathbb{T}^{d})$, $\rho_0>0$, 
and $u_{0}=\nabla U_{0}\in H^{s}(\mathbb{T}^{d})$ for some $s>1$ sufficiently large. 
Let $V_{0}=\log(\rho_{0})$. Set
\begin{align}
\psi^{\mathrm{in}}_{\varepsilon,\hbar}\coloneqq \sqrt{e^{V_{0}}-\varepsilon\Delta V_{0}}e^{\frac{iU_{0}(x)}{\hbar}}.
\end{align}
Then, for this choice of $\psi^{\mathrm{in}}_{\varepsilon,\hbar}$,
\begin{align*}
\mathcal{E}_{\varepsilon,\hbar}(0)\underset{\varepsilon+\hbar\rightarrow 0}{\longrightarrow}0.
\end{align*}
\label{initial hartree}
\end{lem}

\begin{proof}
Clearly, for $\varepsilon$ sufficiently small, $e^{V_{0}}-\varepsilon \Delta V_{0}$ positive 
and $\Delta V_{0}$ has mean $0$, so that $\vert \psi^{\mathrm{in}}_{\varepsilon,\hbar}\vert^{2}$ 
is a probability density. Moreover, it follows from the construction that the Poisson-Boltzman equation 
\begin{align*}
-\varepsilon \Delta V_{0}=\left\vert \psi^{\rm in}_{\varepsilon,\hbar}\right\vert^{2}-e^{V_{0}}     
\end{align*} 
is satisfied.  
To show that $\mathcal{E}_{\varepsilon,\hbar}(0)\underset{\varepsilon +\hbar \rightarrow 0}{\rightarrow}0 $, 
we start by controlling the kinetic part. 
Put $\varrho_{\varepsilon}\coloneqq \sqrt{e^{V_{0}}-\varepsilon \Delta V_{0}}$ so that  
\begin{align*}
\mathcal{K}_{\varepsilon,\hbar}(0)=\int_{\mathbb{T}^{d}} \left\vert (i\hbar\nabla+u_{0})\psi_{\varepsilon,\hbar}^{\mathrm{in}}\right\vert^{2}(x)\,\dd x =&
 \int_{\mathbb{T}^{d}} \left\vert \big(i\hbar\nabla+u_{0}(x)\big)\varrho_{\varepsilon}(x)e^{\frac{iU_{0}(x)}{\hbar}}\right\vert^{2}\,\dd x.
  \label{kinetic part}
\end{align*}
We expand the right-hand side of the last identity. 
First, notice that
\begin{align*}
&\hbar^{2}\int_{\mathbb{T}^{d}} \Big\vert \nabla \big(\varrho_{\varepsilon}e^{\frac{iU_{0}(x)}{\hbar}}\big) \Big\vert^{2}\,\dd x\\
&\quad= \int_{\mathbb{T}^{d}} \left\vert \varrho_{\varepsilon}\nabla U_{0} \right\vert^{2}(x)\,\dd x+2\hbar^{2}\Re\int_{\mathbb{T}^{d}}\varrho_{\varepsilon}(x)\nabla e^{\frac{iU_{0}(x)}{\hbar}}\ \dd x +\hbar^{2}\int_{\mathbb{T}^{d}}\left\vert \nabla \varrho_{\varepsilon}\right\vert^{2}(x)\,\dd x,\\
&\int_{\mathbb{T}^{d}} \Big\vert u_{0}(x) \varrho_{\varepsilon}e^{\frac{iU_{0}(x)}{\hbar}} \Big\vert^{2}\,\dd x= \int_{\mathbb{T}^{d}} \left\vert \varrho_{\varepsilon}u_{0}  \right\vert^{2}(x)\,\dd x,    
\end{align*}
and
\begin{align*}
&i\hbar\int_{\mathbb{T}^{d}} u_{0}(x)  \varrho_{\varepsilon}(x)e^{-\frac{iU_{0}(x)}{\hbar}} \nabla (\varrho_{\varepsilon}(x)e^{\frac{iU_{0}(x)}{\hbar}})\,\dd x-i\hbar\int_{\mathbb{T}^{d}} \varrho_{\varepsilon}(x)e^{\frac{iU_{0}(x
)}{\hbar}}u_{0}(x)  \nabla (\varrho_{\varepsilon}(x)e^{-\frac{iU_{0}(x)}{\hbar}})\,\dd x\\
&=i\hbar\int_{\mathbb{T}^{d}} u_{0}(x)  \varrho_{\varepsilon}^{2}(x)e^{-\frac{iU_{0}(x)}{\hbar}} \nabla (e^{\frac{iU_{0}(x)}{\hbar}})\,\dd x-i\hbar\int_{\mathbb{T}^{d}} \varrho_{\varepsilon}^{2}(x)e^{\frac{iU_{0}(x
)}{\hbar}}u_{0}(x)  \nabla (e^{-\frac{iU_{0}(x)}{\hbar}})\,\dd x
\\
&=-2\int_{\mathbb{T}^{d}} \varrho_{\varepsilon}^{2}(x)\nabla U_{0}(x)\cdot u_{0}(x)\,\dd x.
\end{align*}
Thus, expanding the square yields
\begin{equation}\label{4.2a}
\mathcal{K}_{\varepsilon,\hbar}(0)
=\int_{\mathbb{T}^{d}} \left\vert \varrho_{\varepsilon}(\nabla U_{0}-u_{0})\right\vert^{2}(x)\,\dd x
+2\hbar^{2}\Re\int_{\mathbb{T}^{d}}\varrho_{\varepsilon}(x)\nabla e^{\frac{iU_{0}(x)}{\hbar}}\, \dd x
+\hbar^{2}\int_{\mathbb{T}^{d}}\left\vert \nabla \varrho_{\varepsilon}(x)\right\vert^{2}\,\dd x. 
\end{equation}
Since $u_{0}=\nabla U_{0}$, the first term in \eqref{4.2a} vanishes identically. 
Clearly, $\left\Vert \nabla \varrho_{\varepsilon}\right\Vert_{2}$ is uniformly bounded in $\varepsilon$
so that the second and third terms in \eqref{4.2a} are bounded by 
\begin{align*}
2\hbar\left\Vert \varrho_{\varepsilon}\right\Vert_{1}\left\Vert \nabla U_{0}\right\Vert_{\infty}+\hbar^{2} \left\Vert \nabla \varrho_{\varepsilon}\right\Vert_{2}^{2}\underset{\varepsilon+\hbar\rightarrow0}{\longrightarrow}0. 
\end{align*}
To conclude, this implies that 
$$
\mathcal{K}_{\varepsilon,\hbar}(0)\underset{\varepsilon+\hbar\rightarrow0}{\longrightarrow}0.
$$
Finally, it is direct that the last term in \eqref{Modulated energy intro} vanishes identically and the second term $\mathcal{V}_{\varepsilon,\hbar}(0)$ satisfies 
\begin{align*}
\mathcal{V}_{\varepsilon,\hbar}(0)=\frac{\varepsilon}{2}\int_{\mathbb{T}^{d}}\left\vert \nabla V_{0}(x)\right\vert^{2} \ \dd x\underset{\varepsilon\rightarrow0}{\longrightarrow}0.     
\end{align*}
Therefore, we conclude that  
$$
\mathcal{E}_{\varepsilon, \hbar}(0)\underset{\varepsilon+\hbar\rightarrow 0}{\rightarrow}0.
$$
\end{proof}

\medskip
The construction of $R^{\mathrm{in}}_{\varepsilon,\hbar,N}$ proceeds as follows:
\begin{lem}
Let $\rho_{0} \in H^{s}(\mathbb{T})\cap \mathcal{P}(\mathbb{T})$, $\rho_0>0,$ 
and $u_{0}= U_{0}'\in H^{s}(\mathbb{T})$ for some $s>1$ sufficiently large. 
Let $V_0=\log (\rho_{0})$.
Let $\psi^{\mathrm{in}}_{\varepsilon,\hbar}\coloneqq \sqrt{e^{V_{0}}-\varepsilon\Delta V_{0}}e^{\frac{iU_{0}}{\hbar}}$ and
$R^{\mathrm{in}}_{\varepsilon,\hbar}\coloneqq \big|\psi_{\varepsilon,\hbar}^{\mathrm{in}}\left\rangle \right\langle \psi_{\varepsilon,\hbar}^{\mathrm{in}}\big|$.  
Set $R^{\mathrm{in}}_{\varepsilon,\hbar,N}={R_{\varepsilon,\hbar}^{\mathrm{in}^{\otimes N}}}$. 
Then there exist some $\Lambda>0$ such that
\begin{align*}
\mathcal{E}_{\varepsilon,\hbar,N}(0)\underset{\varepsilon+\hbar+\frac{1}{N}\rightarrow0}{\longrightarrow}0
\end{align*}
provided that $\varepsilon=\varepsilon(N)$ is chosen such that $\frac{1}{\varepsilon^{2}N^{\Lambda}}\underset{N \rightarrow \infty}{\rightarrow} 0$.
\end{lem}

\begin{proof}
Note that
\begin{align}
    \mathcal{K}_{\varepsilon,\hbar,N}(0)
    =\frac{1}{2N}\sum^N_{j=1}\mathrm{tr}\big((i\hbar\partial_{x_{j}}+u_{0}(x_{j}))^{2}R^{\mathrm{in}}_{\varepsilon,\hbar,N}\big)
    =\frac{1}{2}\mathrm{tr}\big((i\hbar\partial_{x}+u_{0})^{2}R^{\mathrm{in}}_{\varepsilon,\hbar}\big).\label{step1N}
\end{align}
The same considerations demonstrated in Lemma \ref{initial hartree} show that the right-hand side of \eqref{step1N} tends to $0$ as $\varepsilon+\hbar+\frac{1}{N}\rightarrow 0$. Note also that $\rho_{\varepsilon,\hbar,N}(0,X_{N}) \equiv \big(\vert \psi_{\varepsilon,\hbar}^{\mathrm{in}}\vert^{2}\big)^{\otimes N}\coloneqq\rho_{\varepsilon}^{\otimes N}$.
To show that the entropy part vanishes asymptotically, it follows from Lemma \ref{1d stability for measure} that 
\begin{align*}
\int_{\mathbb{T}}m_{\varepsilon,X_{N}}(x)\log\big(\frac{m_{\varepsilon,X_{N}}(x)}{\rho_{0}(x)}\big)\,\dd x&=\int_{\mathbb{T}} m_{\varepsilon,X_{N}}(x)(V_{\varepsilon,X_{N}}-V_{0})(x)\,\dd x\leq \left\Vert V_{\varepsilon,X_{N}}-V_{0}\right\Vert_{\infty}\\
&\leq \left\Vert V_{\varepsilon,X_{N}}-V_{0}\right\Vert_{H^{1}}\leq \frac{5}{4\varepsilon^{\frac{3}{2}}}W_{1}(\mu_{X_{N}},\rho_{\varepsilon}) .     
\end{align*}
Thus, owing to Remark \ref{large deviation rem} below, we see that 
\begin{align}
& \int_{\mathbb{T}^{N}}\int_{\mathbb{T}} m_{\varepsilon,X_{N}}(x)\log\big(\frac{m_{\varepsilon,X_{N}}(x)}{\rho_{0}(x)}\big)\,\dd x\rho_{N}(0,X_{N})\,\dd X_{N} \notag\\
&\leq \frac{5}{4\varepsilon^{\frac{3}{2}}}\int_{\mathbb{T}^{N}}W_{1}(\mu_{X_{N}},\rho_{\varepsilon})\rho_{\varepsilon}^{\otimes N}(X_{N})\,\dd X_{N}\leq \frac{5}{4N^{\Lambda}\varepsilon^{\frac{3}{2}}}.  \label{entropypart}   
\end{align}
\\
Next, we treat the interaction part $\mathcal{V}_{\varepsilon,\hbar,N}(0)$. 
Given a configuration $X_{N}\in \mathbb{T}^{N}$ and $\eta>0$, consider the truncated empirical measure
\begin{align*}
\mu_{X_{N}}^{(\eta)}=\frac{1}{N}\sum^N_{i=1}\delta_{x_{i}}^{(\eta)},
\end{align*}
where $\delta_{x}^{(\eta)}$ designates the $\eta$-truncation of $\delta_{x}$.
\footnote{For instance, we can take $\delta_{x}^{(\eta)}=\chi_{\eta}\star \delta_{x}$, where $\chi$ is a smooth probability density supported on $B(0,\frac{1}{4})\setminus B(0,\frac{3}{16})$ and $\chi_{\eta}(y)\coloneqq \frac{1}{\eta}\chi(\frac{y}{\eta})$.}
Denote by $V_{\varepsilon,X_{N},\eta}$ the solution of
\begin{align*}
    -\varepsilon V_{\varepsilon,X_{N},\eta}''= \mu^{(\eta)}_{X_{N}}-e^{V_{\varepsilon,X_{N},\eta}}.
\end{align*}
According to Lemma \ref{1d stability for measure}, we have
\begin{align*}
\varepsilon \| \hat{V}_{\varepsilon,X_{N},\eta}'-\hat{V}_{0}'\|_{2}^2
&\leq \frac{1}{4\varepsilon^{2}}  W_{1}^{2}(\mu^{(\eta)}_{X_{N}},\rho_{\varepsilon})\\
&\leq \frac{1}{2\varepsilon^{2}}\left( W_{1}^{2}(\mu_{X_{N}},\rho_{\varepsilon})+W_{1}^{2}(\mu_{X_{N}}^{(\eta)},\mu_{X_{N}})\right)\\
&\leq\frac{1}{2\varepsilon^{2}}\left(W_{1}^{2}(\mu_{X_{N}},\rho_{\varepsilon})+\eta^{2}\right) ,
\end{align*}
where the last inequality is due to the general estimate
$W_{1}(\chi_{\eta}\star \mu,\mu)\leq \eta$ (see {\it e.g.} Lemma 7.1 in \cite{griffin2020singular}).
From the same considerations, we have
\begin{align*}
 \varepsilon \|  \tilde{V}_{\varepsilon,X_{N},\eta}'-\tilde{V}_{0}'\|_{2}^2
 \leq \frac{2}{\varepsilon}W_{1}^{2}(\mu^{(\eta)}_{X_{N}},\rho_{\varepsilon})\leq \frac{2}{\varepsilon}\left(W_{1}^{2}(\mu_{X_{N}},\rho_{\varepsilon})+\eta^{2}\right) ,
\end{align*}
which implies that
\begin{align*}
 \varepsilon||  V_{\varepsilon,X_{N},\eta}'-V_{0}'||_{2}^{2}\leq \frac{5}{2\varepsilon^{2}} \left(W_{1}^{2}(\mu_{X_{N}},\rho_{\varepsilon})+\eta^{2}\right).
\end{align*}
Integrating with respect to $\rho_{\varepsilon,\hbar,N}(0,X_{N})$, we have
\begin{align*}
\varepsilon\int_{\mathbb{T}^N}\|  V_{\varepsilon,X_{N},\eta}'-V_{0}'\|_{2}^{2}\,\rho_{\varepsilon,\hbar,N}(0,\dd X_{N})\leq \frac{5}{2\varepsilon^{2}} \left(\int_{\mathbb{T}^N} W_{1}^{2}(\mu_{X_{N}},\rho_{\varepsilon})\,\rho_{\varepsilon,\hbar,N}(0,\dd X_{N})+\eta^{2}\right).
\end{align*}
Again, we invoke Remark \ref{large deviation rem} below to find that there is some $\Lambda>0$ such that
\begin{align}
\int_{\mathbb{T}^N} W_{1}^{2}(\mu_{X_{N}},\rho_{\varepsilon})\,\rho_{\varepsilon,\hbar,N}(0, \dd X_{N})=O(\frac{1}{N^{\Lambda}}).\label{ld ine}
\end{align}
If $\varepsilon=\varepsilon(N)$ is chosen such that $\frac{1}{\varepsilon^{2}N^{\Lambda}}\rightarrow 0$ as $N\rightarrow \infty,$ and $\eta=\eta(\varepsilon)$ such that $\frac{\eta(\varepsilon)}{\varepsilon}\underset{\varepsilon \rightarrow0}{\rightarrow}0$, we conclude that
\begin{align*}
\varepsilon\int_{\mathbb{T}^N}\|  V_{\varepsilon,X_{N},\eta}'-V'_{0}\|_{2}^{2}\,\rho_{\varepsilon,\hbar,N}(0,\,\dd X_{N})\underset{\varepsilon+\hbar+\frac{1}{N}\rightarrow0}{\longrightarrow}0,
\end{align*}
and hence 
\begin{align*}
\varepsilon\int_{\mathbb{T}^N}||  V_{\varepsilon,X_{N},\eta}'||_{2}^{2}\,\rho_{\varepsilon,\hbar,N}(0,\,\dd X_{N})\underset{\varepsilon+\hbar+\frac{1}{N}\rightarrow0}{\longrightarrow}0.    
\end{align*}
Therefore, for this choice of $\varepsilon$, we have
\begin{align*}
\mathcal{V}_{\varepsilon,\hbar,N}(0)=\frac{1}{N}\int_{\mathbb{T}^{N}} \mathscr{V}_{\varepsilon,X_{N}}\,\rho_{\varepsilon,\hbar,N}(0,\,\dd X_{N})\underset{\varepsilon+\hbar+\frac{1}{N}\rightarrow0}{\longrightarrow}0,
\end{align*}
because of the relation (see formula 1.29 in \cite{duerinckx2020mean}):
\begin{align*}
\underset{\eta \rightarrow 0}{\lim}\,\varepsilon\int_{\mathbb{T}^{N}}\left\Vert  V_{\varepsilon,X_{N},\eta}'\right\Vert_{2}^2\,\rho_{\varepsilon,\hbar,N}(0,\,\dd X_{N})
=\frac{1}{N}\int_{\mathbb{T}^{N}} \mathscr{V}_{\varepsilon,X_{N}}\,
\rho_{\varepsilon,\hbar,N}(0,\,\dd X_{N}).
\end{align*}
\end{proof}

\begin{rem}
\label{large deviation rem} The large-deviation-type inequality used 
in \eqref{entropypart}--\eqref{ld ine} and its various extensions appear, for instance, in {\rm \cite{fournier2015rate}}.
Here is a short way to deduce this estimate from {\rm Lemma \ref{duerinckslemma}(ii)} 
for the specific choice $\rho_{0}\equiv 1$. 
Indeed, integrating the inequality in {\rm Lemma \ref{duerinckslemma}(ii)} with respect to $X_{N}$, we obtain 
that, for some constant $\lambda,C>0$,
\begin{align}
\int_{\mathbb{T}^{N}}W_{1}^{2}(\mu_{X_{N}},1)\,\dd X_{N}\leq C\Big(N^{-2\lambda}+\int_{\mathbb{T}^{N}}\mathcal{E}(X_{N},1)\,\dd X_{N}+\frac{2}{N^{2}} \Big). \label{coerc reminder}
\end{align}
A direct calculation reveals that
\begin{align*}
\mathcal{E}(X_{N},1)=\frac{1}{N^{2}}\sum_{1\leq i,j\leq N}K(x_{i}-x_{j})-\int_{\mathbb{T}}K(x)\,\dd x.
\end{align*}
Furthermore, utilizing that $K(0)=0$ and $K$ is even, we have
\begin{align*}
\frac{1}{N^{2}}\sum_{1\leq i,j\leq N}\int_{\mathbb{T}^{N}}K(x_{i}-x_{j})\,\dd X_{N}=\frac{N-1}{N}\int_{\mathbb{T}}K(x)\,\dd x,
\end{align*}
which shows that
\begin{align}
\left\vert\int_{\mathbb{T}^{N}}\mathcal{E}(X_{N},1)\,\dd X_{N}\right\vert=\frac{\left\vert \int_{\mathbb{T}}K(x)\,\dd x\right\vert }{N}\leq \frac{1}{N}.    \label{est for int XN}
\end{align}
Substituting  \eqref{est for int XN} into \eqref{coerc reminder} yields
\begin{align*}
\int_{\mathbb{T}^{N}}W_{1}^2(\mu_{X_{N}},1)\,\dd X_{N}=O(\frac{1}{N^{\Lambda}})
\end{align*}
with $\Lambda=\min\left\{2\lambda,1\right\}$.
\end{rem}
\bigskip

\section{Well-Posedness Theory}
\label{Well-posedness theory}
This section is qualitative
in nature, and we take $\varepsilon=\hbar=1,$ for simplicity. We start with the von Neumann equation \eqref{eq:von N intro}. Recall Kato's pertubation theory, which is the main
ingredient of studying the existence theory of linear equations of Schr{\"o}dinger-type.

\begin{thm}[\cite{teschl2014mathematical}, Theorem 6.4]
\label{Kato's theorem }
Let $T,D(T)\subset\mathfrak{H}$ be a {\rm(}essentially{\rm)}
self-adjoint operator and $S,D(S)\subset\mathfrak{H}$ a symmetric
operator such that $D(T)\subset D(S)$. 
Suppose that there exist $0<a<1$ and $b>0$
such that, for each $\varphi\in D(T),$
\begin{equation}
\|S\varphi\|^{2}\leq a\|T\varphi\|^{2}+b\|\varphi\|^{2}.\label{eq:-25-1}
\end{equation}
Then $T+S$ and $D(T+S)=D(T)$ are {\rm(}essentially{\rm)} self-adjoint. In the case
when $T$ is essentially self-adjoint, $D(\overline{T})\subset D(\overline{S})$
and $\overline{T+S}=\overline{T}+\overline{S}$, 
where $\overline{T}$ and $\overline{S}$ stand for the closures of $T$ and $S$ respectively.
\end{thm}

The combination of Theorem \ref{Kato's theorem } and Lemma \ref{existence uniqueness 1D} 
yields the following
conclusion (with the notation $\mathscr{H}_{N}=\mathscr{H}_{1,1,N}$):
\begin{lem}\label{Kato thm }
Let $X_{N}\in\mathbb{T}^{N}$, and let $(\tilde{V}_{X_{N}},\hat{V}_{X_{N}})$ be the solution to the system{\rm :}
\begin{align*}
\begin{cases}
- \tilde{V}_{X_{N}}''=\mu_{X_{N}}-1,\\[0.5mm]
-\hat{V}_{X_{N}}''=1-e^{\tilde{V}_{X_{N}}+\hat{V}_{X_{N}}},
\end{cases}
\end{align*}
guaranteed by {\rm Lemma \ref{existence uniqueness 1D}}. Then $\mathscr{H}_{N}$
is self-adjoint on $H^{2}(\mathbb{T}^{N})$. 
Consequently, if $R^{\mathrm{in}}_{N}\in\mathcal{D}_{s}(\mathfrak{H}^{\otimes N})$ is such that  
$\mathrm{tr}\left((-\Delta_{N})^{2}R^{\mathrm{in}}_{N}\right)<\infty$, 
then there exists a unique solution of the Cauchy problem
\begin{align}\label{eq:von N sec2}
\begin{cases}
i\partial_{t}R_{N}(t)=\left[\mathscr{H}_{N},R_{N}(t)\right],\\[1mm]
R_{N}(0)=R^\mathrm{in}_{N}.
\end{cases}
\end{align}
\end{lem}

\begin{proof}
Recall the notation:
\begin{align*}
&\mathscr{V}_{X_{N}}\coloneqq \frac{N}{2}\int_{\mathbb{T}\times\mathbb{T}}K(x-y)(\mu_{X_{N}}-m_{X_{N}})^{\otimes2}(\dd x\dd y) \qquad\,\mbox{with $\,\, m_{X_{N}}=e^{V_{X_{N}}}$},\\
&\mathscr{I}_{X_{N}}=N\int_{\mathbb{T}}V_{X_{N}}(x)m_{X_{N}}(x)\,\dd x.
\end{align*}
Also, recall that $\| \tilde{V}_{X_{N}} \|_{\infty}\leq 1$. Therefore, by Lemma \ref{existence uniqueness 1D}, 
we see that $\left\Vert m_{X_{N}} \right\Vert_{\infty}\leq C$
for some effective constant $C>0$ so that
\begin{align*}
\left|\mathscr{V}_{X_{N}}\right|
&\leq \frac{N}{2}\left|\int_{\mathbb{T}} K\star(\mu_{X_{N}}-m_{X_{N}})(\mu_{X_{N}}-m_{X_{N}})(\dd x)\right|\\[1mm]
&\leq N\big(C+\left\Vert K\star m_{X_{N}}\right\Vert_{\infty} \left\Vert m_{X_{N}}\right\Vert_{1}\big)\\[1mm]
&\leq
N\big(C+\left\Vert K\right\Vert_{\infty} \left\Vert m_{X_{N}}\right\Vert_{1}^{2}\big)
=N(C+1),
\end{align*}
where $K$ stands for the Green function of the Laplacian on $\mathbb{T}$.
The estimate is uniform in $X_{N}$ so that 
$X_{N}\mapsto\mathscr{V}_{X_{N}}\in L^{\infty}(\mathbb{T}^{N})$. 
From the same considerations, $X_{N}\mapsto \mathscr{I}_{X_{N}}\in L^{\infty}(\mathbb{T}^{N})$.
That $\mathscr{H}_{N}$ is self-adjoint follows from Theorem \ref{Kato's theorem }. 
The existence is now immediate via Stone's theorem, and the uniqueness follows by linearity of the equation.
\end{proof}
Next, we are concerned with the existence and uniqueness theory for
system \eqref{eq:Hartree VPME}. 
As already remarked, the case of the Schr\"odinger-Poisson system \eqref{SP eq} with $-\Delta V=\rho-1$  
is classical; see, for instance, \cite{bove1974existence,ginibre1980class}. 
We plan to first decouple the equations and then apply a fixed-point argument. 
For this to succeed, we need  the following lemma, showing that $\nabla \hat{V}$ 
is stable with respect to $\tilde{V}$ in the $L^{2}$--norm, which is crucial
in order to be able to control properly the terms contributed by the
nonlinearity:
\begin{lem}[\cite{griffin2021global}, Lemma 3.9]
\label{=00007Bhat=00007DV vs =00007Bbar=00007DV}  
Let $d\in\{2,3\}$ and $h_{i}\in L^{\infty}(\mathbb{T}^{d})$
for $i=1,2$. Consider the system{\rm :}
\begin{align}\nonumber
\begin{cases}
-\Delta\tilde{V}_{i}=h_{i}-1,\\[0.5mm]
-\Delta\hat{V}_{i}=1-e^{\tilde{V}_{i}+\hat{V}_{i}}.
\end{cases}
\end{align}
Then 
\[
\| \nabla\hat{V}_{1}-\nabla\hat{V}_{2}\|_{2}^{2}\leq C\| \tilde{V}_{1}-\tilde{V}_{2}\| _{2}^{2},
\]
where 
$C=C\big(c_{d}(\underset{i}{\max}\Vert\tilde{V}_{i}\Vert_{\infty}+\underset{i}{\max}\Vert\hat{V}_{i}\Vert_{\infty})\big)$,
and $c_{d}$ is a dimensional constant.
\end{lem}

As a corollary, we have the following lemma, which will be useful in several instances in the sequel:

\begin{lem}
For each $r\in (0,\frac{1}{4}]$, let $\chi_{r}\coloneqq\frac{1}{r^{d}}\chi(\frac{x}{r})$, where $\chi \geq 0$ is a smooth radially symmetric function on $\mathbb{T}^{d}$ with $\mathrm{supp}(\chi)\subset B(0,1)$ and $\int_{\mathbb{T}^{d}}\chi(x)\, \dd x=1$. 
With the same hypothesis and notation of {\rm Lemma \ref{=00007Bhat=00007DV vs =00007Bbar=00007DV}}, it holds that
\begin{align*}
\left\Vert \chi_{r}\star(V_{1}-V_{2})\right\Vert_{\infty}\leq C\left\Vert h_{1}-h_{2}\right\Vert_{1},
\end{align*}
where $C=C(r,d,\underset{i}{\max}\left\Vert h_{i}\right\Vert_{\infty})$.
\label{Linfty stability}
\end{lem}

\begin{proof}
Notice first that
\begin{align*}
\|\chi_{r}\star(V_{1}-V_{2}) \|_{\infty} \leq \|\chi_{r}\star (\tilde{V}_{1}- \tilde{V}_{2})  \|_{\infty}+
\| \chi_{r}\star(\hat{V}_{1}- \hat{V}_{2})  \|_{\infty}.
\end{align*}
We first have
\begin{align}\label{est for tildeV}
\| \chi_{r}\star(\tilde{V}_{1}- \tilde{V}_{2}) \|_{\infty}\leq \| \chi_{r}\star K\|_{\infty}\| h_{1}-h_{2}\|_{1}\leq C_{r,d}\| h_{1}-h_{2}\|_{1}.
\end{align}
Furthermore, Lemma \ref{existence uniqueness for hat{U}} implies that
\begin{align}
\|\chi_{r}\star(\hat{V}_{1}-\hat{V}_{2})\|_{\infty}^{2}&\leq \| K \|_{2}^{2}\| e^{\hat{V}_{1}+\tilde{V}_{1}}-e^{\hat{V}_{2}+\tilde{V}_{2}}\|_{2}^{2} \notag\\
&\leq \| K \|_{2}^{2}\| e^{\hat{V}_{1}+\tilde{V}_{1}}-e^{\hat{V}_{1}+\tilde{V}_{2}}\|_{2}^{2} 
+\| K \|_{2}^{2}\| e^{\hat{V}_{1}+\tilde{V}_{2}}-e^{\hat{V}_{2}+\tilde{V}_{2}}\|_{2}^{2} \notag\\
&\lesssim \| e^{\hat{V}_{1}}\|_{\infty}^{2} \| e^{\tilde{V}_{1}}-e^{\tilde{V}_{2}}\|_{2}^{2}+\| e^{\tilde{V}_{2}}\|_{\infty}^{2}\| e^{\hat{V}_{1}}-e^{\hat{V}_{2}}\|_{2}^{2} \notag\\
&\lesssim_{d,\underset{i}{\max}\| h_{i}\|_{\infty}} \| e^{\tilde{V}_{1}}-e^{\tilde{V}_{2}}\|_{2}^{2}+\| e^{\hat{V}_{1}}-e^{\hat{V}_{2}}\|_{2}^{2} \notag\\
&\lesssim_{d,\underset{i}{\max}\| h_{i}\|_{\infty}}  \| \tilde{V}_{1}-\tilde{V}_{2}\|_{2}^{2}+\| \hat{V}_{1}-\hat{V}_{2}\|_{2}^{2},
\label{eq:-16b}
\end{align}
where the last inequality in \eqref{eq:-16b} is due to the mean value theorem as applied
to the function: $x\mapsto e^{x}$. 
The first term on the
right-hand side of \eqref{eq:-16b} is
\begin{align*}
\| \tilde{V}_{1}-\tilde{V}_{2}\|_{2}^{2}\leq \|h_{1}-h_{2}\|_{1}^{2}.
\end{align*}
Therefore, by the Poincar\'{e} inequality and Lemma \ref{=00007Bhat=00007DV vs =00007Bbar=00007DV},
we have
\begin{align}
\|\hat{V}_{1}-\hat{V}_{2}\|_{2}^{2}&\lesssim_{d}\| \nabla\hat{V}_{1}-\nabla\hat{V}_{2}\|_{2}^{2} \notag\\
&\leq C(\underset{i}{\max}\|h_{i}\|_{\infty},d)\| \tilde{V}_{1}-\tilde{V}_{2}\|_{2}^{2}\leq  C(\underset{i}{\max}\| h_{i}\|_{\infty},d)\| h_{1}-h_{2}\|_{1}^{2}.\label{est for hatV}
\end{align}
Gathering \eqref{est for tildeV}--\eqref{est for hatV} yields the desired inequality.
\end{proof}

\medskip
We need also to observe that the $L^2$--norm of $e^{V}$ is bounded by means of the $L^2$--norm of $\rho$, 
as encapsulated in the following lemma:
\begin{lem} \label{L^2 bound o e^V}
Let $\rho \in L^{\infty}(\mathbb{T}^{d}),$ and let $V$ be the unique solution of
\begin{align*}
- \Delta V(x)=\rho(x)-e^{V(x)}.
\end{align*}
Then
\begin{align*}
    \Vert e^{V} \Vert_{2}\leq \left\Vert \rho \right\Vert_{2}.
\end{align*}
\end{lem}
\begin{proof} Multiplying the equation by $e^{V}$ and then integrating, we have
\begin{align*}
- \int_{\mathbb{T}^{d}}\Delta V(x)e^{V(x)}\,\dd x=\int_{\mathbb{T}^{d}}\big(\rho(x)e^{V(x)}-e^{2V(x)}\big)\,\dd x.  \end{align*}
Integrating  the left-hand side by parts, we recognize that
\begin{align*}
-\varepsilon \int_{\mathbb{T}^{d}}\Delta V(x)e^{V(x)}\,\dd x= \int_{\mathbb{T}^{d}} \left\vert \nabla V(x)\right\vert^{2}{e^{V(x)}}\,\dd x\geq0.
\end{align*}
Therefore, it follows that
\begin{align*}
\int_{\mathbb{T}^{d}}\rho(x)e^{V(x)}\,\dd x \geq \int_{\mathbb{T}^{d}}e^{2V(x)}\,\dd x.
\end{align*}
Using the Cauchy-Schwartz inequality, we deduce
\begin{align*}
 \Vert e^{V}\Vert_{2}^{2}\leq   \Vert e^{V}\Vert_{2} \Vert \rho\Vert_{2}.
\end{align*}
This completes the proof.
\end{proof}

\medskip
To prove the well-posedness of the Schr\"odinger-Poisson-Boltzmann system, we first study the Cauchy problem of the regularized equation. We recall the following existence and uniqueness result for a Schr\"odinger equation with a time-dependent potential (one can also consult \cite{reed2003methods} for similar results).

\begin{lem}[\cite{fujiwara1979construction}, Theorem 2]
Let $V:[0,T]\times \mathbb{T}^{d}\rightarrow \mathbb{R}$ be a function such that
\begin{enumerate}
    \item [\rm (i)] $V\in C([0,T]; C^{\infty}(\mathbb{T}^{d}))${\rm ,}

\smallskip
    \item[\rm (ii)] $\partial^{\alpha}_{x}V\in L^{\infty}([0,T];C(\mathbb{T}^{d}))$ for all $\alpha \in \mathbb{N}$.
\end{enumerate} 
Let $\psi^{\mathrm{in}}\in C^{\infty}(\mathbb{T}^{d})$. Then the Cauchy problem
\begin{align*}
\begin{cases}
i\partial_{t}\psi(t,x)=-\frac{1}{2}\Delta\psi(t,x)+V(t,x)\psi(t,x),\quad & x\in \mathbb{T}^d,\,t>0,\nonumber\\[0.5mm]
\psi|_{t=0}=\psi^{\mathrm{in}},& x\in \mathbb{T}^d,
\end{cases}
\end{align*}
has a unique solution $\psi \in C^{1}([0,T];C^{\infty}(\mathbb{T}^{d}))$.  \label{well posed linear eq}
\end{lem}

\begin{lem}
Let $\psi^{\mathrm{in}}\in C^{\infty}(\mathbb{T}^{d})$ 
with $\int_{\mathbb{T}^{d}}\left|\psi^{\mathrm{in}}(x)\right|^{2}\ \dd x=1$. 
For each $r\in (0,\frac{1}{4}]$, let $\chi_{r}\coloneqq\frac{1}{r^{d}}\chi(\frac{x}{r})$, 
where $\chi \geq 0$ is a smooth radially symmetric function on $\mathbb{T}^{d}$ with $\mathrm{supp}(\chi)\subset B(0,1)$ and $\int_{\mathbb{T}^{d}}\chi(x)\ \dd x=1$.
Then, for any $T>0$ and $r\in(0,\frac{1}{4}]$ fixed, there exist a unique solution  $\psi_{r}\in C^{1}([0,T]; C^{\infty}(\mathbb{T}^{d}))$ to the regularized Cauchy problem{\rm :}
\begin{equation}\label{r mollified eq}
\begin{cases}
i\partial_{t}\psi_{r}(t,x)=-\frac{1}{2}\Delta\psi_{r}(t,x)+(\chi_{r}\star V_{r})(t,x)\psi_{r}(t,x),& x\in \mathbb{T}^d,\,t>0,
\\[0.5mm]
-\Delta V_{r}(t,x)=(\chi_{r}\star\left|\psi_{r}\right|^{2})(t,x)-e^{V_{r}(t,x)},& x\in \mathbb{T}^d,\,t>0,\\[0.5mm]
\psi_{r}|_{t=0}=\psi^{\mathrm{in}},& x\in \mathbb{T}^d.
\end{cases}
\end{equation}
\label{mollified eq well posed}
\end{lem}
\begin{proof}
 We first regularize the equation in time and study the well-posedness of the resulting equation via a fixed-point argument. Afterwards, we apply a compactness argument with respect to the mollification in time. 
 We divide the proof into three steps:

 \medskip
\textbf{1}. \textit{Well-posedness for the time regularized system}.  Set
\begin{align*}
\mathcal{\mathfrak{X}}\coloneqq\left\{ \psi\in C([0,T];L^{2}(\mathbb{T}^{d}))\,:\,
\left\Vert \psi(t,\cdot)\right\Vert _{2}=1 \text{ for any } t\in[0,T]\right\} .
\end{align*}
Fix a standard mollifier $\zeta_{\delta}$ on $\mathbb{R}$. In this step, we prove that the time regularized system:
\begin{equation}\label{time reg eq}
\begin{cases}
i\partial_{t}\psi_{r,\delta}(t,x)=-\frac{1}{2}\Delta\psi_{r,\delta}(t,x)+(\chi_{r}\star_{x} V
_{r,\delta})(t,x)\psi_{r,\delta}(t,x),\quad& x\in \mathbb{T}^d,\,t>0,\\[0.5mm]
-\Delta V_{r,\delta}(t,x)=(\zeta_{\delta}\star_{t}\chi_{r}\star_{x}\left|\psi_{r,\delta}\right|^{2})(t,x)-e^{V_{r,\delta}(t,x)},& x\in \mathbb{T}^d,\,t>0,\\[0.5mm]
\psi_{r,\delta}|_{t=0}=\psi^{\mathrm{in}},& x\in \mathbb{T}^d
\end{cases}
\end{equation}
is well posed.   For each fixed $\phi\in\mathfrak{X},$ we consider the linear Schr{\"o}dinger
equation with time dependent interaction
\begin{equation}
\begin{cases}
i\partial_{t}\psi_{r, \delta}(t,x)=-\frac{1}{2}\Delta\psi_{r, \delta}(t,x)
+(\chi_{r}\star_{x}V_{r,\delta})(t,x)\psi_{r,\delta}(t,x),\quad& x\in \mathbb{T}^d,\,t>0,\\
-\Delta V_{r,\delta}(t,x)=(\zeta_{\delta}\star_{t}\chi_{r}\star_{x}\left|\phi\right|^{2})(t,x)
  -e^{V_{r,\delta}(t,x)},& x\in \mathbb{T}^d,\,t>0,\\
\psi_{r,\delta}|_{t=0}=\psi^{\mathrm{in}},& x\in \mathbb{T}^d.
\end{cases}
\end{equation}
When there is no ambiguity, we omit $t$ and $x$ from $\star_{t}$ and $\star_{x}$, respectively. 
By Lemma \ref{existence uniqueness for hat{U}}, for each $t\in(0,T]$,
there is a unique solution $V_{r,\delta}(t,\cdot)\in C^{2,\alpha}(\mathbb{T}^{d})$
to the problem
\[
-\Delta V_{r,\delta}\left[\phi\right](t,x)=(\zeta_{\delta}\star \chi_{r}\star\left|\phi\right|^{2})(t,x)-e^{V_{r,\delta}\left[\phi\right](t,x)}.
\]
By Remark \ref{smoothness of V}, $\chi_{r}\star V$ verifies assumption (ii) of Theorem \ref{well posed linear eq}. 
In addition, we notice that $\chi_{r}\star V_{r,\delta}$ is Lipschitz in time due to Lemma \ref{Linfty stability} applied with $h_{1}(t,\cdot)=(\zeta_{\delta}\star\chi_{r}\star\left\vert\psi_{r,\delta}\right\vert^{2})(t,\cdot)$ and $h_{2}(s,\cdot)=(\zeta_{\delta}\star\chi_{r}\star\left\vert\psi_{r,\delta}\right\vert^{2})(s,\cdot)$ for any $t,s \in [0,T]$.
Thus, assumption (i) of Theorem \ref{well posed linear eq} is also satisfied.  
Then Theorem \ref{well posed linear eq} implies the existence and uniqueness of a solution $\Psi_{r,\delta}\in C^{1}\left([0,T];C^{\infty}(\mathbb{T}^{d})\right)$ of the Cauchy problem for the linear equation:
\begin{align}\nonumber
\begin{cases}
i\partial_{t}\psi_{r,\delta}(t,x)=-\frac{1}{2}\Delta\psi_{r,\delta}(t,x)+ (\chi_{r}\star V_{r,\delta}\left[\phi\right])(t,x)\psi_{r,\delta}(t,x),\quad& x\in \mathbb{T}^d,\,t>0,\nonumber\\[0.5mm]
\psi_{r,\delta}|_{t=0}=\psi^{\mathrm{in}},& x\in \mathbb{T}^d.
\end{cases}
\end{align}
We aim to prove that the operator: $\phi \mapsto \Psi_{r,\delta}[\phi]$ is a contraction from $\mathfrak{X}$ to $\mathfrak{X}$ and thereby conclude via the
Banach fixed point theorem. 
Since $\chi_{r}\ast V_{r,\delta}[\phi]$ is real-valued, it is readily checked that the $L^{2}$--norm 
of $\Psi_{r,\delta}[\phi]$ is conserved, {\it i.e.,} $\left\Vert \Psi_{r,\delta}[\phi](t,\cdot)\right\Vert_{2}=1 $, 
so that $\Psi_{r,\delta}[\phi]\in \mathfrak{X}$ indeed. 
Given $\phi_{1},\phi_{2}\in \mathfrak{X}$, set $\Psi_{1}\coloneqq\Psi_{r,\delta}\left[\phi_{1}\right],\Psi_{2}\coloneqq\Psi_{r,\delta}\left[\phi_{2}\right]$,
and $V_{1}\coloneqq V_{r,\delta}\left[ \phi_{1}\right],V_{2}\coloneqq V_{r,\delta}\left[ \phi_{2}\right]$.
We compute
\begin{align*}
&\frac{\dd}{\dd t}\left\Vert \Psi_{1}(t,\cdot)-\Psi_{2}(t,\cdot)\right\Vert _{2}^{2}\nonumber\\
&=2\Re\Big(-i\int_{\mathbb{T}^{d}}\partial_{t}\big(\Psi_{1}(t,x)-\Psi_{2}(t,x)\big)\overline{\big(\Psi_{1}(t,x)-\Psi_{2}(t,x)\big)}\,\dd x\Big)\\
&=\Re\Big(-i\int_{\mathbb{T}^{d}}\Delta\big(\Psi_{2}(t,x)-\Psi_{1}(t,x)\big)\overline{\big(\Psi_{1}(t,x)-\Psi_{2}(t,x)\big)}\,\dd x\Big)
\\&\quad+2\Re\Big(-i\int_{\mathbb{T}^{d}}\big((\chi_{r}\star V_{1})(t,x)\Psi_{1}(t,x)-(\chi_{r}\star V_{2})(t,x)\Psi_{2}(t,x)\big)\overline{\big(\Psi_{1}(t,x)-\Psi_{2}(t,x)\big)}\,\dd x\Big)\\&=: I+J.
\end{align*}
Integrating by parts, we see that

\[
-\int_{\mathbb{T}^{d}}\Delta\big(\Psi_{2}(t,x)-\Psi_{1}(t,x)\big)\overline{\big(\Psi_{1}(t,x)-\Psi_{2}(t,x)\big)}\,\dd x=\int_{\mathbb{T}^{d}}\left|\nabla(\Psi_{1}-\Psi_{2})\right|^{2}(t,x)\,\dd x,
\]
so that

\[
I=\Re\Big(i\int_{\mathbb{T}^{d}}\left|\nabla(\Psi_{1}-\Psi_{2})\right|^{2}(t,x)\,\dd x\Big)=0.
\]
Furthermore, we have
\begin{align}\label{2.82}
\int_{\mathbb{T}^{d}}&\big((\chi_{r}\star V_{1})(t,x)\Psi_{1}(t,x)-(\chi_{r}\star V_{2})(t,x)\Psi_{2}(t,x)\big)\overline{\big(\Psi_{1}(t,x)-\Psi_{2}(t,x)\big)}\,\dd x\nonumber\\
=&\int_{\mathbb{T}^{d}}(\chi_{r}\star V_{1})(t,x)\big(\Psi_{1}(t,x)-\Psi_{2}(t,x)\big)\overline{\big(\Psi_{1}(t,x)-\Psi_{2}(t,x)\big)}\,\dd x\nonumber\\
&+\int_{\mathbb{T}^{d}}\Psi_{2}(t,x)(\chi_{r}\star(V_{1}-V_{2}))(t,x)\,\overline{\big(\Psi_{1}(t,x)-\Psi_{2}(t,x)\big)}\,\dd x\nonumber\\
=&\int_{\mathbb{T}^{d}}(\chi_{r}\star V_{1})(t,x)\left|\Psi_{1}(t,x)-\Psi_{2}(t,x)\right|^{2}\,\dd x\nonumber\\
&+\int_{\mathbb{T}^{d}}\Psi_{2}(t,x)(\chi_{r}\star(V_{1}-V_{2}))(t,x)\,\overline{\big(\Psi_{1}(t,x)-\Psi_{2}(t,x)\big)}\,\dd x.
\end{align}
Clearly, one has
\[
\Re\Big(i\int_{\mathbb{T}^{d}}V_{1}(t,x)\left|\Psi_{1}(t,x)-\Psi_{2}(t,x)\right|^{2}\,\dd x\Big)=0,
\]
so it suffices to deal with the second integral on the right-hand side of \eqref{2.82}. 

Notice that
\begin{align}
&\int_{\mathbb{T}^{d}}\Psi_{2}(t,x)(\chi_{r}\star(V_{1}-V_{2}))(t,x)\overline{\big(\Psi_{1}(t,x)-\Psi_{2}(t,x)\big)}\,\dd x \notag \\
&\leq\frac{1}{2}\int_{\mathbb{T}^{d}}\left|\Psi_{2}\right|^{2}(t,x)\left|\chi_{r}\star(V_{1}-V_{2})\right|^{2}(t,x)\,\dd x+\frac{1}{2}\int_{\mathbb{T}^{d}}\left|\overline{\Psi_{1}(t,x)-\Psi_{2}(t,x)}\right|^{2}\,\dd x.
\label{eq:-17}
\end{align}
In order to control the first term on the right-hand side of (\ref{eq:-17}),
we invoke Lemma \ref{Linfty stability} with $h_{1}=(\zeta_{\delta}\star \chi_{r}\star\left|\phi_{1}\right|^{2})(t,x)$ and $h_{2}=(\zeta_{\delta}\star \chi_{r}\star\left|\phi_{2}\right|^{2})(t,x)$ to see that
\begin{align*}
\left\Vert(\chi_{r}\star (V_{1}-V_{2}))(t,\cdot)\right\Vert^{2}_{\infty}&\leq   C_{r,d}\underset{t \in [0,T]}{\sup}\left\Vert\left\vert \phi_{1}(t,\cdot)-\phi_{2}(t,\cdot)\right\vert (\left\vert  \phi_{1}(t,\cdot)\right\vert +\left\vert  \phi_{2}(t,\cdot)\right\vert)\right\Vert_{1}^{2}\\
&\leq C_{r,d}\underset{t \in [0,T]}{\sup}\left\Vert \phi_{1}(t,\cdot)-\phi_{2}(t,\cdot)\right\Vert_{2}^{2}.
\end{align*}
We have thus proved
\begin{align}
\int_{\mathbb{T}^{d}}\left|\Psi_{2}\right|^{2}(t,x)\left|\chi_{r}\star(V_{1}-V_{2})\right|^{2}(t,x)\,\dd x\leq C_{r,d}\underset{t \in [0,T]}{\sup}\left\Vert \phi_{1}(t,\cdot)-\phi_{2}(t,\cdot)\right\Vert _{2}^{2} \label{first term}
\end{align}
for some constant $C_{r,d}>0.$ Inequalities \eqref{eq:-17} and \eqref{first term} entail
\begin{align*}
\left\Vert \Psi_{1}(t,\cdot)-\Psi_{2}(t,\cdot)\right\Vert _{2}^{2}&\leq C_{r,d}t\underset{\tau \in [0,T]}{\sup}\left\Vert \phi_{1}(\tau,\cdot)-\phi_{2}(\tau,\cdot)\right\Vert _{2}^{2}+\frac{1}{2}\int_{0}^{t}\left\Vert \Psi_{1}(\tau,\cdot)-\Psi_{2}(\tau,\cdot)\right\Vert _{2}^{2}\,\dd \tau\\
&\leq C_{r,d}T\underset{\tau\in[0,T]}{\sup}\left\Vert \phi_{1}(\tau,\cdot)-\phi_{2}(\tau,\cdot)\right\Vert _{2}^{2}+\frac{1}{2}\int_{0}^{t}\left\Vert \Psi_{1}(\tau,\cdot)-\Psi_{2}(\tau,\cdot)\right\Vert _{2}^{2}\,\dd \tau.
\end{align*}
As a result,
\[
\underset{t\in[0,T]}{\sup}\left\Vert \Psi_{1}(t,\cdot)-\Psi_{2}(t,\cdot)\right\Vert _{2}\leq C_{r,d}\sqrt{T}e^{\frac{T}{2}}\underset{t\in[0,T]}{\sup}\left\Vert \phi_{1}(t,\cdot)-\phi_{2}(t,\cdot)\right\Vert _{2},
\]
which shows that, for $T=T_{r,d}$ sufficiently small, 
$\phi \mapsto \Psi_{r,\delta}[\phi]$ is a contraction, thereby ensuring the existence and uniqueness of a fixed point $\psi_{r,\delta}$. 
Note that, by construction of $\Psi$, 
we see that this solution belongs to $C^{1}([0,T];C^{\infty}(\mathbb{T}^{d}))$.
The global existence follows by a standard iteration argument together
with the conservation of the $L^{2}$--norm.

\medskip
\textbf{2}. \textit{Compactness in $\delta$}. In this step, we aim to remove the regularization in time, which will enable us to prove that the system:
\begin{equation}\label{r equation}
\begin{cases}
i\partial_{t}\psi_{r}(t,x)=-\frac{1}{2}\Delta\psi_{r}(t,x)+(\chi_{r}\star V_{r})(t,x)\psi_{r}(t,x),\quad& x\in \mathbb{T}^d,\,t>0,\\[0.5mm]
-\Delta V_{r}(t,x)=(\chi_{r}\star|\psi_{r}|^{2})(t,x)-e^{V_{r}(t,x)},& x\in \mathbb{T}^d,\,t>0,\\[0.5mm]
\psi_{r}(0,x)=\psi^{\mathrm{in}}, & x\in \mathbb{T}^d
\end{cases}
\end{equation}
is well-posed.
Consider the solution $\psi_{r,\delta}$ of \eqref{time reg eq} constructed in Step 1. We obtain that, 
for any $t,s\in[0,T]$,
\begin{align}
 &\| \psi_{r,\delta}(t,\cdot)-\psi_{r,\delta}(s,\cdot)\|_{H^{2}} \notag\\
 &\leq | t-s|\big(\|\Delta \psi_{r,\delta}\|_{L^{\infty}_{t}L^{2}_{x}} 
 +\|\Delta^{2} \psi_{r,\delta}\|_{L^{\infty}_{t}L^{2}_{x}}\notag\\&+ \|(\chi_{r}\star V_{r,\delta})\psi_{r,\delta} \|_{L^{\infty}_{t}L^{2}_{x}}+\|
  \Delta ((\chi_{r}\star V_{r,\delta})\psi_{r,\delta}))\|_{L^{\infty}_{t}L^{2}_{x}}\big).
 \label{deltaH2 compactness}
\end{align}
We start by propagating the Sobolev norms of $\psi_{r,\delta}$ uniformly in $\delta$.
\begin{align*}
&\frac{\dd}{\dd t
}\left\Vert \Delta^{2} \psi_{r,\delta}(t,\cdot)\right\Vert_{2}^{2}\\
&=2\Re\Big(\int_{\mathbb{T}^{d}}\Delta^{2} \partial_{t}\psi_{r,\delta}(t,x)\Delta^{2} \overline{\psi_{r,\delta}}(t,x) \ \dd x\Big)\\
&=
-2\Re\Big(i\int_{\mathbb{T}^{d}} \Delta^{2}\big(-\frac{1}{2}\Delta \psi_{r,\delta}(t,x)
+ (\chi_{r}\star V_{r,\delta})(t,x)\psi_{r,\delta}(t,x)\big)\Delta^{2}\overline{\psi_{r,\delta}}(t,x) \ \dd x\Big).
\end{align*}
Integration by parts reveals
\begin{align*}
\Re\Big(i\int_{\mathbb{T}^{d}}\Delta^{2}\Delta \psi_{r,\delta}(t,x)\Delta^{2}\overline{\psi_{r,\delta}}(t,x)\ \dd x\Big)=\Re\Big(i\int_{\mathbb{T}^{d}}\left\vert\nabla\Delta^{2}\psi_{r,\delta}\right\vert^{2}(t,x) \ \dd x\Big)=0,
\end{align*}
hence
\begin{align*}
\frac{\dd}{\dd t
}\left\Vert \Delta^{2} \psi_{r,\delta}(t,\cdot)\right\Vert_{2}^{2}=2\Re \Big(-i\int_{\mathbb{T}^{d}}\Delta^{2}\big((\chi_{r}\star
V_{r,\delta})\psi_{r,\delta}\big)(t,x)\Delta^{2}\overline{\psi_{r,\delta}}(t,x) \ \dd x  \Big).
\end{align*}
We can write
\begin{align*}
\Delta^{2}\big((\chi_{r}\star V_{r,\delta})\psi_{r,\delta}\big)=\sum_{k=0}^{4}c_{k}\nabla^{k}(\chi_{r}\star V_{r,\delta})\nabla^{4-k}\psi_{r,\delta} 
\end{align*}
for some constants $c_{k}\in \mathbb{N}$, so that
\begin{align}
\frac{\dd}{\dd t
}\left\Vert \Delta^{2} \psi_{r,\delta}(t,\cdot)\right\Vert_{2}^{2}\lesssim \sum_{k=0}^{4} \int_{\mathbb{T}^{d}} \left\vert \nabla^{k}(\chi_{r}\star V_{r,\delta})\nabla^{4-k}\psi_{r,\delta}\Delta^{2}\psi_{r,\delta}\right\vert(t,x) \ \dd x \coloneqq \sum_{k=0}^{4}I_{k}. \label{sum of Ik}\end{align}
We proceed by estimating each one of $I_{k}$:
\begin{align*}
I_{0}&=\int_{\mathbb{T}^{d}} 
\big\vert (\chi_{r}\star V_{r,\delta})(t,x)\big\vert
\,\big\vert \Delta^{2} \psi_{r,\delta} (t,x)\big\vert^{2} \ \dd x
\leq \big\Vert V_{r,\delta}\big\Vert_{L^\infty_{t,x}}\| \Delta^{2}\psi_{r,\delta}(t,\cdot)\|_{2}^{2},\\
I_{1}&=\int_{\mathbb{T}^{d}} \big\vert \nabla (\chi_{r}\star V_{r,\delta})(t,x)\big\vert\,
\big\vert \nabla \Delta \psi_{r,\delta}(t,x)\Delta^{2} \psi_{r,\delta} (t,x)\big\vert \ \dd x\\
&\leq
\big\Vert \nabla V_{r,\delta}\big\Vert_{L^\infty_{t,x}}\big(\big\Vert \nabla \Delta \psi_{r,\delta}(t,\cdot)\big\Vert_{2}^{2}
+\big\Vert \Delta^{2}\psi_{r,\delta}(t,\cdot)\big\Vert^{2}_{2}\big),\\
I_{2}&=\int_{\mathbb{T}^{d}} \big| \Delta( \chi_{r}\star V_{r,\delta})(t,x)\big|\,
\big| (\Delta \psi_{r,\delta}\Delta^{2} \psi_{r,\delta})(t,x)\big| \, \dd x \\
&\leq
\left\Vert \Delta V_{r,\delta}\right\Vert_{L^\infty_{t,x}}\big(\left\Vert \Delta \psi_{r,\delta}(t,\cdot)\right\Vert_{2}^{2}+\| \Delta^{2}\psi_{r,\delta}(t,\cdot)\|^{2}_{2}\big).
\end{align*}
Using Lemma \ref{existence uniqueness for hat{U}}, we see that
\begin{align*}
\left\Vert V_{r,\delta}\right\Vert_{L^{\infty}_{t}C^{2,\alpha}_{x}}\leq C_{r,d}.
\end{align*}
Thus, summarizing the above estimate, we have
\begin{align*}
I_{0}+I_{1}+I_{2}\leq C_{r,d}\big(\| \Delta \psi_{r,\delta}(t,\cdot)\|_{2}^{2}+\| \Delta^{2}\psi_{r,\delta}(t,\cdot)\|_{2}^{2}\big).
\end{align*}
In order to bound $I_{3}$ and $I_{4}$, note that
\begin{align*}
-\Delta^{2}V_{r,\delta}=\zeta_{\delta}\star \Delta \chi_{r}\star\left\vert \psi_{r,\delta}\right\vert^{2}-\Delta(e^{V_{r,\delta}})=\zeta_{\delta}\star \Delta \chi_{r}\star\left\vert \psi_{r,\delta}\right\vert^{2}-e^{V_{r,\delta}}\left\vert \nabla V_{r,\delta}\right\vert^{2}-e^{V_{r,\delta}}\Delta V_{r,\delta} ,
\end{align*}
and therefore Lemma \ref{existence uniqueness for hat{U}} entails that
\begin{align*}
\left\Vert \Delta^{2}V_{r,\delta} \right\Vert_{L^{\infty}_{t,x}}\leq C_{r,d}.
\end{align*}
Thus, we obtain the bounds
\begin{align*}
I_{3}&\leq C_{r,d}\big(\| \nabla \psi_{r,\delta}\|_{2}^{2}+\| \Delta^{2}\psi_{r,\delta}\|_{2}^{2}\big),\\[1mm]
I_{4}&\leq C_{r,d}\big(1+\| \Delta^{2}\psi_{r,\delta}\|_{2}^{2}\big).
\end{align*}
To conclude, we have proved the estimate
\begin{align*}
\sum_{k=0}^{4}I_{k}\leq C_{r,d}\big(1+\| \Delta^{2}\psi_{r,\delta}(t,\cdot)\|_{2}^{2}\big),
\end{align*}
so that, in view of \eqref{sum of Ik}, we obtain
\begin{align*}
\frac{\dd}{\dd t}\|\Delta^{2} \psi_{r,\delta}(t,\cdot)\|_{2}^{2}
\leq C_{r,d}\big(1+\| \Delta^{2}\psi_{r,\delta}(t,\cdot)\|_{2}^{2}\big),
\end{align*}
which yields the estimate
\begin{align}\label{PROP OF H4}
\| \psi_{r,\delta}(t,\cdot)\|_{H^{4}}^{2}\leq C
\end{align}
for some $C=C(r,d,\| \psi^{\mathrm{in}}\|_{H^4})$.
Substituting \eqref{PROP OF H4} into \eqref{deltaH2 compactness} produces the inequality
\begin{align}
 &\left\Vert \psi_{r,\delta}(t,\cdot)-\psi_{r,\delta}(s,\cdot)\right\Vert_{H^{2}}\leq C| t-s|, \label{step2 final ine}
\end{align}
where $C=C(r,d,\| \psi^{\mathrm{in}}\|_{H^4})$.\\

\textbf{3}. \textit{Extraction of a solution}. By the Arzela-Ascoli theorem and \eqref{step2 final ine}, there exist a function $\psi_{r}\in C([0,T];H^{2}(\mathbb{T}^{d})) $ and a subsequence $\delta_{k}$ such that
\begin{align*}
\left\Vert \psi_{r,\delta_{k}}-\psi_{r}\right\Vert_{C_{t}H^{2}_{x}}\underset{k \rightarrow \infty}{\longrightarrow} 0.
\end{align*}
Moreover, inequality \eqref{step2 final ine} ensures that $\psi_{r}\in \mathrm{Lip}([0,T];H^{2}(\mathbb{T}^{d}))$.
We are left to verify that $\psi_{r}$ is the asserted solution. Denote by $V_{r}$ the solution of
\begin{align*}
-\Delta V_{r}=\chi_{r}\star\left\vert \psi_{r}\right\vert^{2}-e^{V_{r}}.
\end{align*}
Lemma \ref{=00007Bhat=00007DV vs =00007Bbar=00007DV} implies that
\begin{align*}
\| V_{r,\delta_{k}}(t,\cdot)-V_{r}(t,\cdot)\|_{2}
&\leq \|(\zeta_{\delta_{k}}\star\chi_{r}\star|\psi_{r,\delta_{k}}|^{2})(t,\cdot)
 -(\chi_{r}\star|\psi_{r}|^{2})(t,\cdot)\|_{2}\\
&\leq
\| (\zeta_{\delta_{k}}\star\chi_{r}\star| \psi_{r,\delta_{k}} |^{2})(t,\cdot)
  -(\zeta_{\delta_{k}}\star\chi_{r}\star| \psi_{r}|^{2})(t,\cdot) \|_{2}\\
&\quad\,\,+\| (\zeta_{\delta_{k}}\star\chi_{r}\star| \psi_{r}|^{2})(t,\cdot)
  -(\chi_{r}\star| \psi_{r}|^{2})(t,\cdot) \|_{2}.
\end{align*}
The first term is bounded by
\begin{align*}
\underset{t \in [0,T]}{\sup}\| \chi_{r}\star| \psi_{r,\delta_{k}}|^{2}-\chi_{r}\star| \psi_{r}|^{2}\|_{2}&\lesssim_{r}  \underset{t \in [0,T]}{\sup} \| | \psi_{r,\delta_{k}}|^{2}(t,\cdot)-| \psi_{r}|^{2}(t,\cdot)\|_{1}\\
&\lesssim_{r} \underset{t \in [0,T]}{\sup}\| \psi_{r,\delta_{k}}(t,\cdot)-\psi_{r}(t,\cdot)\|_{2}\underset{k \rightarrow \infty}{\longrightarrow}0.
\end{align*}
In addition, note that
\begin{align*}
(\zeta_{\delta_{k}}\star\chi_{r}\star\left\vert \psi_{r}\right\vert^{2})(t,x)\underset{k\rightarrow \infty}{\longrightarrow} (\chi_{r}\star\left\vert \psi_{r}\right\vert^{2})(t,x) 
\qquad \mbox{pointwise a.e. in}   \ (t,x),
\end{align*}
and, by Lebesgue's dominated convergence theorem, we have
\begin{align*}
\| (\zeta_{\delta_{k}}\star\chi_{r}\star| \psi_{r}|^{2})(t,\cdot)-(\chi_{r}\star| \psi_{r}|^{2})(t,\cdot) \|_{2}\underset{k \rightarrow \infty}{\longrightarrow} 0.
\end{align*}
Consequently, we obtain
\begin{align*}
\left\Vert V_{r,\delta_{k}}(t,\cdot)-V_{r}(t,\cdot)\right\Vert_{2} \underset{k\rightarrow \infty}{\longrightarrow} 0.
\end{align*}
As a result, we can pass to the limit as $k\rightarrow \infty$ in order to conclude that, for any
$\varphi \in C^{\infty}_{0}((0,T)\times \mathbb{T}^{d})$, 
\begin{align*}
&\int_{[0,T]\times \mathbb{T}^{d}}i\psi_{r}(t,x)\partial_{t}\varphi(t,x)\ \dd x \dd t\\
&=-\frac{1}{2}\int_{[0,T]\times \mathbb{T}^{d}} \Delta \psi_{r}(t,x)\varphi(t,x)\ \dd x \dd t
+\int_{[0,T]\times \mathbb{T}^{d}} (\chi_{r}\star V_{r})(t,x)\psi_{r}(t,x)\varphi(t,x)\ \dd x \dd t.
\end{align*}
Since $\psi_{r}\in \mathrm{Lip}([0,T];H^{2}(\mathbb{T}^{d}))$, 
we can integrate the term on the left-hand side by parts
in order to conclude that $\psi_{r}\in \mathrm{Lip}([0,T];H^{2}(\mathbb{T}^{d}))$
verifies the Cauchy problem \eqref{r equation}. Finally, by Lemma \ref{well posed linear eq} 
and uniqueness, we can prove that $\psi_{r}\in C^{1}([0,T];C^{\infty}(\mathbb{T}^{d})).$
\end{proof}

We continue by showing the compactness in $r$ of the family of solutions $\left\{\psi_{r} \right\} _{r>0}$ 
constructed above, which will enable us to extract a converging subsequence, 
thereby proving the existence of a solution to the original equation. 
As remarked in \cite{griffin2020singular}, the advantage of using a double regularization for $V_{r}$ 
is because this regularization procedures ensures the conservation of the following time-dependent quantity $\mathcal{F}_{r}$ defined in the next lemma. 
The proof is omitted, since it is similar to the proof of the conservation of total energy that 
has been given in \S\ref{The 1-Body Semi-Classical Limit}.
\begin{lem}\label{lem:5.8}
Let $\psi_{r}(t,x)\in\mathrm{Lip}([0,T];H^{2}(\mathbb{T}^{d}))$ be the solution to  \eqref{r mollified eq}. Let
\begin{align*}
\mathcal{F}_{r}(t)\coloneqq
\frac{1}{2}\int_{\mathbb{T}^{d}}\left\vert \nabla \psi_{r}(t,x)\right\vert^{2}\,\dd x+\frac{1}{2}\int_{\mathbb{T}^{d}}\left|\nabla V_{r}(t,x)\right|^{2}\,\dd x+\int_{\mathbb{T}^{d}}V_{r}(t,x)e^{V_{r}(t,x)}\,\dd x.
\end{align*}
Then
\begin{align}\label{conservation of energy sec 2}
 \frac{\dd}{\dd t}\mathcal{F}_{r}(t)=0.
\end{align}
\end{lem}

\medskip
We will also need the following lemma, stated with a scaling parameter $\varepsilon$, 
since it is also
useful in the asymptotic analysis later.
\begin{lem}[\cite{griffin2020singular}, Proposition 4.4]\label{stability with respect to W_2}
Let $d\in\{2,3\}$. For each $i=1,2$, let $h_{i}\in L^{\infty}\cap L^{\frac{d+2}{d}}(\mathbb{T}^{d})$, and 
let $(\tilde{V}_{i},\hat{V}_{i})$ be the solution of the system{\rm :}
\begin{align*}
\begin{cases}
- \Delta \tilde{V}_{i}=h_{i}-1,\\[0.5mm]
 -\Delta \hat{V}_{i}=1-e^{\tilde{V}_{i}+\hat{V}_{i}}.
\end{cases}
\end{align*}
Then
\begin{align*}
&\| \nabla \tilde{V}_{1}-\nabla \tilde{V}_{2}\|_{2}^2\leq \underset{i}{\max} \| h_{i} \|_{\infty} W_{2}^{2}(h_{1},h_{2}),
\\[0.5mm]
&\|\nabla \hat{V}_{1}-\nabla \hat{V}_{2}\|_{2}^2\leq C\,\underset{i}{\max} \| h_{i} \|_{\infty}W_{2}^{2}(h_{1},h_{2}),
\end{align*}
where $C=C_{d}(1+\underset{i}{\max} \| h_{i} \|_{\frac{d+2}{d}})$, and $W_{2}$ designates the $2$-Wasserstein distance.
\end{lem}

\medskip
\textit{Proof of Theorem \ref{wellposed intro}}.
We divide the proof into five steps.\\

\textbf{1}. \textit{Uniform bound in $r$ on $\underset{t\in[0,T]}{\sup}\left\Vert \psi_{r}(t,\cdot)\right\Vert_{H^{1}(\mathbb{T}^{d})}$}. Consider the mollified system:
\begin{align}\label{mollifiedequation}
\begin{cases}
i\partial_{t}\psi_{r}(t,x)=-\frac{1}{2}\Delta\psi_{r}(t,x)+(\chi_{r}\star V_{r})(t,x)\psi_{r}(t,x),\quad& x\in \mathbb{T}^d,\,t>0,\\[0.5mm]
-\Delta V_{r}(t,x)=(\chi_{r}\star\left|\psi_{r}\right|^{2})(t,x)-e^{V_{r}(t,x)},& x\in \mathbb{T}^d,\,t>0,\\[0.5mm]
 \psi_{r}|_{t=0}=\psi^{\mathrm{in}},& x\in \mathbb{T}^d.
\end{cases}
\end{align}
and denote by $\psi_{r}\in C^{1}([0,T];C^{\infty}(\mathbb{T}^{d}))$ the solution to this system ensured 
thanks to Lemma \ref{mollified eq well posed}.
Using Lemma \ref{lem:5.8},
we have
\begin{align}
&\frac{1}{2}\int_{\mathbb{T}^{d}}\left|\nabla\psi_{r}(t,x)\right|^{2}\,\dd x+\frac{1}{2}\int_{\mathbb{T}^{d}}\left|\nabla V_{r}(t,x)\right|^{2}\,\dd x+\int_{\mathbb{T}^{d}}V_{r}(t,x)e^{V_{r}(t,x)}\,\dd x \notag\\
&=\frac{1}{2}\int_{\mathbb{T}^{d}}\left|\nabla\psi^{\mathrm{in}}(x)\right|^{2}\,\dd x+\frac{1}{2}\int_{\mathbb{T}^{d}}\left|\nabla V_{r}(0,x)\right|^{2}\,\dd x+\int_{\mathbb{T}^{d}}V_{r}(0,x)e^{V_{r}(0,x)}\,\dd x \notag\\
&\leq\frac{1}{2}\left\Vert \nabla\psi^{\mathrm{in}}\right\Vert _{2}^{2}+\frac{1}{2}\left\Vert \nabla V_{r}(0,\cdot)\right\Vert _{2}^{2}+\left\Vert V_{r}(0,\cdot)\right\Vert _{\infty},\label{eq:-18}
\end{align}
where we have used that $\int_{\mathbb{T}^{d}}e^{V_{r}(0,x)}\,\dd x=1$ in the last inequality.
By the Sobolev embedding, we have
\begin{align*}
\| \chi_{r}\star | \psi^{\mathrm{in}}|^{2}\|_{\frac{d+2}{d}}\leq \| | \psi^{\mathrm{in}}|^{2} \|_{\frac{d+2}{d}}\leq C_{S}\| \psi^{\mathrm{in}}\|_{H^{1}(\mathbb{T}^{d})}^{2},   \end{align*}
where $C_{S}$ stands for the Sobolev constant.
Therefore, by Lemma \ref{existence uniqueness for hat{U}}, we see that
\begin{align*}
\| \nabla\hat{V}_{r}(0,\cdot)\| _{2}^{2}\leq C_{d}\Big(1+e^{4(1+C_{S}\| \psi^{\mathrm{in}}\|_{H^{1}(\mathbb{T}^{d})}^{2})}\Big).
\end{align*}
In addition, using the H\"older inequality, we obtain
\begin{align*}
\| \nabla\tilde{V}_{r}(0,\cdot)\|_{2}^{2}\leq \Big(1+C_{S}\| \psi^{\mathrm{in}}\|_{H^{1}(\mathbb{T}^{d})}^{2}\Big).
\end{align*}
Therefore, we have
\begin{align*}
\|\nabla V_{r}(0,\cdot)\|_{2}^{2}&\leq2\| \nabla\tilde{V}_{r}(0,\cdot)\|_{2}^{2}+2\| \nabla\hat{V}_{r}(0,\cdot)\|_{2}^{2}\\
&\leq2\big(1+C_{S}\|\psi^{\mathrm{in}}\|_{H^{1}(\mathbb{T}^{d})}^{2}\big)+2C_{d}\big(1+e^{4(1+C_{S}\| \psi^{\mathrm{in}}\|_{H^{1}(\mathbb{T}^{d})}^{2})}\big),
\end{align*}
and similarly
\[
\|V_{r}(0,\cdot)\|_{\infty}\leq2\big(1+C_{S}\|\psi^{\mathrm{in}}\|_{H^{1}(\mathbb{T}^{d})}^{2}\big)+2C_{d}\big(1+e^{4(1+C_{S}\| \psi^{\mathrm{in}}\|_{H^{1}(\mathbb{T}^{d})}^{2})}\big),
\]
which, in view of \eqref{eq:-18}, shows that
\begin{align}
\underset{t\in[0,T]}{\sup}\| \psi_{r}(t,\cdot)\|_{H^{1}(\mathbb{T}^{d})}\leq\mathcal{M}_{\mathrm{in}}, \label{step1 nabla Psi}
\end{align}
for some constant $\mathcal{M}_{\mathrm{in}}=\mathcal{M}_{\mathrm{in}}(\| \psi^{\mathrm{in}}\|_{H^{1}(\mathbb{T}^{d})},d)$. \\

\medskip
\textbf{2}. \textit{Uniform bound in $r$ on $\underset{t\in[0,T]}{\sup}\| \psi_{r}(t,\cdot)\|_{H^{2}(\mathbb{T}^{d})}$}. 
We compute
\begin{align*}
&\frac{\dd}{\dd t}\int_{\mathbb{T}^{d}}\frac{1}{2}\left|\Delta\psi_{r}(t,x)\right|^{2}\,\dd x\\
&= \Re\Big(\int_{\mathbb{T}^{d}}\Delta\partial_{t}\psi_{r}(t,x)\Delta\overline{\psi_{r}}(t,x)\,\dd x\Big)\\ 
&=\Re\Big(-i\int_{\mathbb{T}}\Delta\big(-\frac{1}{2}\Delta\psi_{r}(t,x)
+(\chi_{r}\star V_{r})(t,x)\psi_{r}(t,x)\big)\Delta\overline{\psi_{r}}(t,x)\,\dd x\Big)\\
&= \Re\Big(-i\int_{\mathbb{T}^{d}}\Delta\big((\chi_{r}\star V_{r})(t,x)\psi_{r}(t,x)\big)\Delta\overline{\psi_{r}}(t,x)\,\dd x\Big)\\ &= \Re\Big(-i\int_{\mathbb{T}^{d}}(\chi_{r}\star V_{r})(t,x)\left|\Delta\psi_{r}(t,x)\right|^{2}\,\dd x\Big)\\
&\quad+\Re\Big(-i\int_{\mathbb{T}^{d}}\Delta (\chi_{r}\star V_{r})(t,x)\psi_{r}(t,x)\Delta\overline{\psi_{r}}(t,x)\,\dd x\Big)\\
&\quad+2\Re\Big(-i\int_{\mathbb{T}^{d}}\nabla (\chi_{r}\star V_{r})(t,x)\nabla\psi_{r}(t,x)\Delta\overline{\psi_{r}}(t,x)\,\dd x\Big)\\
&= \Re\Big(-i\int_{\mathbb{T}^{d}}\Delta (\chi_{r}\star V_{r})(t,x)\psi_{r}(t,x)\Delta\overline{\psi_{r}}(t,x)\,\dd x\Big)\\
&\quad+2\Re\Big(-i\int_{\mathbb{T}^{d}}\nabla (\chi_{r}\star V_{r})(t,x)\nabla\psi_{r}(t,x)\Delta\overline{\psi_{r}}(t,x)\,\dd x\Big)\nonumber\\
&=: I+J.
\end{align*}
First, note that
\begin{align}
\left\vert J \right\vert &\leq 2\int_{\mathbb{T}^{d}}\left|\nabla (\chi_{r}\star V_{r})(t,x)\nabla\psi_{r}(t,x)\Delta\overline{\psi_{r}}(t,x)\right|\,\dd x \notag\\&\leq\int_{\mathbb{T}^{d}}\left|\nabla (\chi_{r}\star V_{r})(t,x)\nabla\psi_{r}(t,x)\right|^{2}\,\dd x+\int_{\mathbb{T}^{d}}\left|\Delta\psi_{r}(t,x)\right|^{2}\,\dd x \notag\\
 &\leq\|(\chi_{r}\star \nabla V_{r})(t,\cdot)\|_{4}^{2}\Big(\int_{\mathbb{T}^{d}}\left|\nabla\psi_{r}(t,x)\right|^{4}\,\dd x\Big)^{\frac{1}{2}}+\int_{\mathbb{T}^{d}}\left|\Delta\psi_{r}(t,x)\right|^{2}\,\dd x \notag\\
&\leq \left\Vert \nabla V_{r}(t,\cdot)\right\Vert _{4}^{2}\Big(\int_{\mathbb{T}^{d}}\left|\nabla\psi_{r}(t,x)\right|^{4}\,\dd x\Big)^{\frac{1}{2}}+\int_{\mathbb{T}^{d}}\left|\Delta\psi_{r}(t,x)\right|^{2}\,\dd x \notag\\
&\lesssim \big(\| \nabla\tilde{V}_{r}(t,\cdot)\|_{4}^{2}+\| \nabla\hat{V}_{r}(t,\cdot)\|_{4}^{2}\big)\Big(1+\int_{\mathbb{T}^{d}}\left|\Delta\psi_{r}(t,x)\right|^{2}\,\dd x\Big)
+\int_{\mathbb{T}^{d}}\left|\Delta\psi_{r}(t,x)\right|^{2}\,\dd x,\,\,\,
\label{JINE}
\end{align}
where we have used the Sobolev inequality, 
according to which $\left\Vert \nabla\psi\right\Vert_{4}^{2}\lesssim 1+\left\Vert \Delta \psi\right\Vert_{2}^{2} $, in order to derive the last inequality. To bound the first term in the above inequality, observe that
\begin{align}
 \| \nabla\tilde{V_{r}}(t,\cdot)\|_{4}^{2}\leq \left\Vert \psi_{r}(t,\cdot)\right\Vert_{4}^{2}+1\leq C_{S}(1+\left\Vert \nabla\psi_{r}(t,\cdot)\right\Vert_{2}^{2})\lesssim_{\mathcal{M}_{\mathrm{in}},d} 1,
 \label{1st}
\end{align}
where the last inequality is due to the Sobolev inequality and \eqref{step1 nabla Psi} in Step 1. In addition, Lemma \ref{existence uniqueness for hat{U}} entails
\begin{align}
\| \nabla\hat{V}_{r}(t,\cdot)\|_{4}\lesssim_{\mathcal{M}_{\mathrm{in}},d} 1. \label{2nd}   \end{align}
Hence, inequalities \eqref{JINE}--\eqref{2nd} imply that
\begin{align}\label{est on J}
J\lesssim_{\mathcal{M}_{\mathrm{in}},d} 1+\left\Vert \Delta \psi_{r}(t,\cdot)\right\Vert_{2} ^{2}.
\end{align}
As for $I$, we first observe that
\begin{align*}
\| \Delta V_{r}(t,\cdot)\|_{2}\leq 2\|(\chi_{r}\star | \psi_{r}|^{2})(t,\cdot)  \|_{2}+2\| e^{V_{r}(t,\cdot)}\|_{2}\leq 2\| \psi_{r}(t,\cdot)\|_{4}^{2}+2\| \psi_{r}(t,\cdot)\|_{4}^{2}=4\| \psi_{r}(t,\cdot)\|_{4}^{2},
\end{align*}
where the second inequality is due to the estimate:
$\Vert e^{V_{r}}\Vert_{2}\leq \Vert (\chi_{r}\star\vert \psi\vert^{2})(t,\cdot) \Vert_{2}$ 
proved in Lemma \ref{L^2 bound o e^V}. Therefore, it follows from the Sobolev embedding and \eqref{step1 nabla Psi} that
\begin{align}
\left\Vert \Delta V_{r}(t,\cdot)\right\Vert_{2}\lesssim_{\mathcal{M}_{\mathrm{in}},d} 1. \label{L2 norm of DElta}
\end{align}
We proceed by observing that
\begin{align}\label{est I}
I &\leq\int_{\mathbb{T}^{d}}\left\vert (\chi_{r}\star\Delta V_{r})(t,x)\psi_{r}(t,x)\Delta\overline{\psi_{r}}(t,x) \right\vert \,\dd x\nonumber\\&\leq
\int_{\mathbb{T}^{d}}\left\vert (\chi_{r}\star\Delta V_{r})(t,x)\psi_{r}(t,x)\right\vert^{2} \,\dd x+\int_{\mathbb{T}^{d}}\left\vert\Delta\overline{\psi_{r}}(t,x)\right\vert^{2}\,\dd x \notag\\
& \leq \left\Vert \psi_{r}(t,\cdot)\right\Vert_{\infty}^{2}
\left\Vert (\chi_{r}\star\Delta V_{r})(t,\cdot)\right\Vert_{2}^{2}+\left\Vert \Delta \psi_{r}(t,\cdot)\right\Vert_{2}^{2} \notag\\
 &\leq \left\Vert \psi_{r}(t,\cdot)\right\Vert_{\infty}^{2}\left\Vert \Delta V_{r}(t,\cdot)\right\Vert_{2}^{2}+\left\Vert \Delta \psi_{r}(t,\cdot)\right\Vert_{2}^{2} \notag\\
&\lesssim_{\mathcal{M}_{\mathrm{in}},d} ( 1+\left\Vert \Delta \psi_{r}(t,\cdot)\right\Vert_{2}^{2}),
\end{align}
where we have used \eqref{L2 norm of DElta} and the interpolation inequality:
\begin{align}
\left\Vert \Psi \right\Vert_{\infty}\leq C_{d}\big(\Vert \Psi \Vert_{2}+\Vert \Delta \Psi \Vert_{2}\big) \label{interpo ine}
\end{align}
for the last inequality in \eqref{est I}.
Thus, gathering \eqref{est I} and \eqref{est on J}, we have proved that
\begin{align*}
\frac{\dd}{\dd t} \left\Vert \psi_{r}(t,\cdot)\right\Vert_{H^{2}(\mathbb{T}^{d})}^{2}\leq C\left\Vert \psi_{r}(t,\cdot)\right\Vert_{H^{2}(\mathbb{T}^{d})}^{2}\
\end{align*}
for some $C=C(\Vert \psi^{\mathrm{in}}\Vert_{H^{1}},d)$, which, owing to the Gr\"onwall inequality, 
implies that
\begin{align}
\| \psi_{r}(t,\cdot)\|_{H^{2}(\mathbb{T}^{d})}\leq e^{Ct}\| \psi^{\mathrm{in}}\|_{H^{2}(\mathbb{T}^{d})}\leq \mathcal{M}_{\mathrm{in}}', \label{propgation of H^2}
\end{align}
where $C=C(\mathcal{M}_{\mathrm{in}},d)$ and $\mathcal{M}_{\mathrm{in}}'=\mathcal{M}_{\mathrm{in}}'(T,d,\Vert \psi^{\mathrm{in}}\Vert_{H^{2}(\mathbb{T}^d)})$.

\medskip
\textbf{3.} \textit{Uniform bound in $r$ on $\underset{t\in[0,T]}{\sup}\| \psi_{r}(t,\cdot)\|_{H^{4}(\mathbb{T}^{d})}$}. We follow a similar approach to the one taken in Step 2 in the proof of Theorem \ref{mollified eq well posed}, keeping track of the constants to ensure that they are uniform in $r$. 
From the same calculations leading to \eqref{sum of Ik}, we have\begin{align}
\frac{\dd}{\dd t
}\left\Vert \Delta^{2} \psi_{r}(t,\cdot)\right\Vert_{2}^{2}\lesssim \sum_{k=0}^{4} \int_{\mathbb{T}^{d}} \left\vert \nabla^{k}(\chi_{r}\star V_{r})\nabla^{4-k}\psi_{r}\Delta^{2}\psi_{r}\right\vert(t,x) \ \dd x \coloneqq \sum_{k=0}^{4}I_{k}. \label{sum of Iksecond}\end{align}
We proceed by estimating each one of $I_{k}$:
\begin{align*}
I_{0}&=\int_{\mathbb{T}^{d}} \big\vert (\chi_{r}\star V_{r})(t,x)\big\vert\,
\big\vert \Delta^{2} \psi_{r} (t,x) \big\vert^{2}\ \dd x\leq \left\Vert V_{r}\right\Vert_{L^\infty_{t,x}}\| \Delta^{2}\psi_{r}(t,\cdot)\|_{2}^{2},\\
I_{1}&=\int_{\mathbb{T}^{d}} \big| \nabla (\chi_{r}\star V_{r})(t,x)| \big|\,\big|\nabla \Delta \psi_{r}(t,x)\Delta^{2} \psi_{r} (t,x)\big| \ \dd x\\
&\leq
\| \nabla V_{r}\|_{L^\infty_{t,x}}\big(\| \nabla \Delta \psi_{r}(t,\cdot)\|_{2}^{2}+\| \Delta^{2}\psi_{r}(t,\cdot)\|^{2}_{2}\big),\\
I_{2}&=\int_{\mathbb{T}^{d}} \big| \Delta( \chi_{r}\star V_{r})(t,x)\big|\,\big| (\Delta \psi_{r}\Delta^{2} \psi_{r}) (t,x)\big| \, \dd x \\
&\leq
\| \Delta V_{r}\|_{L^\infty_{t,x}}\big(\| \Delta \psi_{r}(t,\cdot)\|_{2}^{2}+\| \Delta^{2}\psi_{r}(t,\cdot)\|^{2}_{2}\big).
\end{align*}
Using Lemma \ref{existence uniqueness for hat{U}}, \eqref{propgation of H^2}, and the Sobolev embedding $H^{2}(\mathbb{T}^{d})\hookrightarrow L^{\infty}(\mathbb{T}^{d})$, we see that
\begin{align}
\left\Vert V_{r}\right\Vert_{L^{\infty}_{t,x}}+\left\Vert \Delta V_{r}\right\Vert_{L^{\infty}_{t,x}}\leq \left\Vert V_{r}\right\Vert_{L^{\infty}_{t,x}}+\left\Vert \psi_{r}\right\Vert_{L^{\infty}_{t,x}}^{2}+\left\Vert e^{V_{r}}\right\Vert_{\infty}\lesssim_{T,d,\left\Vert \psi^{\mathrm{in}}\right\Vert_{H^{2}(\mathbb{T}^{d})}} 1, \label{H2infty norm}
\end{align}
so that
\begin{align}
I_{0}+I_{1}+I_{2}\lesssim_{T,d,\left\Vert \psi^{\mathrm{in}}\right\Vert_{H^{2}(\mathbb{T}^{d})}} \big(\| \Delta \psi_{r}(t,\cdot)\|_{2}^{2}+\| \Delta^{2}\psi_{r}(t,\cdot)\|_{2}^{2}\big). \label{IOI1I2}
\end{align}
In order to bound $I_{3}$ and $I_{4}$, note that
\begin{align*}
-\nabla \Delta V_{r}=\chi_{r}\star \nabla \left\vert \psi_{r}\right\vert^{2}-e^{V_{r}}\nabla V_{r}
\end{align*}
and
\begin{align*}
-\Delta^{2}V_{r}=  \chi_{r}\star \Delta\left\vert \psi_{r}\right\vert^{2}-\Delta(e^{V_{r}})=  \chi_{r}\star \Delta\left\vert \psi_{r}\right\vert^{2}-e^{V_{r}}\left\vert \nabla V_{r}\right\vert^{2}-e^{V_{r}}\Delta V_{r}.
\end{align*}
Therefore, we have
\begin{align*}
I_{4}&\leq \| \Delta^{2}\psi_{r}(t,\cdot)\|_{2}^{2}+\| (\chi_{r}\star \Delta^{2}V_{r})(t,\cdot)\|_{\infty}^{2}\leq \| \Delta^{2}\psi_{r}(t,\cdot)||_{2}^{2}+\| \Delta^{2}V_{r}(t,\cdot)\|_{\infty}^{2}\\
&\leq \| \Delta^{2}\psi_{r}(t,\cdot)||_{2}^{2}+\| (\chi_{r}\star \Delta| \psi_{r}|^{2})(t,\cdot)\|_{\infty}+\| (e^{V_{r}}| \nabla V_{r}|^{2})(t,\cdot)\|_{\infty}
+\| (e^{V_{r}}\Delta V_{r})(t,\cdot)\|_{\infty}\\
&\coloneqq I_{40}+I_{41}+I_{42}+I_{43}.
\end{align*}
First, using \eqref{interpo ine}, we obtain
\begin{align}
I_{41}&\leq \|\Delta| \psi_{r}|^{2}(t,\cdot)\|_{\infty}=2\| \psi_{r}(t,\cdot)\|_{\infty} \|\Delta \psi_{r}(t,\cdot)\|_{\infty}+2\| \nabla \psi_{r}(t,\cdot)\|_{\infty}^{2}\notag\\
&\leq
\| \psi_{r}(t,\cdot)\|_{\infty}^{2}+\| \Delta \psi_{r}(t,\cdot)\|_{\infty}^{2}+2\| \nabla \psi_{r}(t,\cdot)\|_{\infty}^{2}\notag\\
&\lesssim_{d} 1+\| \Delta{\psi_{r}}(t,\cdot)\|_{2}^{2}+\| \Delta^{2}{\psi_{r}}(t,\cdot)\|_{2}^{2}+\| \nabla \psi_{r}(t,\cdot)\|_{2}^{2}+\| \nabla \Delta \psi_{r}(t,\cdot)\|_{2}^{2} \notag\\
&\lesssim_{d} 1+\| \Delta^{2}{\psi_{r}}(t,\cdot)\|_{2}^{2}. \label{I41}
\end{align}
Secondly, note that, by Lemma \ref{existence uniqueness for hat{U}} (or alternatively by \eqref{H2infty norm}), 
we have
\begin{align*}
\Vert e^{V_{r}}\Vert_{\infty}\lesssim_{T,d,\Vert \psi^{\mathrm{in}}\Vert_{H^{2}(\mathbb{T}^{d})}} 1,
\qquad
\Vert \Delta V_{r}(t,\cdot)\Vert_{\infty}\lesssim_{T,d,\Vert \psi^{\mathrm{in}}\Vert_{H^{2}(\mathbb{T}^{d})}} 1.
\end{align*}
It follows that
\begin{align}
I_{42}+I_{43}\lesssim_{T,d,\left\Vert \psi^{\mathrm{in}}\right\Vert_{H^{2}(\mathbb{T}^{d})}} 1+\left\Vert \nabla V_{r}\right\Vert_{\infty}^{2}+\left\Vert \Delta V_{r}\right\Vert_{\infty}^{2}\lesssim_{T,d,\left\Vert \psi^{\mathrm{in}}\right\Vert_{H^{2}(\mathbb{T}^{d})}} 1. \label{I42I43}  \end{align}
Gathering \eqref{I41} with \eqref{I42I43} shows that
\begin{align*}
I_{4}\lesssim_{T,d,\left\Vert \psi^{\mathrm{in}}\right\Vert_{H^{2}(\mathbb{T}^{d})}} \big(1+\left\Vert \Delta^{2}\psi_{r}(t,\cdot)\right\Vert_{2}^{2} \big).
\end{align*}
It follows from the same considerations that
\begin{align*}
I_{3}\lesssim_{T,d,\left\Vert \psi^{\mathrm{in}}\right\Vert_{H^{2}(\mathbb{T}^{d})}}\big(1+\left\Vert \Delta^{2}\psi_{r}(t,\cdot)\right\Vert_{2}^{2} \big),
\end{align*}
which implies that
\begin{align}
I_{3}+I_{4}\lesssim_{T,d,\left\Vert \psi^{\mathrm{in}}\right\Vert_{H^{2}(\mathbb{T}^{d})}} \big(1+\left\Vert \Delta^{2}\psi_{r}(t,\cdot)\right\Vert_{2}^{2} \big). \label{I3I4}
\end{align}
Thus, the combination of \eqref{IOI1I2} with \eqref{I3I4} yields
\begin{align*}
\frac{\dd}{\dd t
}\left\Vert \Delta^{2} \psi_{r}(t,\cdot)\right\Vert_{2}^{2}\lesssim_{T,d,\left\Vert \psi^{\mathrm{in}}\right\Vert_{H^{2}(\mathbb{T}^{d})}} 1+\left\Vert \Delta^{2}\psi_{r}(t,\cdot)\right\Vert_{2}^{2}    \end{align*}
and therefore
\begin{align}
\| \psi_{r}(t,\cdot)\|_{H^{4}(\mathbb{T}^{d})}\leq e^{Ct}\| \psi^{\mathrm{in}}\|_{H^{4}(\mathbb{T}^{d})}\leq \mathcal{M}_{\mathrm{in}}'', \label{propgation of H^4}
\end{align}
where $C=C(T,d,\| \psi^{\mathrm{in}}\|_{H^{4}(\mathbb{T}^d)})$ and 
$\mathcal{M}_{\mathrm{in}}''=\mathcal{M}_{\mathrm{in}}''(T,d,\| \psi^{\mathrm{in}}\|_{H^{4}(\mathbb{T}^d)})$.

\medskip
\textbf{4.} \textit{Compactness in $r$}. We utilize the preceding estimates in order to extract a converging subsequence. We have
\begin{align}
\| \psi_{r}(t,\cdot)-\psi_{r}(s,\cdot)\|_{H^{2}}^{2}&\leq
2\left\Vert \int_{t}^{s} \partial_{\tau}\psi_{r}(\tau,\cdot)\,\dd \tau\right\Vert_{2}^{2}+2\left\Vert \int_{t}^{s} \partial_{\tau}\Delta\psi_{r}(\tau,\cdot)\,\dd \tau\right\Vert_{2}^{2} \notag\\
&\leq 2\left\vert t-s \right\vert^{2}\Big(\,\,\underset{t\in [0,T]}{\sup} \left\Vert \partial_{t}\psi_{r}(t,\cdot)\right\Vert_{2}^{2}+\underset{t\in [0,T]}{\sup}\left\Vert \partial_{t}\Delta \psi_{r}(t,\cdot) \right\Vert_{2}^{2}\Big). \label{AS}
\end{align}
Using \eqref{mollifiedequation} and  \eqref{propgation of H^2}, we have
\begin{align*}
 \left\Vert \partial_{t}\psi_{r}(t,\cdot)\right\Vert_{2}&\leq \left\Vert \Delta\psi_{r}(t,\cdot)\right\Vert_{2}+ \left\Vert (\chi_{r}\star V_{r})(t,\cdot)\psi_{r}(t,\cdot)\right\Vert_{2}\nonumber\\
& \leq
C(T,\Vert \psi^{\mathrm{in}}\Vert_{H^{2}},d)+ \left\Vert (\chi_{r}\star V_{r})(t,\cdot)\psi_{r}(t,\cdot)\right\Vert_{2}.
\end{align*}
In addition, recall that, by \eqref{step1 nabla Psi},
\begin{align*}
\left\Vert (\chi_{r}\star V_{r})(t,\cdot)\psi_{r}(t,\cdot)\right\Vert_{2}
\leq \left\Vert (\chi_{r}\star V_{r})(t,\cdot)\right\Vert_{\infty}
\leq \left\Vert V_{r}(t,\cdot)\right\Vert_{\infty}\leq  C(\Vert \psi^{\mathrm{in}}\Vert_{H^{1}},d).
\end{align*}
It follows that
\begin{align}
\underset{t \in [0,T]}{\sup}\left\Vert \partial_{t}\psi_{r}(t,\cdot)\right\Vert_{2}\lesssim_{T,d,\left\Vert \psi^{\mathrm{in}}\right\Vert_{H^2}}1. \label{H2 ARZELA}
\end{align}
Moreover, we have
\begin{align*}
&\| \partial_{t}\Delta \psi_{r}(t,\cdot)\|_{2}\\
&\leq \| \Delta^{2} \psi_{r}(t,\cdot)\|_{2}+\| \Delta(\chi_{r}\star V_{r}\psi_{r})(t,\cdot)\|_{2}\\
&\leq \| \Delta^{2} \psi_{r}(t,\cdot)\|_{2}+\| V_{r}(t,\cdot)\|_{\infty}\| \Delta \psi_{r}(t,\cdot)\|_{2}+\| \Delta V_{r}(t,\cdot)\|_{\infty}
+2\| \nabla V_{r}(t,\cdot)\|_{\infty}\| \nabla \psi_{r}(t,\cdot)\|_{2}.
\end{align*}
Thanks to \eqref{H2infty norm}, \eqref{propgation of H^4}, 
and the last inequality above, we see that
\begin{align}
\underset{t\in [0,T]}{\sup}\left\Vert \partial_{t}\Delta \psi_{r}(t,\cdot) \right\Vert_{2}^{2}\lesssim_{T,d,\left\Vert \psi^{\mathrm{in}}\right\Vert_{H^4}}1   \label{H4 ARZELA}
\end{align}
Substituting \eqref{H2 ARZELA} and \eqref{H4 ARZELA} into \eqref{AS} gives
\begin{align*}
\left\Vert \psi_{r}(t,\cdot)-\psi_{r}(s,\cdot)\right\Vert_{2}\leq C\left\vert t-s \right\vert,
\end{align*}
where $C=C(T,d,\Vert \psi^{\mathrm{in}}\Vert_{H^4})$.
Therefore, it follows from the Arzela-Ascoli theorem that there are both
a sequence $r_{k}\rightarrow 0$ as $k\rightarrow \infty$ and some $\psi\in C([0,T];H^{2}(\mathbb{T}^{d}))$
such that
\[
\left\Vert \psi_{r_{k}}(t,\cdot)-\psi(t,\cdot)\right\Vert _{C([0,T];H^{2}(\mathbb{T}^{d}))}\underset{k\rightarrow\infty}{\longrightarrow}0
\]
with the estimate:
\begin{align}
\underset{t\in [0,T]}{\sup}\left\Vert \psi(t,\cdot)\right\Vert_{H^{2}}\leq \mathcal{M}'_\mathrm{in}  \label{H2 est for final limit}
\end{align}
where $\mathcal{M}'_\mathrm{in}$ is as in \eqref{propgation of H^2}.

\medskip
\textbf{5}. \textit{$\psi(t,\cdot)$ is a solution}.
It remains to show that the limit $\psi(t,\cdot)$ is a solution to system \eqref{well posedness of hartree vpme}.
Note that $\psi(t,\cdot)\in L^{\infty}(\mathbb{T}^d)$, because of the embedding $H^{2}(\mathbb{T}^{d})\hookrightarrow L^\infty(\mathbb{T}^{d})$. Therefore, by Lemma \ref{existence uniqueness for hat{U}}, 
there exists a unique solution $V$  of the equation:
\begin{align*}
-\Delta V=\left\vert \psi \right\vert^{2}-e^{V}.
\end{align*}
To show that $\psi$ is a solution, the key component is to show that
\[
\left\Vert (\chi_{r_{k}}\star V_{r_{k}})(t,\cdot)\psi_{r_{k}}(t,\cdot)-V(t,\cdot)\psi(t,\cdot)\right\Vert _{2}\underset{k\rightarrow\infty}{\longrightarrow}0.
\]
By the triangle inequality, we have
\begin{align*}
&\left\Vert  (\chi_{r_{k}}\star V_{r_{k}})(t,\cdot)\psi_{r_{k}}(t,\cdot)-V(t,\cdot)\psi(t,\cdot)\right\Vert _{2}\\
&\leq \left\Vert (\chi_{r_{k}}\star V_{r_{k}})(t,\cdot)\psi_{r_{k}}(t,\cdot)-(\chi_{r_{k}}\star V)(t,\cdot)\psi(t,\cdot)\right\Vert_{2}\\
&\quad\,\,+
\left\Vert (\chi_{r_{k}}\star V)(t,\cdot)\psi(t,\cdot)- V(t,\cdot)\psi(t,\cdot)\right\Vert_{2}
\\
&\leq\left\Vert (\chi_{r_{k}}\star V)(t,\cdot)\left(\psi_{r_{k}}(t,\cdot)-\psi(t,\cdot)\right)\right\Vert _{2}\\
&\quad\,\,+\left\Vert \left(\chi_{r_{k}}\star V_{r_{k}}-\chi_{r_{k}}\star V\right)(t,\cdot)\psi_{r_{k}}(t,\cdot)\right\Vert _{2}\\
&\quad\,\,+
\left\Vert \psi(t,\cdot)\left((\chi_{r_{k}}\star V)(t,\cdot)- V(t,\cdot)\right)\right\Vert_{2}
\coloneqq \mathcal{I}+\mathcal{J}+\mathcal{L}.
\end{align*}
The term $\mathcal{I}$ is mastered as 
\[
\mathcal{I}\leq\| V(t,\cdot)\|_{\infty}\| \psi_{r_{k}}(t,\cdot)-\psi(t,\cdot)\|_{2}\leq\big(\| \tilde{V}(t,\cdot)\|_{\infty}+\| \hat{V}(t,\cdot)\|_{\infty}\big)\|\psi_{r_{k}}(t,\cdot)-\psi(t,\cdot)\|_{2}.
\]
By the Sobolev embedding and \eqref{H2 est for final limit},
we have
\begin{align*}
 \| \tilde{V}(t,\cdot)\| _{\infty}&\leq\| K\|_{2}\big(1+\left\Vert \psi(t,\cdot)\right\Vert _{4}^{2}\big)\nonumber\\
&\leq\left\Vert K\right\Vert _{2}\big(1+C_{S}\underset{t\in[0,T]}{\sup}\left\Vert \psi(t,\cdot)\right\Vert _{H^{1}(\mathbb{T}^{d})}^{2}\big)\lesssim_{T,d,\left\Vert \psi^{\mathrm{in}}\right\Vert_{H^2}} 1.
\end{align*}
In addition, Lemma \ref{existence uniqueness for hat{U}} and \eqref{H2 est for final limit} entail
\[
\underset{t\in[0,T]}{\sup}\| \hat{V}(t,\cdot)\|_{\infty}\lesssim_{T,d,\left\Vert \psi^{\mathrm{in}}\right\Vert_{H^2}} 1.
\]
Thus, we obtain the estimate:
\[
\mathcal{I}\lesssim_{T,d,\left\Vert \psi^{\mathrm{in}}\right\Vert_{H^2}} \underset{t\in[0,T]}{\sup}\left\Vert \psi_{r_{k}}(t,\cdot)-\psi(t,\cdot)\right\Vert _{2}\underset{k\rightarrow\infty}{\longrightarrow}0.
\]
To treat $\mathcal{J}$, we note that
\begin{align*}
 \mathcal{J}\leq \left\Vert (\chi_{r_{k}}\star V_{r_{k}}-\chi_{r_{k}}\star V)(t,\cdot)\psi_{r_{k}}(t,\cdot)\right\Vert_{2} \leq \left\Vert \psi_{r_{k}}(t,\cdot)\right\Vert_{\infty}\left\Vert (V_{r_{k}}-V)(t,\cdot) \right\Vert_{2}.
\end{align*}
Using Lemma \ref{stability with respect to W_2},  \eqref{propgation of H^2}, and \eqref{H2 est for final limit}, 
we see that
\begin{align*}
\left\Vert V_{r_{k}}(t,\cdot)-V(t,\cdot) \right\Vert_{2}\lesssim_{T,d,\left\Vert \psi^{\mathrm{in}}\right\Vert_{H^2}} W_{2}^{2}\big((\chi_{r_{k}}\star \left\vert \psi_{r_{k}}\right\vert^{2})(t,\cdot),\left\vert\psi\right\vert^{2}(t,\cdot)\big)\underset{k\rightarrow\infty}{\rightarrow}0.
\end{align*}
In addition, we have
\begin{align*}
 \left\Vert \psi_{r_{k}}(t,\cdot) \right\Vert_{\infty}\leq  \left\Vert \psi_{r_{k}}(t,\cdot) \right\Vert_{H^{2}(\mathbb{T}^{d})}\lesssim_{T,d,\left\Vert \psi^{\mathrm{in}}\right\Vert_{H^2}}1,
\end{align*}
which eventually yields that
\begin{align*}
\mathcal{J}\underset{k\rightarrow\infty}{\longrightarrow}0.
\end{align*}
Finally, again by \eqref{H2 est for final limit}, we obtain
\begin{align*}
\mathcal{L}&\leq \left\Vert \psi(t,\cdot)\right\Vert_{\infty} \left\Vert (\chi_{r_{k}}\star V)(t,\cdot)-V(t,\cdot)\right\Vert_{2}\nonumber\\
&\lesssim_{T,d,\left\Vert \psi^{\mathrm{in}}\right\Vert_{H^2}}\left\Vert (\chi_{r_{k}}\star V)(t,\cdot)-V(t,\cdot)\right\Vert_{2}\underset{k \rightarrow \infty}{\rightarrow}0. \end{align*}

This completes the proof of Theorem \ref{wellposed intro}.
\qed

\bigskip
\bigskip
\noindent{\bf Acknowledgments.}
We are grateful to Megan Griffin-Pickering and Enno Lenzmann for helpful discussions. IBP, GQC, and DF acknowledge support by EPSRC grant EP/V051121/1. GQC was also partially supported by EPSRC grants EP/L015811/1 and EP/V008854. DF was also partially supported by the Fundamental Research Funds for the Central Universities No. 2233100021 and No. 2233300008. For the purpose of open access, the authors have applied a CC BY public copyright license to any Author Accepted Manuscript (AAM) version
arising from this submission.

\bigskip
\medskip
\noindent{\bf Conflict of Interest:} The authors declare that they have no conflict of interest.
The authors also declare that this manuscript has not been previously published,
and will not be submitted elsewhere before your decision.

\bigskip
\noindent{\bf Data availability:} Data sharing is not applicable to this article as no datasets were generated or analyzed during the current study.

\end{document}